\documentclass[reqno,11pt]{amsart}
\usepackage{amssymb,amsmath,amsthm}
\usepackage{amscd}
\usepackage{amsfonts}
\usepackage{amssymb}
\usepackage{latexsym}
\usepackage{color}
\usepackage{esint}
\usepackage{graphicx, subfigure}
\usepackage{caption}
\usepackage{amsthm}
\usepackage{hyperref}
\usepackage{tikz}
\usepackage{todonotes} 
\usepackage{bigints}
\usetikzlibrary{shapes.geometric}
\DeclareMathOperator*{\esssup}{ess\,sup}

 \usepackage[makeroom]{cancel}
 \usetikzlibrary{decorations.pathreplacing}

\let\hide\iffalse

\newtheorem{theorem}{Theorem}

\newtheorem{corollary}[theorem]{Corollary}
\newtheorem{definition}[theorem]{Definition}
\newtheorem{lemma}[theorem]{Lemma}
\newtheorem{proposition}[theorem]{Proposition}
\newtheorem{remark}[theorem]{Remark}

\let\p=\partial

\let\O=\Omega

\let\b=\beta

\newcommand{\R}{\mathbb{R}}

\renewcommand{\S}{\mathbb{S}}

\newcommand{\be}{\begin{equation}}
\newcommand{\bm}{\begin{multline}}
\newcommand{\ee}{\end{equation}}

\newcommand{\xb}{x_{\mathbf{b}}}

\newcommand{\tb}{t_{\mathbf{b}}}

\newcommand{\X}{\mathbf{x}}
\newcommand{\V}{\mathbf{v}}
\newcommand{\T}{\Theta}
\newcommand{\half}{\frac{1}{2}}

\newcommand{\bv}{\bar{v}}
\newcommand{\bx}{\bar{x}}
\newcommand{\tv}{\tilde{v}}
\newcommand{\tx}{\tilde{x}}
\newcommand{\bO}{\overline{\Omega}}

\newcommand{\bX}{\overline{X}}
\newcommand{\bV}{\overline{V}}
\newcommand{\z}{\zeta}
\newcommand{\tX}{\widetilde{X}}

\newcommand{\Bes}{\begin{eqnarray*}}
\newcommand{\Ees}{\end{eqnarray*}}
\newcommand{\Be}{\begin{equation}}
\newcommand{\Ee}{\end{equation}}

 \numberwithin{equation}{section}
 \numberwithin{theorem}{section}

\def\p{\partial}

\def\O{\Omega}
\def\R{\mathbb{R}}

\def\B{\begin{equation}}
\def\E{\end{equation}}
\def\BN{\begin{eqnarray*}}
\def\EN{\end{eqnarray*}}

\usepackage{color}

\begin{document}
\title[Optimal Regularity]{Optimal $C^{\half}$ regularity of the Boltzmann equation in non-convex domains} 

\author[G. An]{Gayoung An}
\address{Department of Mathematics, Yonsei University, South Korea}
\email{gayoungan@yonsei.ac.kr}
\author[D. Lee]{Donghyun Lee*}
\address{Department of Mathematics, Pohang University of Science and Technology, South Korea }
\email{donglee@postech.ac.kr}

\thanks{*Corresponding author}




\begin{abstract}	
	Regularity of the Boltzmann equation, particularly in the presence of physical boundary conditions, heavily relies on the geometry of the boundaries. In the case of non-convex domains with specular reflection boundary conditions, the problem remained outstanding until recently due to the severe singularity of billiard trajectories near the grazing set, where the trajectory map is not differentiable. This challenge was addressed in \cite{CD2023}, where $C^{\frac{1}{2}-}_{x,v}$ H\"{o}lder regularity was proven. In this paper, we introduce a novel \textit{dynamical singular regime integration} methodology to establish the optimal $C^{\frac{1}{2}}_{x,v}$ regularity for the Boltzmann equation past a convex obstacle. 
\end{abstract}

\maketitle
\tableofcontents
 
\section{Introduction}

In this paper, we study the Boltzmann equation
\begin{align*}
    \partial_t F + v \cdot \nabla_x F = Q(F,F), \quad F(0,x,v)=F_0(x,v),\quad t \geq 0, \,\, x\in \O, \,\, v\in\R^{3},
\end{align*} 
where $F(t,x,v) \geq 0$ is probability density function which describes behavior of rarefied gas particles. Here, $t$ is time, $x$ is position variable in domain $\O$, and $v$ is velocity variable in $\R^{3}$. The Boltzmann equation describes evolution of the density function $F$ taking elastic binary collision $Q(F,F)$ into account. Explicitly, quadratic nonlinear colllisional operator $Q(F,F)$ is defined by 
\begin{align*} 
    Q(F,F)(t,x,v)&:= \int_{\mathbb{R}^{3}} \int_{\mathbb{S}^{2}_{+}} B(v-u,\omega)  \left(F(t,x,u')F(t,x,v')-F(t,x,u)F(t,x,v)\right)   d\omega du\\
    &:=Q_{gain}(F,F)(t,x,v)-Q_{loss}(F,F)(t,x,v),
\end{align*} 
where $B(v-u,\omega)$ is collision kernel, and post collisional velocities $u'$ and $v'$ are written as
\begin{align*} 
	u'=u+\left((v-u)\cdot\omega\right) \omega, \text{\quad and \quad} v'=v-\left((v-u)\cdot\omega\right) \omega,\quad \omega\in \S^{2}_{+}.
\end{align*} 

	For cutoff type Boltzmann equations, the collision kernel is commonly written as
\[
B(v-u,\omega)= |v-u|^{\gamma}b(\cos\theta),\quad -3 < \gamma \leq 1,
\]
where $\cos\theta := \frac{(v-u)\cdot \omega}{|v-u|}$ and $0 \leq b(\cos\theta) \leq C|\cos\theta|$.
In this paper, we only consider the hard sphere model
\[
B(v-u,\omega)=|(v-u)\cdot \omega|,
\]
which corresponds to $\gamma=1$. \\

When it comes to boundary, we are interested in exterior domain of uniform convex object.
\begin{definition}\label{def:domain}
	Let $\Omega = \mathbb{R}^3 \setminus \overline{\mathcal{O}}$ be the exterior domain of a uniformly convex, bounded open set $\mathcal{O} \subset \mathbb{R}^3$, where $\mathcal{O}= \{x \in \R^3: \xi(x)>0\}$ for a sufficiently smooth function $\xi(x)$. Additionally, there exists a constant $\theta_{\O} >0$ such that 
	\begin{align}\label{convex_xi}
\zeta \cdot(- \nabla^2 \xi(x))   \zeta	:=  \sum_{i,j=1}^3  (-\p_{ij} \xi(x)) \zeta_i\zeta_j \geq \theta_{\O} |\zeta|^2 \text{\quad for \quad} \zeta \in \R^3,
	\end{align}
	which means that $\mathcal{O}$ is a uniformly convex object. In other words, $\O = \{x \in \R^3: \xi(x)<0 \}$ with $\xi$ of \eqref{convex_xi}. We define the boundary of the domain
    \begin{align*}
        \partial\O  :=\overline{\mathcal{O}}  \backslash \mathcal{O}= \{x \in \R^3: \xi(x)=0 \}, 
        \quad \bO := \O \cup \partial\O,
    \end{align*} 
   a unit normal vector 
	\begin{align*}
	    n(x) :=  \frac{\nabla \xi (x)}{|\nabla \xi(x)|} \text{\quad for \quad}  x \in \partial \O,
	\end{align*}
    and the distance to the boundary
    \[
    	d(x,\partial\O) := \min_{y \in \partial \O}|x-y| \text{\quad for \quad}  x \in  \bO.
    \]
   
	\end{definition} 
    
    We impose specular reflection boundary condition on the boundary $\p\O$,
	\begin{equation*}
		F(t,x,v) = F(t,x,R_{x}v) \  \textit{ for } x\in\p\O, \  \textit{ where } \ R_x v := \left(I - 2n (x)\otimes n(x)\right)v,
	\end{equation*} 
	where $R_{x}$ is reflection operator on $x\in \p\O$.  \\

	If $F$ has sufficiently fast decay in large $|v|$, it is well-known that $Q(F,F)=0$ is obtained when $F$ is a local Maxwellian
	\[
		F(t,x,v) = \frac{\rho(t,x)}{\sqrt{2\pi T(t,x)}^3} e^{-\frac{|v-U(t,x)|^2}{2T(t,x)}}
	\]
	where macroscopic quantities $\rho$, $U$, and $T$ are defined by 
	\[
		\rho = \int_{v} F dv,\quad \rho U = \int_{v}vF dv,\quad 3\rho T = \int_{v}|v-U|^2 F dv.
	\]
	When the macroscopic quantities are independent to time and position, the local Maxwellian is a uniform Gaussian function of $v$ and called global Maxwellian. In this paper, we use $\mu$ to denote  $\mu(v) := e^{- \frac{ |v|^2}{2} }$. We rewrite the Boltzmann equation with the specular reflection boundary condition in terms of scaled density function $f(t,x,v) := F(t,x,v)/\sqrt{\mu}(v)$. Then $f$ solves
\begin{align}
& \p_{t}f + v\cdot \nabla_{x} f := \Gamma_{\text{gain}}(f,f) - \nu(f) f ,  
\ \ 
f|_{t=0} =f_0,  \label{f_eqtn} \\
& f(t,x,v)  = f(t,x,R_{x}v), \quad x\in \p\O, 
\label{specular}
\end{align}
where
\begin{align*}
&\Gamma_{\text{gain}}(f,f) = \frac{1}{\sqrt{\mu}(v)}Q_{\text{gain}}(\sqrt{\mu}f,\sqrt{\mu}f),\\
&\nu(f) = \int_{\R^3}\int_{\mathbb{S}^{2}_{+}} |(v-u)\cdot\omega| \sqrt{\mu}(u) f(t,x,u) d\omega du. 
\end{align*}

To define mild formulation, let us define characteristic. Backward exit time and position is defined by 
 	\begin{equation} \notag 
 	\begin{split}
 	\tb(x,v)   &:=  \sup \big\{  s \geq 0 :  x-\tau v  \in \O  \ \ \text{for all}  \ \tau \in ( 0, s )  \big\}, \quad t_{1}(t,x,v) := t-\tb(x,v),  \\
 	\xb(x,v)   &:=  x- \tb( x,v) v \in \partial\O.  
 	\end{split}
 	\end{equation}
 	If $x-\tau v \in \O$ for all $\tau\in \R$, i.e., backward in time trajectory from $(x,v)$ does not intersect with $\p\O$, we define $\tb(x,v)=\infty$ and $t_1(t,x,v)=-\infty$. \\
 	
 	In the case of $\O$, the domain is uniformly concave and billiard trajectory undergoes only one bounce on $\p\O$ at most. Therefore, we define billiard trajectory $(X(s;t,x,v), V(s;t,x,v))$ as 
 	\begin{align*}
 		\begin{split}
 			&(X(s;t,x,v), V(s;t,x,v)) \mathbf{1}_{\{\tb(x,v) < \infty\}} \\
 			&= \begin{cases} (x-(t-s) v, v)  
 				\ \ \text{for} \ \ s \in (t-\tb(x,v),t]
 				\\
 				(\xb(x,v) - (t-\tb(x,v) - s) R_{\xb(x,v)} v, R_{\xb(x,v)} v
 				)
 				\ \ \text{for} \ \ s \in [0,t-\tb(x,v) ].
 			\end{cases}
 		\end{split}
 	\end{align*}
 	and
 		\begin{align*}
 		\begin{split}
 			(X(s;t,x,v), V(s;t,x,v)) \mathbf{1}_{\{\tb(x,v) = \infty\}} = (x-(t-s) v, v),  
 		\end{split}
 	\end{align*}
 	with Hamiltonian structure
 	\begin{equation} \notag
 		\frac{d}{ds} (X(s;t,x,v), V(s;t,x,v)  )   =    (V(s;t,x,v), 0),  \ \  \ 
 		(X(s;t,x,v)  , V (s;t,x,v)  )|_{s=t}=(x,v).
 	\end{equation} 

Now we read the the scaled Boltzmann equation \eqref{f_eqtn} in terms of $(X(s;t,x,v), V(s;t,x,v))$ and take integration in $s\in (0,t)$ to obtain the mild formulation
\Be \label{f_expan}
\begin{split} 
	f(t,x,v)
	&=  
	e^{- \int^t_ 0 \nu(f) (\tau, X(\tau;t,x,v), V(\tau;t,x,v)) d \tau}
	f(0,X(0;t,x,v), V(0;t,x,v))\\
	& \ \ + \int^t_0
	e^{- \int^t_ s \nu(f) (\tau, X(\tau;t,x,v), V(\tau;t,x,v)) d \tau}
	\Gamma_{\text{gain}}(f,f)(s,X(s;t,x,v), V(s;t,x,v)) ds
\end{split}
\Ee 
for \eqref{f_eqtn} and \eqref{specular}. The local in time existence of the mild solution \eqref{f_expan} is known as following.

\begin{lemma}[Local existence, \cite{GuoKim_boundary}] \label{lem:loc}
	Assume that the initial data $f_{0}$ satisfies $\|w_{0}(v)f_{0}\|_{\infty} := \|e^{\vartheta_{0}|v|^{2}}f_{0}\|_{\infty} < \infty$ for $0 < \vartheta_{0} < \frac{1}{4}$ and the initial compatibility condition \eqref{specular}. Then (\ref{f_expan}) has a unique local-in-time solution $f(t,x,v)$ on the interval $[0,T^*]$ for some $T^*>0$. Moreover, the solution satisfies
	\begin{equation*}
		\sup_{0\leq s \leq T^{*}} \|w(v) f(s)\|_{\infty} :=
		\sup_{0\leq s \leq T^{*}} \|e^{\vartheta |v|^{2}}f(s)\|_{\infty} \lesssim \|e^{\vartheta_{0} |v|^{2}}f_{0}\|_{\infty},
	\end{equation*}
	for some $0 < \vartheta < \vartheta_{0}$.  \\
\end{lemma}

	We note that \cite{GuoKim_boundary} considered bounded domains where we are dealing with unbounded domains. However, the proof of above lemma is quite similar as \cite{GuoKim_boundary} since Lemma \ref{lem:loc} states local existence in weighted $L^\infty$ and characteristic under specular reflection boundary condition is well-defined for both cases. Indeed, we construct sequence 
	\[
		\partial_t f^{n+1} +v\cdot\nabla_x f^{n+1} + \nu(F^n)f^{n+1} = \Gamma_{\text{gain}}(f^n, f^n),\quad f^{n+1}(0,x,v) = f_0(x,v)
	\]
	which satisfies
	\[
		f^{n+1}(t,x,v) = f^n(t,x,R_x v)\quad\text{on}\quad x\in \p\O,\quad n(x)\cdot v < 0.
	\]
	We use Duhamel formula along the characteristic $(X(s;t,x,v), V(s;t,x,v))$, multiply the equation by $e^{(\vartheta_0 - t)|v|^2}$, and perform nonlinear $\Gamma_{\text{gain}}$ estimate to obtain uniform estimate
	\[
		\sup_{t\in[0,T]}\|e^{\vartheta|v|^2}f^{n+1}(t)\|_{L^\infty_{x,v}} \leq C\|e^{\vartheta_0|v|^2}f_0\|_{L^\infty_{x,v}},\quad \forall n
	\]
	for some $0 < \vartheta < \vartheta_0$. Here, $T <\vartheta_0$ can be chosen sufficiently small. Using above uniform estimate, deriving cauchy sequence argument and checking that the limit is a unique solution is easy. \\

\subsection{Historical background}
	The Boltzmann equation stands as one of the most fundamental equations in rarefied collisional gas theory, with extensive literature on its well-posedness and regularity. For the spatially homogeneous case, the well-posedness theory and convergence to the global Maxwellian have been thoroughly investigated in \cite{Ark1, Ark2, Ark83, Ark_stab, Weak_homo, MW,  CMP_Mouhot,  nonuniq} and references therein. \\
	
	However, the well-posedness theory for the spatially inhomogeneous Boltzmann equation is considerably more challenging than for the homogeneous case. Notably, DiPerna and Lions \cite{DL} proved the global existence of renormalized solutions, and we also refer to \cite{hamdache} for weak solutions of boundary value problems. Nevertheless, the uniqueness of renormalized solutions remains an open problem. A complete well-posedness theory in near-perturbation regimes was established after Guo \cite{Guo_VPB, Guo_Classic, Guo_VMB, Guo_VPL} within sufficiently smooth function spaces. Furthermore, Desvillettes and Villani \cite{DV} demonstrated that if a global solution $F$ is uniformly bounded in a high-order Sobolev space, it converges to the global Maxwellian at a rate of $t^{-\infty}$ for $\mathbb{T}^{3}_{x}$ and specular reflection boundary condition (BC) cases. For $\mathbb{T}^{3}_{x}$, this $t^{-\infty}$ decay rate was improved to exponential decay by \cite{GMM}. However, proving uniform higher-order regularity is a formidable task. For the non-cutoff Boltzmann equation, regularizing effects are known , and Silvestre and Imbert \cite{CyrilLouis} showed that a solution exhibits uniform $C^{\infty}$ regularity given proper macroscopic assumptions.  \\
	
	For the hard-sphere Boltzmann equation (or other cutoff-type Boltzmann equations), regularizing effects are absent , and studying regularity in the presence of physical boundary conditions in some smooth domains poses a significant challenge. In \cite{GuoKim_boundary}, Guo et al first proved $C^1$ and $W^{1,p}$ type regularity for solutions with specular BC and diffuse BC, respectively, away from the grazing set $\gamma_0$, assuming $\O$ is uniformly convex. The convex geometric property is crucial in obtaining these results, because trajectories generally undergo grazing singularity in non-convex domains. For general non-convex domains, propagation of discontinuity has been established (for inflow, bounce-back, and diffuse BCs) \cite{Kim_disconti}, and optimal BV solutions of the Boltzmann equation were studied in \cite{BV} for diffuse and inflow BCs. We also refer to \cite{ChenKim, KRM, Cao2019, ChenHsiaKawa, IKChen2022, WuWang24} for some related recent works.  \\
	
	Meanwhile, the regularity of the Boltzmann equation solution with specular reflection boundary conditions in a non-convex domain remained outstanding until recently. Recently, in \cite{CD2023}, Kim-Lee succeeded in obtaining $C^{\half-}$regularity by assuming the domain to be the exterior of a convex object and averaging the specular singularity near the grazing singular regime. In the exterior of a convex object, the billiard trajectory possesses optimal $C^{\half}$ regularity, and since the regularity of the Boltzmann equation solution for specular BC depends on the billiard trajectory, it is very natural to expect the equation's solution to exhibit optimal $C^{\half}$ regularity. However, in the singular regime integration used in \cite{CD2023}, obtaining $C^{\half}$ regularity was not possible because the geometric effects of the domain vanish when the position is sufficiently close to the boundary. This paper introduces a new methodology to obtain optimal regularity results. We note that for domains combining convex and non-convex effects, the regularity can worsen \cite{Mixed}.  \\	
	
	Despite the fact that the solution of the Boltzmann equation for the physical boundary condition problem does not possess sufficient regularity, research on existence, uniqueness, and convergence to equilibrium is actively being pursued based on low-regularity function methodologies. In \cite{Guo10}, Guo introduced an $L^2-L^\infty$	bootstrap argument to prove a unique global solution and asymptotic stability near the global Maxwellian in $L^\infty$ for various boundary conditions problems. In the case of specular reflection boundary condition, however, the problem is very tricky because it is closely related to the geometric property of domain $\O$ and chaotic billiard trajectory in $\O$ should be studied. In \cite{CD2018}, Kim and Lee proved the specular reflection boundary problem in general $C^3$ convex domain. Specular reflection boundary condition problem in general non-convex smooth domain is outstanding open problem because of non-convexity of the domain cause grazing singularity of trajectory map. We also refer to \cite{TJM25} for the problem in 3D toroidal domains where rotational symmetry was crucially used to control grazing singularity. For various Boltzmann boundary problems near equilibrium, we also refer to \cite{Gradient, Non-iso}.  \\
	
	Furthermore, research on Boltzmann equation boundary problems in low regularity spaces has also been conducted for complex models with additional self-consistent forces like the Vlasov-Poisson-Boltzmann equation \cite{Cao-Kim-Lee, CKL2020}, and in another direction, this research is diversely developing in combination with studies on the Boltzmann solution with large amplitudes and small $L^p$. We refer to \cite{DHWZ, DKL2020, Adv_Duan, Ext_diffuse}.  \\
	
\subsection{Main Theorem}
	Before we state main theorem of this paper, introduce some basic notations. The $L^{\infty}$ norm is defined as 
    \begin{align*}
    \|f(t)\|_{\infty}:= \esssup_{(x,v)\in \bO \times \mathbb{R}^3} |f(t,x,v)|.
    \end{align*} 
    The indicator function of a subset $S$ within a set $X$ is a function $\mathbf{1}_{S} : X \rightarrow \{0,1\}$, defined as
\begin{align*} 
    \mathbf{1}_{S}(x) :=
    \begin{cases}
        1,\quad x \in S, \\
        0,\quad x \notin S
    \end{cases} 
    \text{\quad for \quad} x \in X.
\end{align*}
We denote by $A\lesssim B$ that there exists a uniform constant $C$ such that $A\le CB$. If the constant $C$ depends on $K$, we write $A \lesssim_{K} B$ or $A \leq C(K)B$. The notations used are as follows:
\begin{align*}
    |v|:=\sqrt{v_1^2+v_2^2+v_3^2}
    \text{\quad for \quad} v=(v_1,v_2,v_3) \in \mathbb{R}^3,
\end{align*}
\begin{align*}
    \langle v \rangle := \sqrt{1+|v|^2}, \quad
    \hat{v} := \begin{cases}
		v/|v|,\text{\; for \;} v\neq 0, \\ 0,\text{\; for \;} v=0,
		\end{cases}
\end{align*}
\begin{align*}
    \fint_{a}^{b} := \frac{1}{b-a} \int_{a}^{b}
     \text{\; for \;} a,b \in \mathbb{R}, \text{\quad and \quad}
     \mathcal{P}_{n}(x):=\sum_{i=0}^{n}x^{i}
     \text{\; for \;} x \in \mathbb{R}, \,\,n \in \mathbb{N}.
\end{align*}

Now, we state our main theorem, which establishes the optimal $C^{\frac{1}{2}}_{x,v}$ regularity for the Boltzmann equation past a convex obstacle.

 \begin{definition} \label{def_initial}
 	We define
 	\begin{align*} 
 		\mathbf{A}_{\b}(f_0) :=\sup_{ \substack{ (x,\bx,v) \in  \bO \times \bO \times \mathbb{R}^3 \\  |x-\bx| \leq 1   } }\langle v \rangle \frac{|f_0(x,v)-f_0(\bx,v)|}{|x-\bx|^{2\b}} +\sup_{ \substack{ (x,v,\bv) \in  \bO \times \mathbb{R}^3 \times \R^3\\  |v-\bv| \leq 1   } }\langle v \rangle^2 \frac{|f_0(x,v)-f_0(x,\bv)|}{|v-\bv|^{2\b}}
 	\end{align*} 
 	to denote weighted $C^{0,2\b}$ H\"older regularity of initial datum $f_0$.
 \end{definition}
 
\begin{theorem} \label{thm:Holder0.5}
	Suppose that the domain is given as in Definition \ref{def:domain}. Assume $f_0$ satisfies the compatibility condition \eqref{specular}, $\|e^{\vartheta_{0}|v|^{2}} f_0 \|_\infty< \infty$ for $0< \vartheta_{0} < \frac{1}{4}$, and 
    \begin{align*}
        \mathbf{A}_{\half}(f_0) := \sup_{ \substack{ (x,\bx,v) \in  \bO \times \bO \times \mathbb{R}^3 \\  |x-\bx| \leq 1   } }\langle v \rangle \frac{|f_0(x,v)-f_0(\bx,v)|}{|x-\bx|} +\sup_{ \substack{ (x,v,\bv) \in  \bO \times \mathbb{R}^3 \times \mathbb{R}^3 \\  |v-\bv| \leq 1   } }\langle v \rangle^2 \frac{|f_0(x,v)-f_0(x,\bv)|}{|v-\bv|} < \infty.
    \end{align*}
    For $0<\delta <1$ and $0< \epsilon<1$, we  define the weight function W as
        \begin{align*}
            W\left((x,v), (\bx,\bv);\epsilon,\delta\right) &:= 
            \frac{1}{|v|^{\delta}}\mathbf{1}_{\{d(x,\partial\O) \leq \epsilon\}}+ \frac{1}{|\bv|^{\delta}}\mathbf{1}_{\{d(\bx,\partial\O) \leq \epsilon\}}+1.
        \end{align*}
       Then, there exist $T > 0$ and a constant $C>0$, depending on $\delta, \epsilon$, and $\vartheta_{0}$, such that $f(t,x,v)$ satisfies the H\"older continuity condition as follows:
		
        \begin{align*}
        \begin{split}
		&\sup_{0\leq t \leq T} 
		\sup_{ \substack{ (x,\bx,v,\bv) \in  \bO \times \bO \times \mathbb{R}^3 \times \R^3 \\ 0<|(x,v)-(\bx, \bv)|\leq 1   } }
         \left(W^{-1}  \langle v \rangle^{-1} 
		 e^{-\varpi\langle v \rangle^{2}t}\frac{ | f(t,x,v) - f(t, \bx, \bv) | }{ |(x,v) - (\bx, \bv)|^{\frac{1}{2}} }\right) 
         \leq C\mathcal{P}_3(\|w_0f_0\|_{\infty}) \left(\mathbf{A}_{\frac{1}{2}}(f_0)+1\right).
        \end{split} 
		\end{align*} 
        Here, $\varpi$ is a large constant depending on $\vartheta_0, f_0, \epsilon$.
\end{theorem}

	Before we give a remark for the main theorem, let us recall previous results for $C^{\half-}$ regularity to compare with Theorem \ref{thm:Holder0.5}.  \\
	\textbf{Main Theorem of \cite{CD2023}} (previous result, simplified version) If $f_0$ satisfies 
	\[
	\mathbf{A}_{\b}(f_0) < \infty,\quad \b < \half,
	\]
	then for a local interval $[0,T_1]$, we have 
	\begin{align*}
		\begin{split}
			&\sup_{0\leq t \leq T_1} 
			\sup_{ \substack{ (x,\bx,v,\bv) \in  \bO \times \bO \times \mathbb{R}^3 \times \R^3 \\ 0<|(x,v)-(\bx, \bv)|\leq 1   } }
			\left( \langle v \rangle^{-2\beta} 
			e^{-\varpi_1\langle v \rangle^{2}t}\frac{ | f(t,x,v) - f(t, \bx, \bv) | }{ |(x,v) - (\bx, \bv)|^{\beta} }\right) 
			\leq C(\mathbf{A}_{\b}(f_0), f_0, \b).  \\
		\end{split} 
	\end{align*}

 \begin{remark}

According to Theorem \ref{thm:Holder0.5}, the solution has $C^{\frac{1}{2}}_{x,v}$ regularity away from the set 
\[
\gamma_* := \left\{ (x,v) \in \bO \times \mathbb{R}^3 : x \in \partial \Omega \ \text{and} \ |v| = 0 \right\}.
\]
For weight $W$, let us consider non-singular weight case: both $d(x,\p\O), d(\bar{x},\p\O) \geq \epsilon$. In this case, $W=1$ and our result is exactly same as the Theorem in \cite{CD2023}. For the most singular case, we have $d(x,\p\O)+d(\bar{x},\p\O) \leq \epsilon$, i.e., both $x$ and $\bar{x}$ are very near $\p\O$, and $|v|, |\bar{v}| \ll 1$. In such case,
\[
W\left((x,v), (\bar{x},\bar{v}); \epsilon, \delta\right) \leq \frac{1}{|v|^{\delta}} + \frac{1}{|\bar{v}|^{\delta}} + 1 <  \frac{1}{|v|} + \frac{1}{|\bar{v}|} + 1.
\]
Thus, the weight in our estimate does not impose a significant penalty even when $x$ (or $\bar{x}$) is near the boundary. 
\end{remark}

\begin{remark}
	 Physical meaning of the singular weight $W=W((x,v), (\bar{x}, \bar{v});\epsilon,\delta)$ is worth mentioning. Compared to previous result in \cite{CD2023}, Theorem \ref{thm:Holder0.5} includes the singular weight $W$ in the norm of LHS. Weight $W$ implies that singular weighted norm on LHS could be singular near $v\sim 0$(resp, $\bar{v}\sim 0$), if $x$(resp, $\bar{x}$) is very close to the boundary $\p\O$. This is a geometrically very natural phenomenon, because when $v$ is very small, trajectory $X(s)$ moves very slowly and singular regime integration behaves like
	\[
		\sim \int_{|v|\leq N} \frac{1}{|v\cdot \nabla\xi(\xb(x, v))|^{2\b}} \mathbf{1}_{\{\tb < \infty\}} dv. 
	\]
	Note that this is infinite for $\b = \half$ if we choose $x\in\p\O$, because once the position $x\in\p\O$ is fixed, the characteristic $(X(s), V(s))$ cannot be influenced by the geometric effects of concavity. For this reason, previous result in \cite{CD2023} was limited to a H\"older exponent $\b < \half$.  \\
	\end{remark}
	
	\begin{remark}
	If both $x$ and $\bar{x}$ are uniform away from the boundary $\p\O$ by $\epsilon$ (see definition of $W$), Theorem \ref{thm:Holder0.5} completely cover the theorem of \cite{CD2023} including $\b=\half$. What is even more surprising is that novel \textit{dynamical singluar regime integration} scheme enable us to obtain $C^{\half}_{x,v}$ regularity, \textbf{even if} $x$ and $\bar{x}$ are close to $\p\O$, unless velocities $|v|$ and $|\bar{v}|$ are small.  
	\end{remark}
	
	\begin{remark}
		When backward in time trajectory $(X(s;t,x,v), V(s;t,x,v))$ undergoes grazing collision on the boundary, it enjoys only $C^{\half}_{x,v}$ regularity for $s < t-\tb(t,x)$. Therefore, $C^{\b}_{x,v}$ regularity of $f$ with $\b > \half$ is impossible in general. For example, see \eqref{gain_sk_X_1/2}. In this sense, $C^{\half}_{x,v}$ regularity is optimal. \\
	\end{remark}

\subsection{Shift method and Specular Singularity}
Successful measurement for geometric effect of uniform concavity depends on choosing proper coordinate for billiard trajectory. We introduce shifted position and velocity to decompose fraction of trajectories into singular part and non-singular part.  

\begin{definition} \label{def_tilde}
	(1)(Position shift) For fixed $x, \bx \in \O, \,v \in \mathbb{R}^3$, let us  assume
	\begin{equation} \label{assume_x}
		 \text{ $x-\bx \neq 0$ is neither parallel nor anti-parallel to $v\neq 0$,  i.e., $(x-\bx)\cdot v \neq \pm |x-\bx||v|$. } 
	\end{equation}
	In this case, we define shifted position $\tilde{x}$ as
	\begin{equation} \label{def_tildex} 
		\tilde{x} =  \tilde{x}(x, \bx, v) := \bar{x} + \big( (x-\bar{x})\cdot \hat{v} \big) \hat{v}
	\end{equation} and 
		\begin{equation} \notag
			S_{(x, \bx, v)} := x + \text{span}\{ x- \tilde{x}, v \} = x + \text{span}\{ x- \bx, v \}.
		\end{equation}
		We define the parametrization $\X(\tau)$ as 
		\begin{equation} \notag
			\X(\tau)  = 
			\X(\tau; x, \bar x, v)
			:= (1-\tau)\tilde{x} (x, \bar x, v) + \tau x , \text{\quad so that \quad} \X(0)=\tx, \;\;\X(1)=x.
		\end{equation}
	We further assume that $\p\O\cap S_{(x, \bx , v)}$ is closed curve, and $\X(\tau) \in \O$ is well-defined for $0 \leq \tau \leq 1$. By convexity \eqref{convex_xi} of $\O$, we can define $\tau_{\pm}(x, \bx, v)$ as 
		\begin{equation} \notag
			v\cdot \nabla\xi(\xb(\X(\tau_{\pm}), v)) = 0,\quad \tau_{-} < \tau_{+},\quad \tau_{\pm} = \tau_{\pm}(x, \bx, v),
		\end{equation}
		where $\xb(\X(\tau),v)$ is well defined for $\tau_{-} \leq \tau \leq \tau_{+}$. See Figure \ref{fig1}.
        
         \vspace{3mm}
        \noindent (2)(Velocity shift) For fixed $v, \bv, \zeta \in \mathbb{R}^3$, let us  assume
	\begin{equation} \label{assume_v}
		\text{ $v+\zeta\neq 0$ is neither parallel nor anti-parallel to $\bv+\zeta\neq 0$,  i.e., $(v+\zeta)\cdot(\bv+\zeta)\neq \pm|v+\zeta||\bv+\zeta|$ }.
	\end{equation}
	In this case, we define shifted velocity $\tilde{v}$ as
	\begin{equation} \label{def_tildev} 
	\tilde{v} + \zeta = \tilde{v}(v, \bv, \zeta) + \zeta := |v+\zeta| \widehat{(\bar{v} +\zeta)}
	\end{equation}
    and
    \begin{equation*} 
			S_{(x, v, \bv, \zeta)} := x + \text{span}\{ v+\zeta, \tilde{v}+\zeta \} = x + \text{span}\{ v+\zeta, \bv+\zeta \}.
		\end{equation*}
    	We define the parametrization $\V(\tau)$ as 
        \begin{align*}
            \V(\tau) =
			\V(\tau; x, v , \bar v, \zeta) 
				:=
			 |v+\zeta| R_{(v,\bv,\zeta)}  \begin{bmatrix}
				\cos\T(\tau) \\ \sin\T(\tau) \\ 0
			\end{bmatrix}, 			
            \quad            \T(\tau) := \tau\theta, \quad \dot{\T}(\tau) = \dot{\T} := \theta,
        \end{align*}
        so that
        \begin{align*}
            \T(0) =  0 , \quad 
			\T(1) = \theta 
            \text{\quad and \quad}
			\V(0) = \tilde{v}+\zeta , \ \ \V(1)=v+\zeta,
        \end{align*}
        where $\theta = 
	   \cos^{-1} (\widehat{v+\zeta} \cdot \widehat{\bv+\zeta})
	   \in[0,2\pi)$ is the angle between $v+\zeta$ and $\bv+\zeta$, and 
		\[
			R_{(v,\bv,\zeta)} := 
			\begin{bmatrix}
				 & & \\
				 \widehat{\bv+\zeta} & \widehat{v+\zeta} & \widehat{ (\bv+\zeta)\times (v+\zeta) } \\
				 & & \\
			\end{bmatrix}
			\begin{bmatrix}
				1 & \cos\theta & 0 \\
				0 & \sin\theta & 0 \\
				0 & 0 & 1
			\end{bmatrix}^{-1}. 	\]
        Note that $|\V(\tau)|=|v+\z|$ for all $\tau$. We further assume that $\p\O\cap S_{(x, v, \bv, \z)}$ is closed curve, and $\V(\tau) \in \O$ is well-defined for $0 \leq \tau \leq 1$. By convexity \eqref{convex_xi} of $\O$, we can define $\tau_{\pm}(x, v, \bv,\z)$ as 
        \begin{align*}
            \V(\tau_{\pm})\cdot \nabla\xi(\xb(x, \V(\tau_{\pm}))) = 0,\quad \tau_{-} < \tau_{+},\quad \tau_{\pm} = \tau_{\pm}(x, v, \bv, \zeta),
        \end{align*}
        where $\xb(x, \V(\tau))$ is well defined for $\tau_{-} \leq \tau \leq \tau_{+}$. See Figure \ref{fig2}.
	\end{definition}

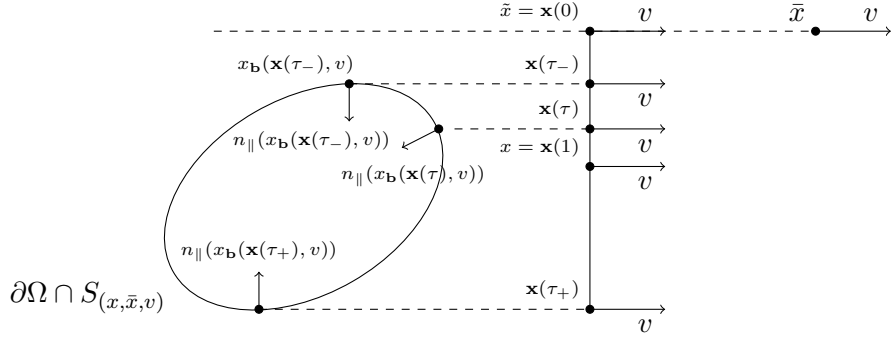
\begin{figure}
	\begin{center}
		\begin{tikzpicture}
			\draw[rotate=30] (-0.7,0.4) ellipse (2cm and 1.3cm);
			\draw (-2.5,-1) node [below left] {$\p\O\cap S_{(x, \bx ,v)}$};
			\filldraw (-0.2,1.5) circle[radius=1.5pt];
			\draw (0,1.5) node [above left] {\tiny{$\xb(\X(\tau_{-}),v)$}};
			\draw[arrows=->] (-0.2,1.5)-- (-0.2,1);
			\draw (-0.7,1) node[below]{\tiny{$n_{\parallel}(\xb(\X(\tau_{-}), v))$}} ;
			\filldraw (-1.4,-1.5) circle[radius=1.5pt];
			\draw[arrows=->] (-1.4,-1.5)-- (-1.4,-1);
			\draw (-1.4,-1) node[above]{\tiny{$n_{\parallel}(\xb(\X(\tau_{+}), v))$}} ;
			\filldraw (0.99, 0.9) circle[radius=1.5pt];
			\draw[arrows=->] (0.99,0.9)-- (0.5,0.65);
			\draw (0.65,0.55) node[below]{\tiny{$n_{\parallel}(\xb(\X(\tau), v))$}} ;
			\draw [dashed](-2,2.2) -- (6,2.2) ;
			\draw [dashed](-0.2,1.5) -- (3,1.5) ;
			\draw [dashed](-1.4,-1.5) -- (3,-1.5) ;
			\draw [dashed](1.2,0.9) -- (3,0.9) ;
			\draw (3,-1.5) -- (3,2.2);
			\draw[arrows=->] (3,0.4)-- node[below right]{$v$} (4,0.4);
			\filldraw (3,0.4) circle[radius=1.5pt];
			\draw (3,0.4) node [above left] {\tiny{$x=\X(1)$}};
			\draw[arrows=->] (3,0.9)-- node[below right]{$v$} (4,0.9);
			\filldraw (3,0.9) circle[radius=1.5pt];
			\draw (3,0.9) node [above left] {\tiny{$\X(\tau)$}};
			\draw[arrows=->] (3,1.5)-- node[below right]{$v$} (4,1.5);
			\filldraw (3,1.5) circle[radius=1.5pt];
			\draw (3,1.5) node [above left] {\tiny{$\X(\tau_{-})$}};
			\draw[arrows=->] (3,-1.5)-- node[below right]{$v$} (4,-1.5);
			\filldraw (3,-1.5) circle[radius=1.5pt];
			\draw (3,-1.5) node [above left] {\tiny{$\X(\tau_{+})$}};
			\draw[arrows=->] (3,2.2)-- node[above right]{$v$} (4,2.2);
			\filldraw (3,2.2) circle[radius=1.5pt];
			\draw (3,2.2) node [above left] {\tiny{$\tilde{x}=\X(0)$}};
			\draw[arrows=->] (6,2.2)-- node[above right]{$v$} (7,2.2);
			\filldraw (6,2.2) circle[radius=1.5pt];
			\draw (6,2.2) node [above left] {$\bar{x}$};
		\end{tikzpicture}
	\end{center}
	\caption{\cite{CD2023} $\tilde{x}$, $\X(\tau)$, and trajectories on projected plane $S_{(x, \bx ;v)} := x+ \text{span}\{ x-\tilde{x}, v\}$} \label{fig1}
\end{figure}

\begin{figure}
	\begin{center}
		\begin{tikzpicture}
			\draw[rotate=30] (-0.6,0.35) ellipse (2cm and 1.3cm);
			\draw[dashed] (3,1.5) circle (1);
			\draw (-2.2,-1) node [below left] {$\p\O\cap S_{(x, v, \bv, \zeta)}$};
			\filldraw (0,1.5) circle[radius=1.5pt];
			\draw (1,1.5) node [above left] {\tiny{$\xb(x, \V(\tau_{-}))$}};
			\draw[arrows=->] (0,1.5)-- (0,1);
			\draw (-.5,1) node[below]{\tiny{$\nabla\xi(\xb(x, \V(\tau_{-})))$}} ;
			\filldraw (1.12,0.84) circle[radius=1.5pt];
			\draw[arrows=->] (1.12,0.84)-- (0.65,0.65);
			\draw (0.75,0.55) node[below]{\tiny{$\nabla\xi(\xb(x, \V(\tau)))$}} ;
			\draw [dashed] (0,1.5) -- (3,1.5) ;
			\draw [dashed] (1.2,0.9)-- (3,1.5);
			\draw (5.3, 1) node{$\bar{v}+\zeta$};
			\draw[arrows=->] (3,1.5)--  (5.1,0.9);
			\draw[arrows=->] (3,1.5)-- (4,1.22);
			\draw (4,1.22) node[below]{\tiny{$\tilde{v}+\zeta = \V(0)$}} ;
			\draw (4,1.5) node[right]{\tiny{$\V(\tau_{-})$}} ;
			\draw[arrows=->] (3,1.5)-- (4,1.5);
			\draw[arrows=->] (3,1.5)-- (3.95,1.8);
			\draw (3.95,1.8) node[right]{\tiny{$\V(\tau)$}} ;
			\draw[arrows=->] (3,1.5)-- (3.8,2.1);
			\draw (3.8,2.1) node [above right] {\tiny{$v+\zeta = \V(1)$}};
			\filldraw (3,1.5) circle[radius=1.5pt];
			\draw (3,1.5) node [above left] {$x$};
		\end{tikzpicture}
	\end{center}
	\caption{\cite{CD2023} $\tilde{v}$, $\V(\tau)$, and  trajectories on projected plane $S_{(x, v, \bv, \zeta)} := x + \text{span}\{ v+\zeta, \bar{v}+\zeta \}$} \label{fig2}
\end{figure}
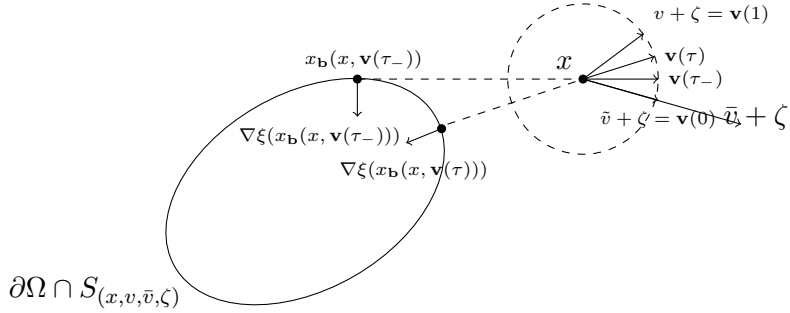

Now, we define the Specular singularity \cite{CD2023} which is crucial one to measure grazing singularity of specular reflection in three-dimensional billiard trajectory.

\begin{definition}[Specular Singularity] \label{def:sing_1}
Let $x, \tilde{x} \in \Omega$, and $v \in \mathbb{R}^3$ be such that $(x - \tilde{x}) \cdot v = 0$. Suppose $\xb(\X(\tau), v)$ is well-defined on $\p\O$. We define 
\begin{align*}
        \mathfrak{S}_{sp}(\tau;x,\tx, v) 
        : = 
		\frac{- \nabla \xi(\xb( \X(\tau), v)) \cdot v  }{
		\left| \frac{\dot{\X}}{|\dot \X|} \cdot \nabla \xi (\xb( \X(\tau),v))\right|}.
\end{align*}

\noindent Let $x\in\O$, $v, \tv, \zeta \in\R^{3}$ be such that $|v+\z|=|\tv+\z|$. Suppose $\xb(x, \V(\tau))$ is well-defined on $\p\O$. We define 
\begin{align*}
\mathfrak{S}_{vel}(\tau;x,  v,  \tilde {v}, \zeta) &: =  
		\frac{- \nabla \xi(\xb(x, \V(\tau))) \cdot \V(\tau)  }{\tb(x, \V(\tau))
		\left| \frac{\dot{\V}(\tau)}{|\dot{\V}(\tau)|}\cdot \nabla \xi (\xb(x , \V(\tau)))\right|
		}.
\end{align*}
\end{definition}

Averaging inverse of the Specular singularities $\mathfrak{S}_{sp}$ and $\mathfrak{S}_{vel}$ (during $\tb < \infty$) is a key step to handle fraction estimates for billiard trajectories. 
\begin{definition} \label{def:sing_2}
         Let $t>0, \,x, \tilde{x} \in \Omega$ and $v \in \mathbb{R}^3$ be such that $(x - \tx) \cdot v = 0$. We define
	\begin{align*}
		\mathcal{T}_{sp}(x, \tx, v;t)  &:=   \fint_{0}^{1} \frac{1}{\mathfrak{S}_{sp}(\tau; x, \tx, v)}   
		\mathbf{1}_{ \Big\{ \substack{
				\tb(\X(\tau), v) < \infty, \ 0\leq \tau \leq 1  
				\\ 
				\min_{0\leq \tau \leq 1  }\tb(\X(\tau), v) \leq t }  \Big\} } d\tau
		\\
		&\quad +  \fint_{\tau_{-}}^{1} \frac{1}{\mathfrak{S}_{sp}(\tau; x, \tx, v)}     
		\mathbf{1}_{ \Big\{ \substack{
				\tb(\X(\tau), v) < \infty, \ {0\leq} \tau_{-}\leq \tau \leq   1
				\\ 
				\min_{ \tau_{-}\leq \tau \leq 1  }\tb(\X(\tau), v) \leq t }  \Big\} } d\tau
		\\
		&\quad +  \fint_{0}^{\tau_{-}} \frac{1}{\mathfrak{S}_{sp}(\tau; x, \tx, v)}  
		\mathbf{1}_{ \Big\{ \substack{
				\tb(\X(\tau), v) < \infty, \ 0\leq \tau \leq \tau_{-}  
				\\ 
				\min_{0\leq \tau \leq \tau_{-}  }\tb(\X(\tau), v) \leq t }  \Big\} } d\tau.
       \end{align*}
	 If $\tb(\X(\tau), v) < \infty$ for $0\leq \tau \leq 1$ or $\tau_{-} \leq \tau \leq 1,\, \tau_{-} \neq 1$ or $0\leq \tau \leq {\tau_{-}},\, \tau_{-} \neq 0$, then $\partial\O \cap S_{(x,\tx,v)}$ is closed curve.(i.e., \( \partial\O \cap S_{(x, \tilde{x}, v)} \) is neither the empty set nor a single point.)
 
    Let $t>0,\,x \in \Omega$ and $v, \tv, \z \in \mathbb{R}^3$ be such that $|v+\z|=|\tv+\z|$.
	We define
	\begin{align*}
		\mathcal{T}_{vel}(x, v, \tv, \zeta;t)  &:= \fint_{0}^{1} \frac{1}{\mathfrak{S}_{vel}(\tau; x, v, \tv, \zeta)} 
		\mathbf{1}_{ \Big\{ \substack{
				\tb(x, \V(\tau)) < \infty, \ 0\leq \tau \leq 1  
				\\ 
				\min_{0\leq \tau \leq 1}\tb(x, \V(\tau)) \leq t }  \Big\} } d\tau  
		\\
		&\quad 
		+  \fint_{\tau_{-}}^{1} \frac{1}{\mathfrak{S}_{vel}(\tau; x, v, \tv, \zeta)}   
		\mathbf{1}_{ \Big\{ \substack{
				\tb(x, \V(\tau)) < \infty, \ 0 \leq \tau_{-}\leq \tau \leq 1  
				\\ 
				\min_{\tau_{-} \leq \tau \leq 1}\tb(x, \V(\tau)) \leq t }  \Big\} } d\tau  
		\\
		&\quad +  \fint_{0}^{\tau_{-}} \frac{1}{\mathfrak{S}_{vel}(\tau; x, v, \tv, \zeta)} 
		\mathbf{1}_{ \Big\{ \substack{
				\tb(x, \V(\tau)) < \infty, \ 0\leq \tau \leq \tau_{-}  
				\\ 
				\min_{0\leq \tau \leq \tau_{-}  }\tb(x, \V(\tau)) \leq t }  \Big\} } d\tau  .
    \end{align*}
     If $\tb(x, \V(\tau)) < \infty$ for $0\leq \tau \leq 1$ or $\tau_{-} \leq \tau \leq 1,\, \tau_{-} \neq 1$ or $0\leq \tau \leq {\tau_{-}},\, \tau_{-} \neq 0$, then {$\partial\O \cap S_{(x,v,\tv,\zeta)}$} is closed curve.(i.e., \( \partial\O \cap S_{(v,\tv,\zeta)} \) is neither the empty set nor a single point.)
    \end{definition}

{The following seminorms represent a central concept in this paper and play a crucial role in the proof of Theorem 1.4. 
To estimate the difference of the solution $f$, we compare the mild formulations at $(x,v)$ and $(\bar{x},\bar{v})$. 
This leads to the differences of $\Gamma_{\mathrm{gain}}(f,f)$ and $\nu(f)$ along the corresponding trajectories, which are estimated using Lemma \ref{lem:est_Gam} and Lemma \ref{lem:diff_nu}. 
In particular, the $u$-integration comes from the Carleman representation used in the estimate of $\Gamma_{\mathrm{gain}}$, while the $s$-integration originates from the Duhamel formula.}
    
    \begin{definition}[Seminorm] \label{def:iter}
    For $x \in \O, \,v \in \mathbb{R}^3$ and $\epsilon>0$, we define
    \begin{align*}
        G(x,v;\epsilon) = \ln\left(1+\frac{1}{|v|}\right)\mathbf{1}_{\{d(x,\partial\O) \leq \epsilon\}}+1
    \end{align*}
    For $t>0, \,\varpi >1$ and $\epsilon>0$, we define
        \begin{align*}
        &\mathfrak{X}(t,\varpi;\epsilon) :=\\
        &\sup_{ \substack{ (x,\bx,v,\bv)\in \O\times \O\times \R^{3}\times \R^{3} \\ 0<|(x,v)-(\bx,\bv)|\leq 1   } }  \left(G(x,v;\epsilon)+G(\bx,\bv;\epsilon)\right)^{-1} e^{-\varpi \langle v \rangle^{2} t} \\
        &\quad\quad\quad\times  \bigintsss_0^{t}  \max\left\{(t-s)^{\frac{1}{2}}, \frac{1}{\langle v \rangle}\right\}\bigintsss_{\mathbb{R}^3}  \left(\frac{e^{-c|u|^2}}{|u|}+|u|e^{-\frac{1}{4}|V(s;t,x,v)+u|^2}\right)\\
        &\quad\quad\quad\times \frac{|f(s,X(s;t,x,v),V(s;t,x,v)+u)-f(s,X(s;t,\bx,\bv),V(s;t,x,v)+u)|}{|X(s;t,x,v)-X(s;t,\bx,\bv)|} du ds 
    \end{align*} 
    and
     \begin{align*}
        &\mathfrak{V}(t,\varpi;\epsilon) :=\\
        &\sup_{ \substack{  (x,\bx,v,\bv)\in \O\times \O\times \R^{3} \times \R^{3} \\ 0<|(x,v)-(\bx,\bv)|\leq 1   } }  G^{-1}(x,v;\epsilon)  e^{-\varpi \langle v \rangle^{2} t} \bigintsss_0^{t}  \bigintsss_{\mathbb{R}^3}  \frac{e^{-c|u|^2}}{|u|}\\
        &\quad\quad\quad\times \frac{|f(s,X(s;t,x,v),V(s;t,x,v)+u)-f(s,X(s;t,x,v),V(s;t,\bx,\bv)+u)|}{|V(s;t,x,v)-V(s;t,\bx,\bv)|} du ds. 
    \end{align*}
   Here, the constant \( c > 0 \) comes from Lemma \ref{lem:est_Gam}. \\
\end{definition}

\subsection{Proof Sketch of Theorem \ref{thm:Holder0.5}} 
The key idea of \cite{CD2023} is to introduce `Specular singularity', which allows one to average the grazing
singularity near the grazing regime. Such a technique enables us to treat velocity integration over fraction of characteristics as grazing singularity integration in velocity side for fixed point \(x\) (and $\bar{x}$), i.e.,
\begin{align} \label{prev idea}
	\int_{\zeta}  {\frac{e^{-c|\zeta|^2}}{|\zeta|}} \frac{|X(s; t,x,v+\zeta)-X(s; t,\bar{x},\bar{v}+\zeta)|^{2\beta}}{|(x,v)-(\bar{x},\bar{v})|^{2\beta}} \lesssim \sup_{x\in\overline{\O}}\int_{\zeta}  {\frac{e^{-c|\zeta|^2}}{|\zeta|}}\frac{1}{| {\widehat{(v+\zeta)}}\cdot \nabla\xi(\xb(x,v+\zeta))|^{2\beta}}.
\end{align}

If the position $x$ is uniformly separated from the boundary, as velocity $v+\zeta$ moves away from the grazing velocity, we quickly escape from the region of grazing singularity. This is a major advantage in singular regime integrability afforded by uniform non-convexity of \(\p\O\). However, as position $x$ approaches the boundary \(\p\O\), the geometric effects of non-convexity gradually disappear, and the boundary exhibits almost flat geometric effects. Therefore, in previous results \cite{CD2023}, it was only possible to obtain \(\beta < \half\) in performing singular regime integration, without utilizing the geometric advantage of this non-convexity. \\

To overcome the previous result \(\beta < \half\) of \cite{CD2023}, we propose a novel idea called `dynamical singular regime integration.'  We will perform singular regime integration by following the characteristic \(X(s;t,x,v)\) for position, instead of fixing location \(x\), and examine whether it is integrable in time. In this case, singular regime integration will likely be possible for most of the time, and the time integration of singularities occurring when the position \(X(s;t,x,v)\) is very close to the boundary will be the main difficulty. The important point is that if velocity \(v\) is close to $0$ and our position is near the boundary, we will be confined to a region where flat geometry is dominant, and in this case, even performing \textit{dynamical singular regime integration} would only yield regularity comparable to that of \cite{CD2023}. Therefore, it is very natural for our new Theorem \ref{thm:Holder0.5} to have a norm structure implying singular behavior near $v=0$. \\

Unfortunately, however, for a general convex object \(\O^{c}\), performing such a dynamical estimate is extremely hard in general.  Instead of direct computation, we will first consider the uniform sphere case to show that explicit dynamical singular regime integration can be performed. Subsequently, we will consider the case of a general uniformly convex object \(\O^c\) and obtain results for general objects through a comparison principle with an outer circumscribing large sphere. The main idea of our paper is explained by the following steps. To convey the main idea more clearly, some formulas are presented in a simplified form.  \\

\textbf{(Step 1: Nonlocal to local estimate)} Fix \( t \in (0, T) \), and let \( x, \bar{x} \in \Omega \), \( v, \bar{v} \in \mathbb{R}^3 \) satisfy
\(|(x, v) - (\bar{x}, \bar{v})| \leq 1\) and \(0<T<1\). Using the mild formulation \eqref{f_expan} along the characteristics and the triangle inequality,
we decompose the difference of the solution evaluated at \((x,v)\) and \((\bx,\bv)\): 
{
\begin{align} 
& |f(t,x,v)-f(t,\bar{x}, \bar{v})|   \label{basic f-f} \\
&\lesssim_{\vartheta_0}   
\left|f(0,X(0 ), V(0 ))-f(0,\overline{X}(0 ), \overline{V}(0 ))\right| 
 \label{basic f-f1} \\
&\quad + \int^t_0 
\left|\Gamma_{\text{gain}}(f,f)(s,X(s ), V(s ))
-\Gamma_{\text{gain}}(f,f)(s,\overline{X}(s ), \overline{V}(s ))\right|  ds 
 \label{basic f-f2} \\
&\quad + (t+1) \mathcal{P}_2(\|w_0f_0\|_{\infty})
\int_{0}^{t } \left| \nu(f) (s, X(s), V(s ))  - \nu(f) (s, \overline{X}(s), \overline{V}(s)) \right| ds, \label{basic f-f4}  
\end{align}
where
\begin{align*}
   (X(s),V(s))= (X(s;t,x,v),V(s;t,x,v))
   \text{\quad and \quad}
   (\bX(s),\bV(s))= (X(s;t,\bx,\bv),V(s;t,\bx,\bv)).
\end{align*}
For notational simplicity, we suppress the dependence on \(\|w_0 f_0\|_\infty\). }

We apply the difference estimates of \( \Gamma_{gain}(f,f) \) and \( \nu(f) \), given in Lemma \ref{lem:est_Gam} and Lemma \ref{lem:diff_nu}, to \eqref{basic f-f2} and \eqref{basic f-f4}. Then, for \(0 \leq s \leq t\), we have
\begin{align} 
        &|f(t,x,v)- f(t,\bx,\bv)| \label{gain_sk}\\
        & {\lesssim_{\vartheta_0} }    \int_0^t \int_{\mathbb{R}^3}  {\frac{e^{-c|u|^2}}{|u|}}
       |f(s,X(s),V(s)+u)-f(s,\bX(s),V(s)+u)| du ds \notag \\
       &\quad+\int_0^t \int_{\mathbb{R}^3}  {\frac{e^{-c|u|^2}}{|u|}}
       |f(s,\bX(s),V(s)+u)-f(s,\bX(s),\bV(s)+u)| du ds + \text{other terms}. \notag 
\end{align}

We observe that the trajectory map \( (x,v) \mapsto X(s;t,x,v) \) is \( C^{0,\frac{1}{2}}_{x,v} \), and so is \( (x,v) \mapsto V(s;t,x,v) \) for \( s \neq t - \tb(x,v) \), outside uniformly convex domains. For example, consider the 2D circle \( x^2 + (y-1)^2 = 1 \) and points near grazing, such as \( x = (1,0) \), \( \bar{x} = (1,\varepsilon) \), and \( v = (1,0) \). Then,
\[
| \xb(x,v) - \xb(\bar{x}, v) | = \sqrt{2\varepsilon - \varepsilon^2} \simeq \sqrt{\varepsilon}, \quad \text{for } \varepsilon \ll 1,
\]
which illustrates the square-root-type sensitivity near grazing. The \( C_{x,v}^{0,\frac{1}{2}} \) estimate is stated in Lemma \ref{lem:tra_0.5}, and we recall it below. Assume \( t_1(t,x,v) \leq t_1(t,\bx,\bv)\). We have
\begin{align} \label{outtb_X}
    |X(s)-\bX(s)|  \lesssim  (1+\langle v \rangle (t-s) )|(x,v)-(\bx,\bv)|^{\frac{1}{2}}.
\end{align}
For \( s \notin [t_1(t,x,v), t_1(t,\bx,\bv)]\), we have
\begin{align} \label{outtb_V}
    |V(s)-\bV(s)|  \lesssim \langle v \rangle|(x,v)-(\bx,\bv)|^{\frac{1}{2}}.
\end{align}
For $s \in [t_1(t,x,v), t_1(t,\bx,\bv)]$, we have 
\begin{align} \label{intb}
    |\tb(x,v)-\tb(\bx,\bv)| \lesssim \min\left\{\frac{1}{|v|}, \frac{1}{|\bv|}\right\}|(x,v)-(\bx,\bv)|^{\frac{1}{2}}.
\end{align} 

We divide \(|(x,v)-( \bx,\bv)|^{\frac{1}{2}}\) and multiply \(e^{-\varpi \langle v \rangle^{2} t}\) to \eqref{gain_sk}:
\begin{align} 
    &e^{-\varpi \langle v \rangle^{2} t} \frac{|f(t,x,v)- f(t,\bx,\bv)|}{|(x,v)-( \bx,\bv)|^{\frac{1}{2}}}\notag\\
    & {\lesssim_{\vartheta_0} }  e^{-\varpi \langle v \rangle^{2} t}\int_0^t \frac{|X(s)-\bX(s)|}{|(x,v)-( \bx,\bv)|^{\frac{1}{2}}}\int_{\mathbb{R}^3}  {\frac{e^{-c|u|^2}}{|u|}}
    \frac{|f(s,X(s),V(s)+u)-f(s,\bX(s),V(s)+u)|}{|X(s)-\bX(s)|} du ds\label{gain_sk_X_1/2}\\
    &\quad +e^{-\varpi \langle v \rangle^{2} t} \int_0^t \int_{\mathbb{R}^3}  {\frac{e^{-c|u|^2}}{|u|}}
    \frac{|f(s,\bX(s),V(s)+u)-f(s,\bX(s),\bV(s)+u)|}{|(x,v)-( \bx,\bv)|^{\frac{1}{2}}} du ds + \text{other terms}.\label{gain_sk_V_1/2}
\end{align}

Next, we estimate \eqref{gain_sk_X_1/2} and \eqref{gain_sk_V_1/2}.
{
In \cite{CD2023}, the seminorms in Definition \ref{def_H} are defined
using only the velocity integration after taking the supremum in the
spatial variable. In contrast, starting from the mild formulation
\eqref{f_expan}, our estimates naturally lead to the seminorms
\(\mathfrak{X}\) and \(\mathfrak{V}\)
in Definition \ref{def:iter}, which involve both the velocity and time
integrations. This structure will later enable the dynamical singular
regime integration along trajectories. However, closing the iteration
requires introducing the weight $G$ and controlling the short
transition interval between the first bounce times, which inevitably
produces a small--velocity singularity. The treatment of this
singularity will be explained below.}

By \eqref{outtb_X} to \eqref{gain_sk_X_1/2}, we derive the seminorm \(\mathfrak{X}(t,\varpi;\epsilon)\), which serves as an upper bound for the difference of the solution: 
\begin{align*}
    (\ref{gain_sk_X_1/2}) 
     {\lesssim_{\vartheta_0} }  \langle v \rangle \mathfrak{X}(t,\varpi;\epsilon)\left(G(x,v;\epsilon)+G(\bx,\bv;\epsilon)\right).
\end{align*}

To estimate \eqref{gain_sk_V_1/2}, we divide the cases \(s \notin [t_1(t,x,v), t_1(t,\bx,\bv)]\) and \(s \in [t_1(t,x,v), t_1(t,\bx,\bv)]\). When \(s \notin [t_1(t,x,v), t_1(t,\bx,\bv)]\), we apply \eqref{outtb_V} to \eqref{gain_sk_V_1/2}. Similarly to the above, we obtain 
\begin{align*} 
    &e^{-\varpi \langle v \rangle^{2} t} \int_0^t \mathbf{1}_{\{s \notin [t_1(t,x,v), t_1(t,\bx,\bv)]\}}\int_{\mathbb{R}^3} {\frac{e^{-c|u|^2}}{|u|}}\frac{
    |f(s,\bX(s),V(s)+u)-f(s,\bX(s),\bV(s)+u)|}{|(x,v)-( \bx,\bv)|^{\frac{1}{2}}} du ds \\
    & {\lesssim_{\vartheta_0} } \langle v \rangle \mathfrak{V}(t,\varpi;\epsilon)G(\bx,\bv;\epsilon).
\end{align*}
When \(s \in [t_1(t,x,v), t_1(t,\bx,\bv)]\), we apply \eqref{intb}. 
To weaken the singularity \( |v|^{-1} \) in \eqref{intb}, we apply \eqref{pro:H_sub} together with the fact $|v| \geq \epsilon$ if \(d(x,\partial\O) \geq \epsilon\); see Lemma \ref{lem:gam_x_0.5}. Then 
\begin{align*}
    &e^{-\varpi \langle v \rangle^{2} t} \int_{t_1(t,x,v)}^{t_1(t,\bx,\bv)} \int_{\mathbb{R}^3}  {\frac{e^{-c|u|^2}}{|u|}} \frac{
    |f(s,\bX(s),V(s)+u)-f(s,\bX(s),\bV(s)+u)|}{|(x,v)-( \bx,\bv)|^{\frac{1}{2}}} du ds \\
    & \lesssim_{ {\vartheta_0,}\delta, \epsilon} \left(\mathbf{A}_{\frac{1}{2}}(f_0)+1 \right)
    \left(\frac{1}{|v|^{\delta}}\mathbf{1}_{\{d(x,\partial\O) \leq \epsilon\}}+\frac{1}{|\bv|^{\delta}}\mathbf{1}_{\{d(\bx,\partial\O) \leq \epsilon\}}+1
    \right)
\end{align*}
for \(0<\delta<1\). The above three inequalities yield
\begin{align} \label{1/2_result_sk}
\begin{split}
    &\langle v \rangle^{-1} e^{-\varpi \langle v \rangle^{2} t}\frac{|f(t,x,v)- f(t,\bx,\bv)|}{|(x,v)-(\bx,\bv)|^{\frac{1}{2}}} \\
    &\lesssim_{ {\vartheta_0}, \delta, \epsilon}  \left( \mathfrak{X}(t,\varpi;\epsilon)+\mathfrak{V}(t,\varpi;\epsilon)+\mathbf{A}_{\frac{1}{2}}(f_0)+1 \right)
    \left(\frac{1}{|v|^{\delta}}\mathbf{1}_{\{d(x,\partial\O) \leq \epsilon\}}+\frac{1}{|\bv|^{\delta}}\mathbf{1}_{\{d(\bx,\partial\O) \leq \epsilon\}}+1
    \right)
\end{split}
\end{align}
since \( G(x,v;\epsilon) \lesssim_{\delta} |v|^{-\delta}\mathbf{1}_{\{d(x,\partial\O) \leq \epsilon\}} \). 
\\

\textbf{(Step 2: Estimate of seminorms \(\mathfrak{X}\) and \(\mathfrak{V}\))} Next, we estimate \(\mathfrak{X}(t,\varpi;\epsilon)\) and \(\mathfrak{V}(t,\varpi;\epsilon)\). If we differentiate the trajectory \( (X(s), V(s)) \) and the backward exit time \( \tb(x,v) \); see \eqref{nabla_x_bv_b}, we find that the quantity \( |\nabla \xi(\xb(x,v)) \cdot v| \) appears, which diverges near grazing. From this, we estimate the difference between the trajectory and the backward time; see Lemma \ref{frac sim S} and Lemma \ref{lem:tra_sin_x}. By analyzing the specular singularity via an ODE, we observe that the entire singular behavior can be controlled by the values at the endpoints; see Section 5.2 of \cite{CD2023}. Thus,
\begin{align} \label{sk_dif_1}
    |(X,V)(s)-(\bX,\bV)(s)|, \quad |\tb(x,v)-\tb(\bx,\bv)| \lesssim \frac{  {|v|^k}}{|\nabla \xi(\xb(x,v))\cdot v|} |(x,v)-(\bx,\bv)|
\end{align}
{for $k=0,1,2.$}

We divide \eqref{gain_sk} by \(|(x,v)-(\bx,\bv)|\), and apply \eqref{sk_dif_1} instead of \eqref{outtb_X}, \eqref{outtb_V} and \eqref{intb}. Then, we obtain
\begin{align}
     &e^{-\varpi \langle v \rangle^{2} t}\frac{|f(t,x,v)- f(t,\bx,\bv)|}{|(x,v)-(\bx,\bv)|} \label{f_diff_t,x,v} \\
     &\lesssim  \left( \mathfrak{X}(t,\varpi;\epsilon)+\mathfrak{V}(t,\varpi;\epsilon)+\mathbf{A}_{\frac{1}{2}}(f_0)+1 \right)\frac{  {|v|^k}}{|\nabla \xi(\xb(x,v))\cdot v|} \notag
\end{align}
instead of \eqref{1/2_result_sk}. This process is detailed in Section 4.  {For convenience of exposition, we restrict ourselves to the case \(k=0\) in what follows.}

{In \cite{CD2023}, the seminorms are obtained by integrating
\eqref{f_diff_t,x,v} with respect to the velocity variable,
so that the singularity appears only in the velocity integration.
In contrast, in our formulation we replace \( (t,x,v) \) by
\( (s,X(s),V(s)+u) \) in \eqref{f_diff_t,x,v} and integrate with respect to
\(u\) and \(s\). As a result, the singularity is expressed through the
dynamical integral \eqref{iter_sk}.}

Applying this substitution and integrating with respect to \(u\) and \(s\),
and using the change of variable \(u \rightarrow w = V(s)+u\), we obtain
\begin{align} \label{iter_sk}
\begin{split}
    &e^{-\varpi \langle v \rangle^{2} t}\int_0^t \int_{\mathbb{R}^3}  {\frac{e^{-c|u|^2}}{|u|}}
    \frac{|f(s,X(s),V(s)+u)-f(s,\bX(s),\bV(s)+u)|}{|(X,V)(s)-(\bX,\bV)(s)|} du ds  \\
    &\lesssim  \int_0^t e^{-\varpi \langle v \rangle^{2} (t-s)}\int_{\mathbb{R}^3}  {\frac{e^{-c|w-V(s)|^2}}{|w-V(s)|}} \frac{ 1}{|\nabla \xi(\xb(X(s),w))\cdot w|} dw ds\\
    &\quad \times \left(\sup_{0 \leq t \leq T}\mathfrak{X}(t,\varpi;\epsilon)+\sup_{0 \leq t \leq T}\mathfrak{V}(t,\varpi;\epsilon)+\mathbf{A}_{\frac{1}{2}}(f_0)+1\right)
\end{split}
\end{align}
; see Lemma \ref{pro:X<1} and Lemma \ref{pro:V<1}. To finish estimating \( \mathfrak{X} \) and \( \mathfrak{V} \), it remains to control the integral of the singular term in \( u \) and \( s \) in the (RHS) of \eqref{iter_sk}.\\

\textbf{(Step 3: Dynamical singular regime integration)}  { Fixing \(x\in\Omega\), we explain how to integrate with respect to \(u\)
the inverse incoming angle
\(|\nabla \xi(\xb(x,u))\cdot \hat{u}|^{-1}\) arising from trajectories
starting at \(x\) and hitting the boundary. 
The difficulty occurs when \(X(s;t,x,u)\) grazes the boundary, since the incoming angle vanishes. 
We denote this quantity by \(\theta\), which measures how close the velocity direction is to the tangential direction of the boundary 
(\(\theta \to 0\) corresponds to grazing trajectories).}

{When the spatial point $x$ lies on the boundary, the singularity behaves
like a non-integrable \(1/\theta\)-type singularity. In our setting,
however, the point \(x\) lies in the interior of the domain. In this case,
the geometry introduces an additional scale proportional to the distance
\(d(x,\partial\Omega)\) from the boundary.}

{To illustrate this effect, we first consider the circle in two
dimensions, where the grazing singularity admits a more explicit form.
Heuristically, the singular integral behaves like
\begin{align*}
    \int \frac{1}{|\nabla \xi(\xb(x,u))\cdot \hat{u}|} du
    \;\sim\;
    \int \frac{1}{\sqrt{\theta(\theta+d(x,\partial\Omega))}} d\theta,
\end{align*}
which is integrable and yields a logarithmic bound; see Lemma \ref{lem:int_cir}. From this, the integral in \eqref{prev idea} becomes integrable at the critical exponent $\beta=\frac{1}{2}$.}

To handle general convex objects, we use a kind of comparison principle. First, we slice the domain with planes passing through the fixed point \( x \), so that each intersection forms a uniformly convex planar curve. On each slice, we identify the grazing point and construct a circumscribed circle with smaller curvature that passes through it. {Once} a trajectory from \((x,v)\) undergoes specular reflection on the circumscribed circle, the trajectory stays near singular regime longer than the convex curve case, which enhances the grazing singularity; see Lemma \ref{lem:angle_com}. By comparing the boundary curve with the circumscribed circle, we bound the original integral by one over the circle. Then, we obtain the following using concavity effect and assuming that \(x\notin \p\O\)
\begin{align*}
         \bigintsss_{\mathbb{R}^{3}}  {\frac{e^{-c|u-v|^2}}{|u-v|}}\frac{ 1}{|\nabla \xi(\xb(x,u))\cdot u|} du  \lesssim \ln\left(1+\frac{1}{d(x,\partial\O)} \right)+1
    \end{align*}
; see Lemma \ref{lem:int_sing}. This means that the singular integral over incoming angles can be controlled by the distance from \( x \) to the boundary.

Now, we perform \textit{dynamical} singular regime integration using above \textit{static} singular regime integration. We perform integration in time \(s\), following the trajectory \( X(s;t,x,v) \) to get  
\begin{align*}
    \int_0^t e^{-\varpi \langle v \rangle^{2} (t-s)} \ln \left( 1+ \frac{1}{d(X(s;t,x,v),  \partial\O)} \right) ds 
    &\lesssim { \int_0^\infty e^{-\varpi \langle v \rangle^2 \tau}\ln \left( 1+ \frac{1}{|v|\tau}\right) d\tau}\\
    &\lesssim \frac{1}{\sqrt{\varpi} {\langle v \rangle}}\left[\ln\left(1+\frac{1}{|v|}\right)+1\right]
\end{align*}
; see Lemma \ref{lem:ds}. This is possible because the distance to the boundary, \( d(X(s),\partial\O) \), is expressed in terms of the time variable \(t, s \) and \(t_1(t,x,v)\), which enables us to estimate the singular integral.

In particular, when $d(x,\partial\O) \geq \epsilon$, we also obtain 
\begin{align*} 
      \bigintsss_0^t e^{-\varpi \langle v \rangle^{2} (t-s)}\ln \left(  1+ \frac{1}{d(X(s),\partial\O)} \right) ds \lesssim_{\epsilon} \frac{1}{\sqrt{\varpi} {\langle v \rangle}}
    \end{align*} 
; see Lemma \ref{lem:ds_2}. This shows that when \( x \) is not close to the boundary, the singular behavior near \( |v| = 0 \) does not appear. 

 {This explains how both the $u$- and $s$-integrations can be controlled
in the dynamical singular regime.}
\\

\textbf{(Step 4: Conclusion)} Applying the integration of singularities from \textbf{(Step 3)} to \eqref{iter_sk}, we obtain
\begin{align}
    &e^{-\varpi \langle v \rangle^{2} t}\int_0^t \int_{\mathbb{R}^3}  {\frac{e^{-c|u|^2}}{|u|}}
    \frac{|f(s,X(s),V(s)+u)-f(s,\bX(s),\bV(s)+u)|}{|(X,V)(s)-(\bX,\bV)(s)|} du ds \notag \\
    &\lesssim_{\epsilon} \frac{1}{\sqrt{\varpi}} \left[\ln\left(1+\frac{1}{|v|}\right)\mathbf{1}_{\{d(x,\partial\O) \leq \epsilon\}}+\ln\left(1+\frac{1}{|\bv|}\right)\mathbf{1}_{\{d(\bx,\partial\O) \leq \epsilon\}}+1 \right] \label{sk_G}\\
    &\quad \times \left(\sup_{0 \leq t \leq T}\mathfrak{X}(t,\varpi;\epsilon)+\sup_{0 \leq t \leq T}\mathfrak{V}(t,\varpi;\epsilon)+\mathbf{A}_{\frac{1}{2}}(f_0)+1\right).\notag 
\end{align}
The methods used to estimate \( \mathfrak{X}(t,\varpi;\epsilon) \) and \( \mathfrak{V}(t,\varpi;\epsilon) \) differ slightly due to the different powers of \( \langle v \rangle \) involved; see Lemma \ref{pro:X<1} and Lemma \ref{pro:V<1}. In particular, when estimating \( \mathfrak{V}(t,\varpi;\epsilon) \), only \( G(\bar{x},\bar{v}) \) appears in \eqref{sk_G}, while for \( \mathfrak{X}(t,\varpi;\epsilon) \), both \( G(x,v) \) and \( G(\bar{x},\bar{v}) \) are present. For convenience of exposition, we present the two cases in a unified form.

By multiplying \( (G(x,v) + G(\bx,\bv))^{-1} \), we deduce the estimate
\begin{align*}
     \mathfrak{X}(t,\varpi;\epsilon)+\mathfrak{V}(t,\varpi;\epsilon) \lesssim_{\epsilon} \mathbf{A}_{\frac{1}{2}}(f_0)+1
\end{align*}
holds for sufficiently large \(\varpi \gg 1\); see Proposition \ref{pro:iter_pro}.

Using the above iterated estimation, we obtain
\begin{align*}
    \langle v \rangle^{-1} e^{-\varpi \langle v \rangle^{2} t}\frac{|f(t,x,v)- f(t,\bx,\bv)|}{|(x,v)-(\bx,\bv)|^{\frac{1}{2}}} 
    \lesssim_{\delta, \epsilon} \left(\mathbf{A}_{\frac{1}{2}}(f_0)+1 \right) \left(\frac{1}{|v|^{\delta}}\mathbf{1}_{\{d(x,\partial\O) \leq \epsilon\}}+\frac{1}{|\bv|^{\delta}}\mathbf{1}_{\{d(\bx,\partial\O) \leq \epsilon\}}+1
        \right)
\end{align*}
from \eqref{1/2_result_sk}. Lastly, we multiply \(W^{-1}((x,v),(\bx,\bv);\epsilon, \delta)\) and take supremum over \(x, \bx \in \O\) and \(v, \bv \in \mathbb{R}^3\). \\

\section{Preliminaries}

We estimate the difference between $\Gamma_{gain}(f,f)$ and $\nu(f)$ evaluated at $(x, v)$ and $(\bar{x}, \bar{v})$, and obtain an upper bound for it. These lemmas are applied in Sections~4, 6, and 7 to estimate the difference between solutions.

\begin{lemma}\label{lem:est_Gam} 
	Let $w(v) = e^{\vartheta|v|^{2}}$ for $0<\vartheta < \frac{1}{4}$. 
    For $t>0, \, x,\bx \in \O$, and $v \in \mathbb{R}^3$, we have
   \begin{align} \label{gamma_x}
    \begin{split}
         &\left|\Gamma_{gain}(f,f)(t,x,v)- \Gamma_{gain}(f,f)(t,\bx,v) \right| \\&\lesssim \|wf(t)\|_{\infty}  \int_{\mathbb{R}^3} \frac{1}{|u|}e^{-c|u|^2}|f (t,x, v+u) - f (t,\bar{x}, v+u)|  du
    \end{split}
    \end{align} 
    for some $c>0$.
    For $t>0, \, x \in \O$, and $v,\bv \in \mathbb{R}^3$, we have
     \begin{align} \label{gamma_v}
    \begin{split}
        &\left|\Gamma_{gain}(f,f)(t,x,v)- \Gamma_{gain}(f,f)(t,x,\bv) \right|\\
        &\lesssim \|wf(t)\|_{\infty}  \int_{\mathbb{R}^3}  
			\frac{1}{|u|}e^{-c|u|^2}| f(t,x, v+u) - f(t,x, \bar{v}+u) |  du \\
        &\quad+\|wf(t)\|_{\infty}^2 \min \left\{ \langle v \rangle^{-1}, \langle \bv \rangle^{-1} \right \} |v-\bv|
    \end{split}
    \end{align} 
     for some $c>0$. For $t>0, \,x \in \O$, and $v \in \mathbb{R}^3$, we have
    \begin{align} \label{gamma_upper}
         \Gamma_{gain}(f,f) (t,x,v)\lesssim_{\vartheta}   \|w f(t) \|^2_{\infty}.
    \end{align}
\end{lemma}
\begin{proof}
Lemma 3.2 in \cite{CD2023} provides \eqref{gamma_x} and \eqref{gamma_v}. Since $|u|^2+|v|^2=|u'|^2+|v'|^2$, it follows that
    \begin{align*}
    \Gamma_{gain}(f,f) (t,x,v) 
    &= \int_{\mathbb{R}^3} \int_{\mathbb{S}^2_{+}} |(v-u)\cdot \omega | \frac{ \sqrt{\mu(u)}}{w(u)w(v)} (w(u')f(t,x,u')) (w(v')f(t,x,v')) d\omega du \\
    &\lesssim_{\vartheta}  \|w f(t) \|^2_{\infty}. 
    \end{align*}
\end{proof}

\begin{lemma} \label{lem:diff_nu}
Let $w(v) = e^{\vartheta|v|^{2}}$ for $0<\vartheta < \frac{1}{4}$. 
    For $t>0, \, x,\bx \in \O$, and $v \in \mathbb{R}^3$, we have
    \begin{align} \label{nu_x}
           |\nu(f)(t,x,v)-\nu(f)(t,\bx,v)| \lesssim \int_{\mathbb{R}^3}  |u| e^{-\frac{1}{4}|u+v|^2} |f(t,x,v+u)-f(t,\bx,v+u)| du.
    \end{align} 
    For $t>0, \, x \in \O$, and $v,\bv \in \mathbb{R}^3$, we have
    \begin{align} \label{nu_v}
        |\nu(f)(t,x,v)-\nu(f)(t,x,\bv)| \lesssim |v-\bv| \|f(t)\|_{\infty}.
    \end{align} 
    For $t>0, \,x \in \O$, and $v \in \mathbb{R}^3$, we have
    \begin{align} \label{nu_upper}
        |\nu(f)(t,x,v)| \lesssim \langle v\rangle \|f(t)\|_{\infty}.
    \end{align}
\end{lemma}
\begin{proof}
From the definition of $\nu(f)$, we compute
 \begin{align*}
           |\nu(f)(t,x,v)-\nu(f)(t,\bx,v)|  &=
            \int_{\mathbb{R}^3}  \int_{\mathbb{S}^2_{+}} |(u-v)\cdot \omega |  \sqrt{\mu(u)}  | f(t,x,u)-f(t,\bx,u)| d\omega du \\
            &\lesssim \int_{\mathbb{R}^3}  |u| e^{-\frac{1}{4}|u+v|^2} |f(t,x,v+u)-f(t,\bx,v+u)| du, 
    \end{align*} 
and by using triangle inequality, 
    \begin{align*}
           |\nu(f)(t,x,v)-\nu(f)(t,x,\bv)| \leq
            \int_{\mathbb{R}^3}  \int_{\mathbb{S}^2_{+}} |v-\bv|  \sqrt{\mu(u)} f(t,x,u) d\omega du\lesssim |v-\bv| \|f(t)\|_{\infty}.
    \end{align*}
\end{proof}

We estimate the back exit time \( \tb \) and the difference between the trajectories in a uniformly non-convex domain, as defined in Definition \ref{def:domain}. Assuming \( \tb(x, v) < +\infty \), the following result follows from a direct computation:
		\Be\label{nabla_x_bv_b}
		\begin{split}
			&	\nabla_{x} \tb = \frac{ \nabla \xi(\xb)}{ \nabla \xi (\xb) \cdot v}
			,
			\  \ \
			\nabla_{v} \tb =
			- \tb 	\nabla_{x} \tb,
			\\
			&		\nabla_{x} x_{\mathbf{b}} =  I - \frac{v\otimes n(\xb)}{v\cdot n(\xb)},
			\  \ \ \nabla_{v} x_{\mathbf{b}} = - \tb  \nabla_{x} x_{\mathbf{b}},  \\
			& \nabla_{x} n(\xb) = \frac{1}{|\nabla\xi(\xb)|} \Big( I - n(\xb)\otimes n(\xb) \Big)\nabla^{2}\xi(\xb),  
		\end{split}
		\Ee
and we briefly explain how to derive \eqref{nabla_x_bv_b}. Since $\xb(x,v) = x - v\tb(x,v)$, we compute
\begin{align*}
\nabla_x(\xi(\xb)) = (\nabla_x \xb)^{T}(\nabla_x \xi(\xb))
= \nabla_x \xi(\xb) - \nabla_x \tb( \nabla_x \xi (\xb) \cdot v).
\end{align*}
To obtain $\nabla_x \tb$, we differentiate the identity $\xi(\xb) = 0$ with respect to $x$. Similarly, $\nabla_v \tb$ is derived by differentiating $\xi(\xb)$ with respect to $v$. The expressions for $\nabla_x \xb$ and $\nabla_v \xb$ follow directly from differentiating the formula $\xb(x,v) = x - v\tb(x,v)$. Finally, $\nabla_x n(\xb)$ is obtained by differentiating both sides of the relation $|\nabla \xi(x)| n(x) = \nabla \xi(x)$ with respect to $x$. More details can be found in \cite{Cao-Kim-Lee, GuoKim_boundary}. \\

Now, let us recall some useful results of \cite{CD2023}. First, fraction of characteristics can be controlled by specular singularities defined in Definition \ref{def:sing_1}.

\begin{lemma}[Lemma 5.1 and Lemma 5.3 of \cite{CD2023}] \label{frac sim S}
	Suppose the domain is given as in Definition \ref{def:domain} and \eqref{convex_xi}.  \\
	\noindent	(1) Let $t > 0$, $x, \tilde{x} \in \Omega$, and $v \in \mathbb{R}^3$ be such that $(x - \tilde{x}) \cdot v = 0$. When $\min\{t_1(t,x,v), t_1(t,\tx,v)\} \neq -\infty$ and $\min_{0\leq \tau \leq 1  }\tb(\X(\tau), v) \leq t$, we obtain
        \begin{align} \label{tb-tb_x}
            |\tb(x,v)-\tb(\tx,v)| 
            \leq  \int_{0}^{1} \frac{1}{\mathfrak{S}_{sp}(\tau; x, \tx, v)}   
		\mathbf{1}_{ \Big\{ \substack{
				\tb(\X(\tau), v) < \infty, \ 0\leq \tau \leq 1  
				\\ 
				\min_{0\leq \tau \leq 1  }\tb(\X(\tau), v) \leq t }  \Big\} } d\tau. 
        \end{align}
            For $|x - \tx|\leq 1$, $0 \leq s \leq t$, and either
    \begin{align*}
         s \leq \min\{ t_1(t,x,v ), t_1(t,\tx, v ) \}  \text{\quad or \quad} s > \max\{ t_1(t,x,v ), t_1(t,\tx, v ) \},
    \end{align*}
    we obtain 
	\begin{align} 		\label{est:V/x}
    \begin{split}
		&|V(s;t,x,v) - V(s;t, \tilde x, v)|  \\
		&\leq |x-\tx| \left(|v | + 
		|v |^{2}
		\int_{0}^{1}
		\frac{1}{\mathfrak{S}_{sp}(\tau; x, \tx, v )}\mathbf{1}_{ \Big\{ \substack{
				\tb(\X(\tau), v) < \infty, \ 0\leq \tau \leq 1  
				\\ 
				\min_{0\leq \tau \leq 1  }\tb(\X(\tau), v) \leq t }  \Big\} }
		d\tau \right)
    \end{split}
	\end{align} 
    and
    \begin{align} \label{est:X/x}
    \begin{split}
		& |X(s;t,x,v ) - X(s;t, \tilde x, v )| \\
		&\leq  |x- \tilde x| \left(
		1 + |v |(t-s) + 
		|v |^{2}(t-s)
		\int_{0}^{1}
		\frac{1}{\mathfrak{S}_{sp}(\tau; x, \tx, v )}\mathbf{1}_{ \Big\{ \substack{
				\tb(\X(\tau), v) < \infty, \ 0\leq \tau \leq 1  
				\\ 
				\min_{0\leq \tau \leq 1  }\tb(\X(\tau), v) \leq t }  \Big\} }
		d\tau \right).  
    \end{split}
    \end{align}
	\noindent (2)  Let $t>0,\,x \in \Omega$ and $v, \tv, \z \in \mathbb{R}^3$ be such that $|v+\z|=|\tv+\z|$. When $\min\{t_1(t,x,v+\z), t_1(t,x,\tv+\z)\} \neq -\infty$ and $\min_{0\leq \tau \leq 1  }\tb(x, \V(\tau)) \leq t$, we obtain
    \begin{align} \label{tb-tb_v}
    |\tb(x,v+\z)-\tb(x,\tv+\z)| 
    \leq |v-\tv|\int_{0}^{1} \frac{1}{\mathfrak{S}_{vel}(\tau; x, v, \tv, \zeta)} 
		\mathbf{1}_{ \Big\{ \substack{
				\tb(x, \V(\tau)) < \infty, \ 0\leq \tau \leq 1  
				\\ 
				\min_{0\leq \tau \leq 1}\tb(x, \V(\tau)) \leq t }  \Big\} } d\tau.  
    \end{align}
    For $|v - \tv|\leq 1$, $0 \leq s \leq t$, and either
    \begin{align*}
        s \leq \min\{ t_{1}(t,x,v+\zeta), t_{1}(t,x, \tv+\zeta) \} \text{\quad or \quad}  s > \max\{ t_{1}(t,x,v+\zeta), t_{1}(t,x, \tv+\zeta) \},
    \end{align*}
    we obtain
	\begin{align} \label{est:V/v}
    \begin{split}
		& |V(s;t,x,v+\zeta) - V(s;t, x, \tv+ \zeta)|   \\
		& \lesssim |v-\tv| \left(
		1
		+
		|v+\zeta| (t-s) 
		+
		|v+\zeta|^{2}
		\int_{0}^{1}
		\frac{1}{ \mathfrak{S}_{vel}(\tau; x, v, \tv, \zeta)}	\mathbf{1}_{ \Big\{ \substack{
				\tb(x, \V(\tau)) < \infty, \ 0\leq \tau \leq 1  
				\\ 
				\min_{0\leq \tau \leq 1}\tb(x, \V(\tau)) \leq t }  \Big\} } 
		d\tau \right)
    \end{split}
	\end{align}
    and 
    \begin{align} \label{est:X/v}
    \begin{split}
        	& |X(s;t,x,v+\zeta) - X(s;t, x, \tv+ \zeta)| \\
		& \lesssim |v-\tv| \left(
		(t-s)
		+
		|v+\zeta| (t-s)^{2} 
		+
		|v+\zeta|^{2} (t-s)
		\int_{0}^{1}
		\frac{1}{ \mathfrak{S}_{vel}(\tau; x, v, \tv, \zeta)}	\mathbf{1}_{ \Big\{ \substack{
				\tb(x, \V(\tau)) < \infty, \ 0\leq \tau \leq 1  
				\\ 
				\min_{0\leq \tau \leq 1}\tb(x, \V(\tau)) \leq t }  \Big\} } 
		d\tau\right). 
    \end{split}
    \end{align}
\end{lemma}
\begin{proof}
The estimates of the differences between $\tb$ at perturbed $x$ and $v+\z$, given in \eqref{tb-tb_x} and \eqref{tb-tb_v}, come from Lemma 5.3 in \cite{CD2023}, which are obtained using $\nabla_x \tb$. Next, $\nabla_{x,v} V(s;t,x,v)$ and $\nabla_{x,v} X(s;t,x,v)$ were calculated for $s \leq t_1(t,x,v)$, when $t_1(t,x,v) \neq -\infty$, based on $\nabla_x \tb, \nabla_x \xb$, and $\nabla_x n(\xb)$ in \eqref{nabla_x_bv_b}. Using $\nabla_{x,v} V(s;t,x,v)$ and $\nabla_{x,v} X(s;t,x,v)$, the estimates \eqref{est:V/x}, \eqref{est:X/x}, \eqref{est:V/v}, and \eqref{est:X/v} were obtained in Lemma 5.1 of \cite{CD2023}.
\end{proof}

Let us recall uniform seminorm estimates of \cite{CD2023}. We note that improved form of seminorms in this paper is $\mathfrak{X}$ and $\mathfrak{V}$ in Definition \ref{def:iter}.
    \begin{definition}[Seminorm in \cite{CD2023}]\label{def_H}
	For $s, \varpi_1 > 0$ and $0<\beta<\frac{1}{2}$, we define
	\begin{equation} \notag
		\begin{split}
			\mathfrak{H}^{2\b}_{sp}(s,\varpi_1) &:=  \sup_{\substack{(x,\bx,v)\in \O \times \O \times \R^{3}  \\ 0<|x-\bx|\leq 1}} e^{ - \varpi_1 \langle v  \rangle^{2} s } 
			\int_{\mathbb{R}^3} \frac{1}{|\z|}e^{-c|\z|^2} \frac{ | f(s, x, v+\zeta) - f(s, \bx, v+\zeta) | }{ | x - \bx |^{2\b} } d\zeta,
		\end{split}
	\end{equation} 
	\begin{equation} \notag
		\begin{split}
			\mathfrak{H}^{2\b}_{vel}(s,\varpi_1) &:= \sup_{\substack{(x,v,\bv)\in {\O}\times \mathbb{R}^3 \times \R^3  \\ 0<|v-\bv|\leq 1}} e^{ - \varpi_1 \langle v  \rangle^{2} s } 
			\int_{\mathbb{R}^3} \frac{1}{|\z|}e^{-c|\z|^2}\frac{ | f(s, x, v+\zeta) - f(s, x, \bv+\zeta) | }{ | v - \bv |^{2\b} } d\zeta.
		\end{split}
	\end{equation}
	Here, the constant \( c > 0 \) comes from Lemma \ref{lem:est_Gam}.
\end{definition}

The following uniform estimates for seminorms defined in Definition \ref{def_H}, are crucial ingredients for the Theorem in \cite{CD2023}. Note that $\b < \half$ in the case.
\begin{proposition}[Proposition 6.7 of \cite{CD2023}] \label{pro:H}
	There exists $\varpi_1 \gg \mathcal{P}_2(\|w_0f_0\|_{\infty})$, depending on $\beta$ and $\vartheta_0$, such that
	\begin{align*}
		 \sup_{0\leq s \leq T_1}\mathfrak{H}_{sp}^{2\b}(s, \varpi_1) + \sup_{0\leq s \leq T_1}\mathfrak{H}_{vel}^{2\b}(s, \varpi_1)  \lesssim_{\vartheta_0,\beta} 	\mathbf{A}_{\b}(f_0)
		+ \|w_{0} f_{0}\|_{\infty}
	\end{align*}
	for sufficiently small $T_1>0$ such that $\varpi_1 T_1 \ll 1$ and $0<\beta<\half$.
\end{proposition}

 {
\begin{proof}
	[Brief sketch of the proof] This is one of the main proposition in \cite{CD2023}. To obtain the result, we first estimate the difference of characteristic $(X,V)(s;t,\bar{x},\bar{v})-(X,V)(s;t,{x},{v})$ in terms of specular singularity $\mathfrak{S}_{sp}$ and $\mathfrak{S}_{vel}$ defined in Definition \ref{def:sing_1}. Using crucial geometric properties of uniform convexity of $\O$, we can perform averaging of specular singularities $\int_{0}^{1} \mathfrak{S}_{sp, vel}^{-1}(\tau)$ to obtain 
	\begin{equation*}
		\begin{split}
					&\frac{|X(s;t,x,v) - X(s;t,\bar{x},v)|}{|x-\bar{x}|} 
				\lesssim \fint \frac{|v|^2(t-s)}{\mathfrak{S}_{sp}(\tau;x,\tilde{x},v)}d\tau  \\
				&\lesssim \Bigg[	\frac{1}{|(v+\zeta)\cdot \nabla\xi(\xb(x,v+\zeta))|} + \frac{1}{|(v+\zeta)\cdot \nabla\xi(\xb(\tilde{x},v+\zeta))|} \Bigg]\times |v|^2(t-s)
		\end{split}
	\end{equation*}
	(We also get similar estimates for other three cases :  $|X(s;t,x,v) - X(s;t,x,\bar{v})|/|v-\bar{v}|$, $|V(s;t,x,v) - V(s;t,x,\bar{v})|/|v-\bar{v}|$, and $|V(s;t,x,v) - V(s;t,\bar{x},{v})|/|x-\bar{x}|$) Meanwhile, let us use apply Duhamel expansion (of the Boltzmann equation along the characteristic $(X(s;t,x,v), V(s;t,x,v))$) to obtain fraction estimates $|f(t,x,v+\zeta) - f(t,\bar{x}, {v}+\zeta)|/|x-\bar{x}|^{2\b}$. Multiplying by $\frac{1}{|\zeta|}e^{-c|\zeta|^2}$ and integrating with respect to $\zeta\in \R^3$, we obtain the following estimates of $\mathfrak{H}^{2\b}_{sp}(s, \varpi_1)$ :
	\begin{equation*}
		\begin{split}
			\mathfrak{H}^{2\b}_{sp}(t, \varpi_1) &\leq \big( \mathbf{A}_{\b}(f_0)
			+ \|w_{0} f_{0}\|_{\infty} \big) \\
			&\quad + \int_{0}^{t}e^{-\varpi_1\langle v \rangle^2(t-s)} \sup_{x,\bar{x}}\int_{\zeta}  \frac{e^{-c|\zeta|^2}}{|\zeta|} \Bigg[ \fint d\tau \frac{|v|^2(t-s)}{\mathfrak{S}_{sp}(\tau;x,\tilde{x},v)} \Bigg]^{2\b} \big(  \mathfrak{H}^{2\b}_{sp}(s, \varpi_1) + \mathfrak{H}^{2\b}_{vel}(s, \varpi_1)\big) ds
		\end{split}
	\end{equation*}
	In \cite{CD2023}, the authors treated singular regime integral $\int_{\zeta}  \Big[ \fint d\tau \frac{|v|^2(t-s)}{\mathfrak{S}_{sp}(\tau;x,\tilde{x},v)} \Big]^{2\b}$ non-dynamically and this is why we put $\sup_{x,\bar{x}}$ in front of the singular regime integral. The worst case happens when $x$ (or $\bar{x}$) locates on the boundary in which case the billiard characteristics does not undergo geometric concave effect. For such case, the best we can hope is
	\begin{equation*}
	\begin{split}
		&\int_{\zeta} \frac{e^{-c|\zeta|^2}}{|\zeta|}  \Bigg[ \fint d\tau \frac{|v|^2(t-s)}{\mathfrak{S}_{sp}(\tau;x,\tilde{x},v)} \Bigg]^{2\b} \\
		&\quad \lesssim |v|^{4\b}(t-s)^{2\b} \int_{\zeta} \frac{e^{-c|\zeta|^2}}{|\zeta|} 
		\Bigg[	\frac{1}{|(v+\zeta)\cdot \nabla\xi(\xb(x,v+\zeta))|^{2\b}} + \frac{1}{|(v+\zeta)\cdot \nabla\xi(\xb(\tilde{x},v+\zeta))|^{2\b}} \Bigg] < \infty
	\end{split}
	\end{equation*}
	when $\b < \frac{1}{2}$ since $\nabla\xi(\xb(x,v+\zeta)) \eqsim n(x)$ when $x\in\p\O$. Similar estimate can be done for $\mathfrak{H}^{2\b}_{vel}(t,\varpi_1)$ and we choose sufficiently large $\varpi_1$ to derive smallness from $\int_{0}^{t}e^{-\varpi_1\langle v \rangle^2(t-s)}ds$. This proves Proposition \ref{pro:H}.
\end{proof}
} 

Actually, in \cite{CD2023}, $\mathfrak{H}^{2\b}_{sp}(s,\varpi_1)$ and $\mathfrak{H}^{2\b}_{vel}(s,\varpi_1)$ are defined using the kernel
\begin{align*}
	k_c(v,v+\z) = \frac{1}{|\z|}e^{ - c|\zeta|^{2} - c \frac{ | |v|^{2}-|v+\zeta|^{2} |^{2} }{|\zeta|^{2}} }
\end{align*}
instead of $|\z|^{-1}e^{-c|\z|^2}$. However, even if we replace $k_c(v,v+\z)$ with $|\z|^{-1}e^{-c|\z|^2}$, as we did $\mathfrak{H}_{sp}^{2\b}(s)$ and $\mathfrak{H}_{vel}^{2\b}(s)$ above, Proposition \ref{pro:H} still holds. This can be verified from the proof of \cite{CD2023}. For $0 \leq s \leq T_1,\,x \in \O,\, v, \bv \in \mathbb{R}^3$ such that $|v-\bv| \leq 1$, and $0 <\beta<\half$, we have
\begin{align} \label{pro:H_sub}
\begin{split}
	&e^{-\varpi_1\langle v \rangle^{2}s}\int_{\mathbb{R}^3} \frac{1}{|\z|}e^{-c|\z|^2}| f(s, x, v+\zeta) - f(s, x, \bv+\zeta)| d\zeta \\
	&\lesssim_{\vartheta_0,\beta} | v - \bv |^{2\b} \left( \mathbf{A}_{\beta}(f_0) +  \|w_{0} f_{0}\|_{\infty} \right) \lesssim | v - \bv |^{2\b} \left( \mathbf{A}_{\frac{1}{2}}(f_0) +  \|w_{0} f_{0}\|_{\infty} \right)
\end{split}
\end{align}
from Proposition \ref{pro:H}. We will later use \eqref{pro:H_sub} in the proofs of Lemma \ref{lem:ga_V} and Lemma \ref{lem:gam_x_0.5}.  \\

\section{Difference estimates of characteristics}

\subsection{Spatial variation}

In Section 3.1, we analyze trajectory estimates under spatial variation based on Lemma \ref{frac sim S}-(1). For \( x, \bar{x} \in \Omega \) and \( v \in \mathbb{R}^3 \), we assume \eqref{assume_x}. Then, there exists \( \tilde{x} \in \Omega \) defined in \eqref{def_tildex} such that \( (x - \tilde{x}) \cdot v = 0 \) and \( \tilde{x} - \bar{x} \) is parallel to \( v \). Since the singularity appears only in the variation between \( x \) and \( \tx \), and not between \( \tx \) and \( \bar{x} \), we accordingly divide the analysis into singular and nonsingular parts. If assumption \eqref{assume_x} does not hold, we may simply set \( x = \tilde{x} \), so that only the nonsingular part needs to be considered. Before we begin, note that the singularity forms \( \mathcal{T}_{sp}(x, \tilde{x}, v; t) \) defined in Definition \ref{def:sing_2} will be used.  \\

  \begin{lemma}[Singular part of spatial variation]
    \label{lem:tra_sin_x} Let $t > 0$, $x, \tilde{x} \in \Omega$, and $v \in \mathbb{R}^3$ be such that $(x - \tilde{x}) \cdot v = 0$.
        
    \noindent (1) For $0 \leq s \leq t$, we obtain
		\begin{align*}
        |X(s;t,x,v ) - X(s;t, \tx, v )| \leq |x- \tx| \left(1 + |v |(t-s) + 
		|v |^{2}(t-s)
		\mathcal{T}_{sp}(x, \tx, v;t) \right).
        \end{align*} 
        
     \noindent   (2) When $\min\{t_1(t,x,v), t_1(t,\tx,v)\} \geq 0$, we obtain
        \begin{align*}        
		|V(s;t,x,v) - V(s;t, \tx, v)| \leq |x- \tx| \left(|v | + 
		|v |^{2}\mathcal{T}_{sp}(x, \tx, v;t)\right)
        \end{align*} 
        for $s \in [0,\min\{t_1(t,x,v), t_1(t,\tx,v)\}] \cup [\max\{t_1(t,x,v), t_1(t,\tx,v)\}, t]$. The above inequality also holds for $0 \leq s \leq t$ when $\min\{t_1(t,x,v), t_1(t,\tx,v)\} =-\infty$.
        \end{lemma}
        \begin{proof}
      To prove (1), we divide the cases into (a), (b), and (c).

       \textbf{(a)} Assume $0 \leq \min\{t_1(t,x,v), t_1(t,\tx,v)\} \leq  \max\{t_1(t,x,v), t_1(t,\tx,v)\} \leq t$. \\
       For $s \in [0,\min\{t_1(t,x,v), t_1(t,\tx,v)\}]$, we have \eqref{est:X/x}. For $s \in [\max\{t_1(t,x,v), t_1(t,\tx,v),t]$, we have $|X(s;t,x,v)-X(s;t,\tx,v)|=|x-\tx|$. For $s \in [\min\{t_1(t,x,v), t_1(t,\tx,v)\},\max\{t_1(t,x,v), t_1(t,\tx,v)\}]$, we have
       \begin{align} \label{X-X_min_max}
       \begin{split}
           &|X(s;t,x,v)-X(s;t,\tx,v)| \\
           &\leq |X(s;t,x,v)-\xb(x,v)|+|X(s;t,\tx,v)-\xb(\tx,v)|+|\xb(x,v)-\xb(\tx,v)| \\
           &\leq |v||t_1(t,x,v)-s| + |v||t_1(t,\tx,v)-s|+|v||\tb(x,v)-\tb(\tx,v)|+|x-\tx| \\
           &\leq |x-\tx|+3|v||\tb(x,v)-\tb(\tx,v)| \\
           &\leq |x-\tx|+3|v|\fint_{0}^{1} \frac{1}{\mathfrak{S}_{sp}(\tau; x, \tx, v)}   
		\mathbf{1}_{ \Big\{ \substack{
				\tb(\X(\tau), v) < \infty, \ 0\leq \tau \leq 1  
				\\ 
				\min_{0\leq \tau \leq 1  }\tb(\X(\tau), v) \leq t }  \Big\} } d\tau
       \end{split}
       \end{align}
       by using \eqref{tb-tb_x}.

      \textbf{(b)} Assume $\min\{t_1(t,x,v), t_1(t,\tx,v)\} \leq 0 \leq  \max\{t_1(t,x,v), t_1(t,\tx,v)\} \leq t$. \\
      For $s \in [\max\{t_1(t,x,v), t_1(t,\tx,v)\} ,t]$, we have $|X(s;t,x,v)-X(s;t,\tx,v)|=|x-\tx|$. Next, we let $s \in [0, \max\{t_1(t,x,v), t_1(t,\tx,v)\}]$.
      
      When $\min\{t_1(t,x,v), t_1(t,\tx,v)\} \neq -\infty$, we can use the same arguments as \eqref{X-X_min_max}.
      
      When $t_1(t,x,v) \neq -\infty,\, t_1(t,\tx,v) = -\infty$, there exists $\tau_{-} \in (0,1)$ such that
       \begin{align*}
           |X(s;t,x,v)-X(s;t,\tx,v)| =|X(s;t,x,v)-X(s;t,\X(\tau_{-}),v)|.
       \end{align*}
       By replacing $\tx$ by $\X(\tau_{-})$ in \eqref{X-X_min_max}, we obtain
       \begin{align*}
           &|X(s;t,x,v)-X(s;t,\X(\tau_{-}),v)| \\
           &\leq 
           |x-\tx|+3|v|\fint_{\tau_{-}}^{1} \frac{1}{\mathfrak{S}_{sp}(\tau; x, \tx, v)}     
		\mathbf{1}_{ \Big\{ \substack{
				\tb(\X(\tau), v) < \infty, \ 0 \leq \tau_{-}\leq \tau \leq   1
				\\ 
				\min_{\tau_{-}\leq \tau \leq 1  }\tb(\X(\tau), v) \leq t }  \Big\} } d\tau.
       \end{align*}
       We can use the same arguments above, when $t_1(t,x,v) = -\infty,\, t_1(t,\tx,v) \neq -\infty$. 

      When $t_1(t,x,v)=t_1(t,\tx,v)= -\infty$, we have $|X(s;t,x,v)-X(s;t,\tx,v)|=|x-\tx|$ for $s \in [0,t]$.

      \textbf{(c)} Assume $\max\{t_1(t,x,v), t_1(t,\tx,v)\} \leq 0$. Then $|X(s;t,x,v)-X(s;t,\tx,v)|=|x-\tx|$ for $s \in [0,t]$. \\ 
      
       Next, we prove (2), dividing the cases into (d), (e), and (f).
       
       \textbf{(d)} When $\min\{t_1(t,x,v), t_1(t,\tx,v)\} \geq 0$, we have \eqref{est:V/x} for $s \in [0,\min\{t_1(t,x,v), t_1(t,\tx,v)\}]$ and $|V(s;t,x,v)-V(s;t,\tx,v)|=0$ for $s \in [\max\{t_1(t,x,v), t_1(t,\tx,v)\},t]$. 

      \textbf{(e)} Assume $t_1(t,x,v) \neq -\infty,\, t_1(t,\tx,v) = -\infty$.
   
      When $t_1(t,x,v) \leq 0$, we have $|V(s;t,x,v)-V(s;t,\tx,v)|=0$ for $s \in [0,t]$.

       When $t_1(t,x,v) \geq 0$, there exists $\tau_{-} \in (0,1)$ such that
       \begin{align*}
           |V(s;t,x,v)-V(s;t,\tx,v)| =|V(s;t,x,v)-V(s;t,\X(\tau_{-}),v)|
       \end{align*}
       and $t_1(t,\X(\tau_{-}),v) \leq t_1(t,x,v)$. For $s \leq t_1(t,\X(\tau_{-}),v)$, we obtain
        \begin{align*}
           &|V(s;t,x,v)-V(s;t,\X(\tau_{-}),v)| \\
            &\leq |x-\tx|\left(|v|+|v|^2\fint_{\tau_{-}}^{1} \frac{1}{\mathfrak{S}_{sp}(\tau; x, \tx, v)}     
		\mathbf{1}_{ \Big\{ \substack{
				\tb(\X(\tau), v) < \infty, \,0\leq \tau_{-}\leq \tau \leq   1
				\\ 
				\min_{\tau_{-}\leq \tau \leq 1  }\tb(\X(\tau), v) \leq t }  \Big\} } d\tau\right)
       \end{align*}
       by replacing $\tx$ with $\X(\tau_{-})$ in \eqref{est:V/x}.
      For \( t_1(t,\X(\tau_{-}),v) \leq s \leq t_1(t,x,v) \), the above inequality still holds because the quantity $|V(s; t, x, v) - V(s; t, \X(\tau_{-}), v)|$
        remains unchanged throughout this interval. Indeed, both velocities remain constant: 
        \[
        V(s; t, \X(\tau_{-}), v) = v, \quad \text{and} \quad V(s; t, x, v) = R_{\xb(x,v)} v,
        \]
        just as in the case \( s \leq t_1(t, \X(\tau_{-}), v) \).
         For $t_1(t,x,v) \leq s$, $|V(s;t,x,v)-V(s;t,\X(\tau_{-}),v)|=0$.
        We can use the same arguments above, when $t_1(t,x,v) = -\infty,\, t_1(t,\tx,v) \neq -\infty$.

        \textbf{(f)} Assume $t_1(t,x,v) =t_1(t,\tx,v) = -\infty$. Then $|V(s;t,x,v)-V(s;t,\tx,v)|=0$ for $s \in [0,t]$.
        \end{proof}
        
In the nonsingular part, the particles share the same spatial trajectory, differing only in their time parameters. This fact provides intuitive insight into the following lemma.
        
    \begin{lemma}[Nonsingular part of spatial variation] \label{lem:tra_nsin_x} Let \( t > 0 \), \( \tilde{x}, \bar{x} \in \Omega \), and \( v \in \mathbb{R}^3 \) be given such that \( \tilde{x} - \bar{x} \) is either parallel or anti-parallel to \( v \).

    \noindent (1) When $  
			\tb(\tx, v) < \infty$ and $\tb(\bx,v) < \infty$, we obtain
        \begin{align*}
            |\tb(\tx,v)-\tb(\bx,v)| 
        = |\tx-\bx|\frac{1}{|v|}.
        \end{align*}
        
    \noindent (2) For $0 \leq s \leq t$, we obtain
		\begin{align*} 
        |X(s;t,\tx,v ) - X(s;t, \bx, v )| \leq |\tx-\bx|.
        \end{align*}
        
    \noindent (3) When $\min\{t_1(t,x,v), t_1(t,\tx,v)\} \geq 0$, we obtain
        \begin{align*}
       |V(s;t,\tx,v) - V(s;t, \bx, v)|=0.
        \end{align*}
        for $s \in [0,\min\{t_1(t,\tx,v), t_1(t,\bx,v)\}] \cup [\max\{t_1(t,\tx,v), t_1(t,\bx,v)\},t]$. The above inequality also holds for $0 \leq s \leq t$ when $t_1(t,\tx,v)=t_1(t,\bx,v) =-\infty$
        \end{lemma}
        \begin{proof}
        Since $\xb(\tx,v)=\xb(\bx,v)$, $\tx-|v|\tb(\tx,v)=\bx-|v|\tb(\bx,v)$ holds, and thus we obtain (1). For $s \in [0,\min\{t_1(t,\bx,v), t_1(t,\bx,v)\}] \cup [\max\{t_1(t,\bx,v), t_1(t,\bx,v)\}, t]$, we obtain
        \begin{align} \label{txbx_X,V}
            |X(s;t,\tx,v)-X(s;t,\bx,v)|=|\tx-\bx|, \quad  |V(s;t,\tx,v)-V(s;t,\bx,v)|=0.
        \end{align}
        For $s \in [\min\{t_1(t,\bx,v), t_1(t,\bx,v)\}, \max\{t_1(t,\bx,v), t_1(t,\bx,v)\}]$, we obtain
        \begin{align*}
            &|X(s;t,\tx,v)-X(s;t,\bx,v)| \\
            &\leq |X(s;t,\tx,v)-\xb(\tx,v)|
            +|X(s;t,\bx,v)-\xb(\bx,v)|+|\xb(\tx,v)-\xb(\bx,v)| \\
            &\leq |\tx-\bx|+3|v||\tb(\tx,v)-\tb(\bx,v)| =4|\tx-\bx|.
        \end{align*}
        The above inequality also holds for $s \in [0,\max\{t_1(t,\bx,v), t_1(t,\bx,v)\}]$
        when $\min\{t_1(t,\bx,v), t_1(t,\bx,v)\} \leq 0$ and $\min\{t_1(t,\bx,v), t_1(t,\bx,v)\} \neq -\infty$. If $\min\{t_1(t,\bx,v), t_1(t,\bx,v)\} = -\infty$, $t_1(t,\bx,v)=t_1(t,\bx,v)=-\infty$ holds, we have \eqref{txbx_X,V} for $s \in [0,t]$. Therefore, we get (2) and (3).
        \end{proof}

 To estimate the difference in solution via Duhamel's formula \eqref{f_expan}, we decompose it into \eqref{basic f-f}-\eqref{basic f-f4}.  For the initial term \eqref{basic f-f1}, we prove Lemma~\ref{lem:s=0_spec_x}, which is used later in Lemmas \ref{pro:X<1} and~\ref{pro:V<1}.

    \begin{lemma}\label{lem:s=0_spec_x}
           Let $t>0,\,x, \bx \in \Omega$ and $v \in \mathbb{R}^3$. Then, we obtain
        \begin{align*} &|f(0,X(0;t,x,v),V(0;t,x,v))-f(0,X(0;t,\bx,v),V(0;t,\bx,v))|\\
         &\lesssim |x-\bx|\left( \mathbf{A}_{\frac{1}{2}}(f_0)+\|w_0 f_0\|_{\infty} \right)(1+t)\left(1+ 
		|v |+(|v|+|v|^2)\mathcal{T}_{sp}(x, \tx, v;t)\right).  
        \end{align*} 
    \end{lemma}
    \begin{proof}
        When \eqref{assume_x} holds for given \( x, \bar{x} \in \Omega \) and \( v \in \mathbb{R}^3 \), there exists \( \tilde{x}(x, \bar{x}, v) \in \Omega \) as defined in \eqref{def_tildex}. If \eqref{assume_x} does not hold, we set \( \tilde{x} = x \). We split
    \begin{align}
        &|f(0,X(0;t,x,v),V(0;t,x,v))-f(0,X(0;t,\bx,v),V(0;t,\bx,v))| \notag  \\
      &\leq 
      |f(0,X(0;t,x,v),V(0;t,x,v))-f(0,X(0;t,\tx,v),V(0;t,\tx,v))|\label{s=0_xtx}\\
      &\quad+
      |f(0,X(0;t,\tx,v),V(0;t,\tx,v))-f(0,X(0;t,\bx,v),V(0;t,\bx,v))|.\label{s=0_txbx}
    \end{align}
    
    We first consider \eqref{s=0_xtx} and split
  \begin{align}
      \eqref{s=0_xtx}
      &\leq 
      |f(0,X(0;t,x,v),V(0;t,x,v))-f(0,X(0;t,\tx,v),V(0;t,x,v))|\label{s=0_Xx}\\
      &\quad + 
    |f(0,X(0;t,\tx,v),V(0;t,x,v))-f(0,X(0;t,\tx,v),V(0;t,\tx,v))|.\label{s=0_Vx}
  \end{align}
  By Lemma \ref{lem:tra_sin_x}-(1), we obtain
  \begin{align*}
      \eqref{s=0_Xx}
      &\leq |X(0;t,x,v)-X(0;t,\tx,v)|\left( \mathbf{A}_{\frac{1}{2}}(f_0)+\|w_0 f_0\|_{\infty} \right)\\
      &\leq |x-\tx|\left( \mathbf{A}_{\frac{1}{2}}(f_0)+2\|w_0 f_0\|_{\infty} \right)\left(1+ 
		|v |t+|v|^2t\mathcal{T}_{sp}(x, \tx, v;t)\right).  
  \end{align*}
  We divide cases (a) and (b) for \eqref{s=0_Vx}. 
  
  \textbf{(a)} Assume $0 \leq \min\{t_1(t,x,v), t_1(t,\tx,v)\}$ or $0 \geq \max\{t_1(t,x,v), t_1(t,\tx,v)\}$ or \\
  $\min\{t_1(t,x,v), t_1(t,\tx,v)\}=-\infty$. By Lemma \ref{lem:tra_sin_x}-(2), we obtain 
  \begin{align*}
      \eqref{s=0_Vx}  &\leq |V(0;t,x,v)-V(0;t,\tx,v)|\left( \mathbf{A}_{\frac{1}{2}}(f_0)+2\|w_0 f_0\|_{\infty} \right)\\
      &\leq |x-\tx|\left( \mathbf{A}_{\frac{1}{2}}(f_0)+2\|w_0 f_0\|_{\infty} \right)\left(1+ 
		|v |+|v|^2\mathcal{T}_{sp}(x, \tx, v;t)\right). 
  \end{align*}
  
    \textbf{(b)} Assume $-\infty < t_1(t,\tx,v) \leq 0 \leq t_1(t,x,v)$. Using the specular boundary condition, we obtain
  \begin{align*}
      \eqref{s=0_Vx} &\leq |f(0,X(0;t,\tx,v),V(0;t,x,v))-f(0,\xb(x,v),R_{\xb(x,v)}v)| \\
      &\quad +|f(0,\xb(x,v),v)-f(0,X(0;t,\tx,v),V(0;t,\tx,v))| \\
      &\leq 2|X(0;t,\tx,v)-\xb(x,v)|\left( \mathbf{A}_{\frac{1}{2}}(f_0)+2\|w_0 f_0\|_{\infty} \right).
  \end{align*}
  Since 
  \begin{align*}
      |X(0;t,\tx,v)-\xb(x,v)|  
      &\leq |X(0;t,\tx,v)-X(0;t,x,v)| +|X(0;t,x,v)-\xb(x,v)| \\
      &=|X(0;t,\tx,v)-X(0;t,x,v)|+|v|t_1(t,x,v)\\
      &\leq  |X(0;t,\tx,v)-X(0;t,x,v)|+|v||t_1(t,x,v)-t_1(t,\tx,v)|\\
      &= |X(0;t,\tx,v)-X(0;t,x,v)|+|v||\tb(x,v)-\tb(\tx,v)| \\
      &\leq |x- \tx| \left(1 + |v |t + 
		(|v|+|v |^{2}t )
		\mathcal{T}_{sp}(x, \tx, v;t) \right),
  \end{align*}
 we obtain 
   \begin{align*}
        \eqref{s=0_Vx} \lesssim |x-\tx|(1+t)\left( \mathbf{A}_{\frac{1}{2}}(f_0)+\|w_0 f_0\|_{\infty} \right)\left(1+ |v|+(|v|+|v|^2)\mathcal{T}_{sp}(x, \tx, v;t)\right).
   \end{align*}
We can use the same arguments as in $-\infty<t_1(t,x,v) \leq 0 \leq t_1(t,\tx,v)$. 

Combing the upper bound of \eqref{s=0_Xx} and \eqref{s=0_Vx},  we obtain
\begin{align*}
    \eqref{s=0_xtx} \lesssim  |x-\tx|(1+t)\left( \mathbf{A}_{\frac{1}{2}}(f_0)+\|w_0 f_0\|_{\infty} \right)\left(1+ |v|+(|v|+|v|^2)\mathcal{T}_{sp}(x, \tx, v;t)\right).
\end{align*}
Moreover, we can estimate \eqref{s=0_txbx} by applying Lemma \ref{lem:tra_nsin_x} instead of Lemma \ref{lem:tra_sin_x}.
  
    \end{proof}

\subsection{Velocity variation}

In Section 3.2, we analyze trajectory estimates under velocity variation based on Lemma \ref{frac sim S}-(2). For \( x \in \Omega \) and \( v, \bar{v}, \zeta \in \mathbb{R}^3 \), we assume \eqref{assume_v}. Then, there exists \( \tilde{v} \in \mathbb{R}^3 \), defined in \eqref{def_tildev}, such that \( |v + \zeta| = |\tilde{v} + \zeta| \) and the vectors \( \tilde{v} + \zeta \) and \( \bar{v} + \zeta \) have the same direction. The singularity comes only from the variation between \( v + \zeta \) and \( \tilde{v} + \zeta \), so we divide the analysis into singular and non-singular parts. If assumption \eqref{assume_v} does not hold, we simply set \( v = \tilde{v} \), and only the non-singular part is considered. Before proceeding, we note that the singularity forms \( \mathcal{T}_{vel}(x, v, \tilde{v}, \zeta; t) \) in Definition \ref{def:sing_2} will be used.  \\

\begin{lemma}[Singular part of velocity variation]
\label{lem:tra_sin_v} Let $t>0,\,x \in \Omega$ and $v, \tv, \z \in \mathbb{R}^3$ be such that $|v+\z|=|\tv+\z|$.  

\noindent (1) For $0 \leq s \leq t$, we obtain
    \begin{align*}
    &|X(s;t,x,v+\z ) - X(s;t, x, \tv+\z )|\\
    &\leq |v -\tv| \left( t-s+ |v+\z|(t-s)^2+|v+\z|^2(t-s)\mathcal{T}_{vel}(x, v, \tv, \zeta;t)\right).
    \end{align*}

\noindent (2) When $\min\{t_1(t,x,v+\z), t_1(t,x,\tv+\z)\} \geq 0$, we obtain
    \begin{align*} 
    &|V(s;t,x,v+\z) - V(s;t, x, \tv+\z)| \\
    &\leq |v- \tv| \left(1+
    |v+\zeta| (t-s) +
    |v+\zeta|^{2} \mathcal{T}_{vel}(x, v, \tv, \zeta;t)\right)
    \end{align*} 
      for $s \in [0,\min\{t_1(t,x,v+\z), t_1(t,x,v+\z)\}] \cup [\max\{t_1(t,x,v+\z), t_1(t,x,\tv+\z)\}, t]$. The above inequality also holds for $0 \leq s \leq t$ when $\min\{t_1(t,x,v+\z), t_1(t,x,\tv+\z)\} =-\infty$.
\end{lemma}
\begin{proof}
    Using the same arguments as in Lemma \ref{lem:tra_sin_x}, we derive (1) and (2) from \eqref{est:V/v} and \eqref{est:X/v} in Lemma \ref{frac sim S}.
\end{proof}

In the nonsingular case for velocity, the particles also share the same spatial trajectory, just as in the nonsingular case for space.

\begin{lemma}(Nonsingular part of velocity variation)\label{lem:tra_non_v}
    Let $t > 0$, $x \in \Omega$, and $\bv, \tv, \z \in \mathbb{R}^3$ be such that $\tv + \z$ and $\bv + \z$ have the same direction.
    
    \noindent (1) When $  
    \tb(x, \tv+\z) < \infty$ and $\tb(x,\bv+\z) < \infty$, we obtain
    \begin{align*}
        |\tb(x,\tv+\z)-\tb(x,\bv+\z)| 
        \leq \min\left\{\frac{\tb(x,\tv+\z)}{|\tv+\z|},\frac{\tb(x,\bv+\z)}{|\bv+\z|}\right\}|\tv-\bv|.
    \end{align*} 

    \noindent (2) For $0 \leq s \leq t$, we obtain
        \begin{align*}
        |X(s;t,x,\tv+\z ) - X(s;t, x, \bv+\z )|\leq |\tv -\bv|(t-s). 
        \end{align*} 

        \noindent (3) When $\min\{t_1(t,x,\tv+\z), t_1(t,x,\bv+\z)\} \geq 0$, we obtain
        \begin{align*}
            |V(s;t,x,\tv+\z ) - V(s;t, x, \bv+\z )|= |\tv -\bv|
        \end{align*}
        for $s \in [0,\min\{t_1(t,x,\tv+\z), t_1(t,x,\bv+\z)\}] \cup [\max\{t_1(t,x,\tv+\z), t_1(t,x,\bv+\z)\},t]$. The above inequality also holds for $0 \leq s \leq t$ when $t_1(t,x,\tv+\z)=t_1(t,x,\bv+\z) =-\infty$.
\end{lemma}
\begin{proof}
     Since $\xb(x,\tv+\z)=\xb(x,\bv+\z)$, $x-|\tv+\z|\tb(x,\tv+\z)=x-|\bv+\z|\tb(x,\bv+\z)$ holds, and thus we obtain 
     \begin{align*}
          |\tb(x,\tv+\z)-\tb(x,\bv+\z)| 
          =\left| \frac{|\bv+\z|}{|\tv+\z|}-1 \right|\tb(x,\tv+\z) 
          \leq \frac{\tb(x,\tv+\z)}{|\tv+\z|}|\tv-\bv|.
     \end{align*}
     Moreover, we can derive (2) and (3) using the same arguments as in Lemma \ref{lem:tra_nsin_x}.
\end{proof}

To estimate \eqref{basic f-f1}, which corresponds to the difference in the initial data, we divide the proof of Lemma \ref{lem:s=0_spec_v} into two cases \( |v - \bar{v}| \geq \frac{1}{2}|\bar{v} + \zeta| \) and \( |v - \bar{v}| \leq \frac{1}{2}|\bar{v} + \zeta| \). This case division allows us to take the minimum between \( |v + \zeta|^{-1} \) and \( |\bar{v} + \zeta|^{-1} \) in front of the singular factor \( \mathcal{T}_{\text{vel}}(x, v, \bar{v}, \zeta; t) \).

\begin{lemma} \label{lem:s=0_spec_v}
Let $t>0,\,x \in \Omega$ and $v,\bv,\z \in \mathbb{R}^3$. Then, we obtain
    \begin{align*}
    \begin{split}
    &|f(0,X(0;t,x,v+\z),V(0;t,x,v+\z))-f(0,X(0;t,x,\bv+\z),V(0;t,x,\bv+\z))|\\
    &\lesssim |v-\bv|(1+t+t^2)\left( \mathbf{A}_{\frac{1}{2}}(f_0)+\|w_0 f_0\|_{\infty}\right) \\
    &\quad \times \left( \frac{1}{|v+\z|}+\frac{1}{|\bv+\z|}+|v+\z|+\left(\min\left\{\frac{1}{|v+\z|}, \frac{1}{|\bv+\z|}\right\}+1\right)|v+\z|^2\mathcal{T}_{vel}(x, v, \tv, \zeta;t)\right).
    \end{split}
\end{align*} 
\end{lemma}
\begin{proof}
For each case, we obtain 
\begin{align} \label{v-bv>1/2v}
    \frac{1}{|v - \bv|} \leq \frac{2}{|\bv + \z|}
    \quad \text{if} \quad |v - \bv| \geq \frac{1}{2}|\bv + \z|,
\end{align}
and
\begin{align} \label{min_v,zeta}
    \frac{1}{|v + \z|} \leq \frac{2}{|\bv + \z|}
    \quad \text{if} \quad |v - \bv| \leq \frac{1}{2}|\bv + \z|.
\end{align}
To justify \eqref{min_v,zeta}, we apply the triangle inequality:
\begin{align*}
    |v + \z| \geq |\bv + \z| - |v - \bv| \geq \frac{1}{2}|\bv + \z|.
\end{align*}

In the case of $|v-\bv| \geq |\bv+\z|/2$, we obtain 
    \begin{align*} 
        &|f(0,X(0;t,x,v+\z),V(0;t,x,v+\z))-f(0,X(0;t,x,\bv+\z),V(0;t,x,\bv+\z))|\\
         &\leq 2\|w_0 f_0\|_{\infty}
         \leq 4\|w_0 f_0\|_{\infty}|v-\bv|\frac{1}{|\bv+\z|}
    \end{align*}
using \eqref{v-bv>1/2v}.

 We can use the same arguments as in Lemma \ref{lem:s=0_spec_x}, and apply Lemma \ref{lem:tra_sin_v} and Lemma \ref{lem:tra_non_v}.
Then, we obtain 
\begin{align*}
    &|f(0,X(0;t,x,v+\z),V(0;t,x,v+\z))-f(0,X(0;t,x,\bv+\z),V(0;t,x,\bv+\z))|\\
    &\lesssim |v-\bv|(1+t+t^2)\left( \mathbf{A}_{\frac{1}{2}}(f_0)+\|w_0 f_0\|_{\infty}\right) 
    \left(1 + |v+\z|+\left(|v+\z|+|v+\z|^2\right)\mathcal{T}_{vel}(x, v, \tv, \zeta;t)\right).
\end{align*} 
In the case \( |v - \bv| \leq |\bv + \z|/2 \), using \eqref{min_v,zeta}, we can take the minimum between \( |v + \z|^{-1} \) and \( |\bv + \z|^{-1} \). From above inequality, we obtain
\begin{align*}
    &|f(0, X(0;t,x,v+\z), V(0;t,x,v+\z)) - f(0, X(0;t,x,\bv+\z), V(0;t,x,\bv+\z))| \\
    &\lesssim |v - \bv|(1 + t + t^2)\left( \mathbf{A}_{\frac{1}{2}}(f_0) + \|w_0 f_0\|_{\infty} \right) \\
    &\quad \times \left( 1 + |v + \z| + \left( \min\left\{ \frac{1}{|v + \z|}, \frac{1}{|\bv + \z|} \right\} + 1 \right) |v + \z|^2 \mathcal{T}_{\text{vel}}(x, v, \bv, \z; t) \right).
\end{align*}

Finally, combining the two cases above yields the desired estimate.
\end{proof}

\subsection{Averaging specular singularity}
In Section 3.3, we recall the averaging lemma for singularities arising from the trajectory estimates, as developed in \cite{CD2023}. In Corollary \ref{cor:ave}, we include a constraint of the form \( \mathbf{1}_{\{d(x, \Omega) \lesssim \langle v \rangle\}} \) in the inequality. This condition is helpful when integrating the singularities later in Lemma \ref{lem:int_sing}. \\

Key idea of singularity averaging in \cite{CD2023} to apply Sepcular singularities in Definition \ref{def:sing_1} with shifted position and velocity in Definition \ref{def_tilde} to use geometric effect of uniform concavity. In particular, we get the following ODE (we refer Lemma 5.4 in \cite{CD2023})
	\Be \notag
	\begin{split}
		\frac{d \mathfrak{S}_{sp}(\tau;x,\tilde{x}, v)}{d\tau} 
		&
		\geq  
		\frac{1}{\mathfrak{S}_{sp}(\tau; x, \tilde{x}, v)} \frac{ \theta_\O   |\dot{\X}|^{2} }{ |\dot{\X} \cdot \nabla \xi (\xb(\X(\tau), v))| }  \big(  |v|^2 + \mathfrak{S}^{2}_{sp}(\tau; x, \tilde{x}, v)  \big)  .
	\end{split}
	\Ee	 
	 {
    We define $G_{sp}(\tau; x, \tx, v) := |v|^{2} + \mathfrak{S}^{2}_{sp}(\tau; x, \tx, v)$ to obtain
	\begin{equation} \notag \label{ode_sp}
		\begin{split}
			\frac{d}{dt} G_{sp}(\tau; x, \tx, v) \geq \frac{2\theta_{\O}|\dot{\X}|^{2} }{|\dot{\X}\cdot\nabla\xi(\xb( \X(\tau), v))|  } G_{sp}(\tau; x, \tx, v).
		\end{split}
	\end{equation}
	Since $\mathfrak{S}_{sp}(\tau_{-})=0$, we get an upper bound of $\mathfrak{S}_{sp}(\tau; x, \tx, v)$ by solving above ODE : 
	\begin{equation*} 
		\begin{split}
			\frac{1}{\mathfrak{S}_{sp}(\tau; x, \tx, v)} 
			&\leq \frac{1}{|v|}\Big( e^{\int_{\tau_{-}}^{\tau} \frac{2\theta_{\O}|\dot{\X}|^{2}}{ |\dot{\X}\cdot\nabla\xi(\xb( \X(s), v))| } ds    } -1 \Big)^{-\frac{1}{2}} \\
			&\leq \frac{1}{|v|}\Bigg[ \int_{\tau_{-}}^{\tau} \frac{2\theta_{\O}|\dot{\X}|^{2}}{ \max_{\tau_{-}\leq s\leq \tau}|\dot{\X}\cdot\nabla\xi(\xb( \X(s), v))| }     \Bigg]^{-\frac{1}{2}} \\
			&\lesssim_{\O} \frac{1}{|v||\dot{\X}|} \sqrt{\frac{ |\dot{\X} \cdot\nabla\xi(\xb( \X(\tau_{-} (x, \bar x, v)  ), v))| }{ (\tau-\tau_{-}  (x, \bar x, v) ) }}.
		\end{split}
	\end{equation*}
	Here, we have used the fact $0\leq |{\dot{\X}}\cdot\nabla\xi(\xb( \X(\tau), v))| \leq C_{\O} |{\dot{\X}}\cdot\nabla\xi(\xb( \X(\tau_{-}), v))|$ for $\tau\in[\tau_{-}, \tau_{0}]$ from (4.17) of \cite{CD2023} which comes from uniformly convex geometric property.
	Integrating in $\tau$, we obtain the following Lemma. (Result for $\mathfrak{S}_{vel}$ is similar with some modification. We refer \cite{CD2023}.) 
    }

\begin{lemma}[Proposition 5.6 of \cite{CD2023}] \label{prop_avg S} 
	(1) Let $x, \tilde{x} \in \Omega$ and $v \in \mathbb{R}^3$ be such that $(x - \tx) \cdot v = 0$. Assume $\tb(\X(\tau_*), v) < \infty$ for $\tau_{*}\in [\tau_{-}(x, \tx, v), \tau_{+}(x, \tx, v)]$. Then, we obtain 
	\begin{equation*} 
			\int_{\tau_{-}}^{\tau_{*} }  
			\frac{1	}{\mathfrak{S}_{sp}(\tau; x,\tx, v)} d\tau
	         \lesssim
			\frac{\tau_{*} - \tau_{-}  }{|v\cdot\nabla\xi(\xb(\X(\tau_{*}), v)) |}. 
	\end{equation*} 
	\noindent (2) Let $x \in \Omega$ and $v, \tv, \z \in \mathbb{R}^3$ be such that $|v+\z|=|\tv+\z|$. Assume $\tb(x, \V(\tau_*)) < \infty$ for $\tau_{*}\in [\tau_{-}(x, v, \tv, \zeta), \tau_{+}(x, v, \tv, \zeta)]$. Then, we obtain
	\begin{equation*} 
			\int_{\tau_{-}}^{\tau_{*}}  
			\frac{ 1	}{\mathfrak{S}_{vel}(\tau; x,v, \tv, \zeta)} d\tau
			\lesssim  \frac{ \tau_{*} - \tau_{-} }{ |\V(\tau_{*})\cdot \nabla\xi(\xb(x, \V(\tau_{*})))| }	
			\frac{1}{|\V(\tau_*)|} \big( 1 + \min_{\tau} |\V(\tau)|\tb(x, \V(\tau)) \big).
	\end{equation*} 
	(Remind that $|\V(\tau)| = |v+\zeta|$ for all $\tau$. )  
\end{lemma}

	Using Lemma \ref{prop_avg S}, we can control $\mathcal{T}_{sp,vel}$ (and hence, fraction of characteristics) in terms of inverse of grazing angles measure at each $x$, $\tilde{x}$, $v$, and $\tilde{v}$ as the following lemma.

\begin{corollary}\label{cor:ave}
    (1) Let $0 \leq t\leq 1, \,x, \tx \in \Omega$ and $v \in \mathbb{R}^3$ be such that $(x - \tx) \cdot v = 0$ and $|x-\tx| \leq 1$. Then, we obtain 
    \begin{align*}
        \mathcal{T}_{sp}(x, \tx, v;t) \lesssim \frac{1}{|v|}\left(\frac{1}{|\hat{v}\cdot\nabla\xi(\xb(x, v)) |} \mathbf{1}_{\{\tb(x,v) < \infty\}}+ \frac{1}{|\hat{v}\cdot\nabla\xi(\xb(\tx, v)) |}\mathbf{1}_{\{\tb(\tx,v) < \infty\}}\right)\mathbf{1}_{\{d(x,\partial\O) \lesssim  \langle v \rangle \}}.
    \end{align*}
    (2) Let $0 \leq t\leq 1,\;x \in \Omega$ and $v, \tv, \z \in \mathbb{R}^3$ be such that $|v+\z|=|\tv+\z|$. Then, we obtain
    \begin{align*}
        &\mathcal{T}_{vel}(x, v, \tv, \zeta;t) \\
        &\lesssim \left( \frac{1}{|v+\z|^2} \mathbf{1}_{\{d(x,\partial\O) \lesssim 1\}}+ \frac{1}{|v+\z|}\mathbf{1}_{\{1 \lesssim d(x,\partial\O) \lesssim \langle v+\z \rangle \}}\right)\\
        &\quad \times \left(  \frac{1}{|\widehat{v+\z}\cdot\nabla\xi(\xb(x, v+\z)) |}\mathbf{1}_{\{\tb(x,v+\z)<\infty\}}+\frac{1}{|\widehat{\tv+\z}\cdot\nabla\xi(\xb(x, \tv+\z)) |} \mathbf{1}_{\{\tb(x,\tv+\z)<\infty\}}\right).
    \end{align*} 
\end{corollary}
\begin{proof}
If $\tb(\X(\tau), v) \leq t$ for some $\tau \in [0,1]$, then 
\begin{align*}
    d(x,\partial\O) \lesssim d(\X(\tau),\partial\O) +|x-\X(\tau)|+1 \lesssim |v|\tb(\X(\tau), v)+|x-\tx| + 1 \lesssim \langle v \rangle,
\end{align*}
since $|x| \lesssim 1$ for $x \in \partial\O$. We apply Lemma \ref{prop_avg S}-(1) to $\mathcal{T}_{sp}(x, \tx, v;t)$ with $d(x,\partial\O) \lesssim \langle v \rangle$, and then obtain (1).

If $\tb(x,\V(\tau)) \leq t$ for some $\tau \in [0,1]$, then 
\begin{align*}
    d(x,\partial\O) \lesssim |\V(\tau)|\tb(x,\V(\tau)) +1\lesssim \langle v+\z \rangle.
\end{align*}
So, we divide the case into two parts: $d(x,\partial\O) \lesssim 1$ and $1 \lesssim d(x,\partial\O) \lesssim \langle v+\z \rangle$. Using
\begin{align*}
  |\V(\tau)|\tb(x,\V(\tau)) =|x-\xb(x,v)| \lesssim d(x,\partial\O) + 1,  
\end{align*}
we obtain $ |\V(\tau)|\tb(x,\V(\tau)) \lesssim 1$ for $d(x,\partial\O) \lesssim 1$ and $ |\V(\tau)|\tb(x,\V(\tau)) \lesssim \langle v+\z \rangle$ for $1 \lesssim d(x,\partial\O) \lesssim \langle v+\z \rangle$.
We apply Lemma \ref{prop_avg S}-(2) to $\mathcal{T}_{vel}(x, v, \tv, \zeta;t)$ with the above inequalities, and then obtain (2). \\
\end{proof}

\section{Difference estimates of \texorpdfstring{$\Gamma_{gain}(f,f)$ and $\nu(f)$}{}}

In Section 4, we perform difference estimates of $\Gamma_{gain}$ and $\nu(f)$ which will be used in \eqref{basic f-f2} and \eqref{basic f-f4} later. These terms are treated by separating them into spatial and velocity variations, using the trajectory estimates developed in Sections 3.1 and 3.2. The upper bounds of these estimates are ultimately expressed in terms of the seminorm \( \mathfrak{X}(t, \varpi; \epsilon) \) and \( \mathfrak{V}(t, \varpi; \epsilon) \) from Definition \ref{def:iter}, as well as the singular terms \( \mathcal{T}_{sp}(x, \tilde{x}, v; t) \) and \( \mathcal{T}_{vel}(x, v, \tilde{v}, \zeta; t) \).

    \begin{lemma}[Spatial variation of $\Gamma_{gain}(f,f)$ and $\nu(f)$]\label{lem:ga_X}
     Assume that $x, \bx \in \O$ with $|x-\bx|\leq 1$, and $v \in \mathbb{R}^3$. Recall $T^*>0$ from Lemma \ref{lem:loc}. For any $0<t <  \min\{T^*,1\}, \,\varpi>1$, and $\epsilon>0$, the following hold:
    \begin{align*} 
        &\int_0^{t} \left|\Gamma_{gain}(f,f)(s,X(s;t,x,v),V(s;t,x,v))-\Gamma_{gain}(f,f)(s,X(s;t,\bx,v),V(s;t,\bx,v))\right| ds\\
        &\lesssim_{\vartheta_{0}} |x-\bx|\mathcal{P}_2(\|w_0f_0\|_{\infty} ) e^{\varpi \langle v \rangle^{2} t}
        (\mathfrak{X}(t,\varpi;\epsilon)+\mathfrak{V}(t,\varpi;\epsilon)+1) 
        \left(\frac{1}{|v|}+|v|+(1+|v|^2)\mathcal{T}_{sp}(x,\tx,v;t)\right)
    \end{align*} 
    and 
    \begin{align*} 
        &\int_0^{t} \left|\nu(f)(s,X(s;t,x,v),V(s;t,x,v))-\nu(f)(s,X(s;t,\bx,v),V(s;t,\bx,v))\right| ds\\
        &\lesssim |x-\bx|\mathcal{P}_1(\|w_0f_0\|_{\infty} ) e^{\varpi \langle v \rangle^{2} t}
        (\mathfrak{X}(t,\varpi;\epsilon)+1)
        \left(\frac{1}{|v|}+|v|+(1+|v|^2)\mathcal{T}_{sp}(x,\tx,v;t)\right).
    \end{align*} 
   \end{lemma}
     \begin{proof} 
    When \eqref{assume_x} holds for given \( x, \bar{x} \in \Omega \) and \( v \in \mathbb{R}^3 \), there exists \( \tilde{x}(x, \bar{x}, v) \in \Omega \) as defined in \eqref{def_tildex}. If \eqref{assume_x} does not hold, we set \( \tilde{x} = x \). For $x,\tx$, and $\bx$, we split
    \begin{align}
        &\left|\Gamma_{gain}(f,f)(s,X(s;t,x,v),V(s;t,x,v))-\Gamma_{gain}(f,f)(s,X(s;t,\bx,v),V(s;t,\bx,v))\right|
        \label{X,V_sp_0}\\ 
       &\leq
       \left|\Gamma_{gain}(f,f)(s,X(s;t,x,v),V(s;t,x,v))-\Gamma_{gain}(f,f)(s,X(s;t,\tx,v),V(s;t,x,v))\right|
       \label{X_sp} \\ 
       &\quad +\left|\Gamma_{gain}(f,f)(s,X(s;t,\tx,v),V(s;t,x,v))-\Gamma_{gain}(f,f)(s,X(s;t,\tx,v),V(s;t,\tx,v))\right|\label{V_sp}  \\
       &\quad +\left|\Gamma_{gain}(f,f)(s,X(s;t,\tx,v),V(s;t,\tx,v))-\Gamma_{gain}(f,f)(s,X(s;t,\bx,v),V(s;t,\tx,v))\right| \label{X_nsp}\\
        &\quad +\left|\Gamma_{gain}(f,f)(s,X(s;t,\bx,v),V(s;t,\tx,v))-\Gamma_{gain}(f,f)(s,X(s;t,\bx,v),V(s;t,\bx,v))\right|. \label{V_nsp}
    \end{align}
    
     \vspace{3mm}
     Consider \eqref{X_sp}. By \eqref{gamma_x}, we obtain
     \begin{align*}
          \int_0^t (\ref{X_sp}) ds 
         &\lesssim   \|w_0f_0\|_{\infty}\int_0^t |X(s;t,x,v)-X(s;t,\tx,v)|\\
         &\quad \times \int_{\mathbb{R}^3} \frac{e^{-c|u|^2}}{|u|} 
         \frac{\left|f(s,X(s;t,x,v),V(s;t,x,v)+u)-f(s,X(s;t,\tx,v),V(s;t,x,v)+u)\right|}{|X(s;t,x,v)-X(s;t,\tx,v)|} du ds.
     \end{align*}
     By Lemma \ref{lem:tra_sin_x}-(1), we obtain
     \begin{align*}
          |X(s;t,x,v)-X(s;t,\tx,v)| 
          &\leq |x-\tx| \left(1 + |v |(t-s) + 
		|v |^{2}(t-s)
		\mathcal{T}_{sp}(x, \tx, v;t) \right) \\
        &\leq |x-\tx| 
        \left( 1+2|v| + |v|^2 \mathcal{T}_{sp}(x, \tx, v;t)  \right)
        \max \left\{ (t-s)^{\frac{1}{2}}, \frac{1}{\langle v \rangle}    \right\}
     \end{align*}
     since $t<1$.
     Then, we obtain 
     \begin{align*} 
        \int_0^t (\ref{X_sp}) ds 
        &\lesssim   \|w_0f_0\|_{\infty}|x-\tx| \left( 1+|v| + |v|^2 \mathcal{T}_{sp}(x, \tx, v;t)  \right) e^{\varpi \langle v \rangle^2 t}\left(G(x,v;\epsilon)+G(\tx,v;\epsilon)\right) \\
        &\quad \times e^{-\varpi \langle v \rangle^2 t}\left(G(x,v;\epsilon)+G(\tx,v;\epsilon)\right)^{-1}\int_0^t \max \left\{ (t-s)^{\frac{1}{2}}, \frac{1}{\langle v \rangle}    \right\}  \int_{\mathbb{R}^3} \frac{e^{-c|u|^2}}{|u|}  \\
        &\quad \times \frac{\left|f(s,X(s;t,x,v),V(s;t,x,v)+u)-f(s,X(s;t,\tx,v),V(s;t,x,v)+u)\right|}{|X(s;t,x,v)-X(s;t,\tx,v)|} du ds  \\
         &\leq  |x-\tx|\|w_0f_0\|_{\infty}\left(G(x,v;\epsilon)+G(\tx,v;\epsilon)\right)e^{\varpi \langle v \rangle^{2} t}\mathfrak{X}(t,\varpi;\epsilon) \left(1+|v|+|v|^2\mathcal{T}_{sp}(x,\tx,v;t)\right).
     \end{align*}
 As the argument in this paragraph, which leads to the iterative forms $\mathfrak{X}(t,\varpi,\epsilon)$ and $\mathfrak{V}(t,\varpi,\epsilon)$ in Definition~\ref{def:iter}, is used repeatedly in this section, we will omit the details and state only the results from now on. 
 
 Similarly, by applying Lemma \ref{lem:tra_nsin_x}-(2), we obtain
\begin{align*}
    \int_0^t \eqref{X_nsp} ds 
    \lesssim |\tx-\bx|\|w_0f_0\|_{\infty}\left(G(\tx,v;\epsilon)+G(\bx,v;\epsilon)\right)e^{\varpi \langle v \rangle^{2} t}\mathfrak{X}(t,\varpi;\epsilon) \left(1+|v|\right).
\end{align*}

    \vspace{3mm}
   Next, to estimate \eqref{V_sp}, we consider the following three cases: (1), (2), and (3).
    
    \textbf{(1)} Assume $0 \leq \min\{t_1(t,x,v), t_1(t,\tx,v)\} \leq  \max\{t_1(t,x,v), t_1(t,\tx,v)\} \leq t$. \\
    For $\max\{t_1(t,x,v), t_1(t,\tx,v)\} \leq s \leq t$, we obtain 
    \begin{align}  \label{t>max_gamma_x}
        \int_{\max\{t_1(t,x,v), t_1(t,\tx,v)\}}^t (\ref{V_sp}) ds = 0
    \end{align}
    since $V(s;t,x,v)=V(s;t,\tx,v)$. For $0 \leq s \leq \min\{t_1(t,x,v), t_1(t,\tx,v)\}$, we apply \eqref{gamma_v} and Lemma \ref{lem:tra_sin_x}-(2). Then
   \begin{align} \label{dbx>dx}
   \begin{split}
       &\int_0^{\min\{t_1(t,x,v), t_1(t,\tx,v)\}} (\ref{V_sp}) ds \\
        &\lesssim   |x-\bx|\left( \|w_0f_0\|_{\infty} G(\tx,v;\epsilon)e^{\varpi \langle v \rangle^{2} t}\mathfrak{V}(t,\varpi;\epsilon)+\|w_0f_0\|^2_{\infty} \frac{1}{\langle v \rangle} \right) \left(|v|+|v|^2\mathcal{T}_{sp}(x,\tx,v;t)\right).
   \end{split}
   \end{align}
   For $ \min\{t_1(t,x,v), t_1(t,\tx,v)\} \leq s \leq \max\{t_1(t,x,v), t_1(t,\tx,v)\}$, we apply \eqref{gamma_upper} and \eqref{tb-tb_x}. Then
   \begin{align} \label{max>t>min_gamma_x}
   \begin{split}
       &\int_{\min\{t_1(t,x,v), t_1(t,\tx,v)\}}^{\max\{t_1(t,x,v), t_1(t,\tx,v)\}} (\ref{V_sp}) ds \\
       &\lesssim_{\vartheta_{0}} \|w_0f_0\|^2_{\infty} |t_1(t,x,v)-t_1(t,\tx,v)| =\|w_0f_0\|^2_{\infty}|\tb(x,v)-\tb(\tx,v)|\\
       &\lesssim  |x-\bx| \|w_0f_0\|^2_{\infty} \mathcal{T}_{sp}(x,\tx,v;t).
  \end{split}
   \end{align}

 \textbf{(2)} Assume $\min\{t_1(t,x,v), t_1(t,\tx,v)\} \leq 0 \leq \max\{t_1(t,x,v), t_1(t,\tx,v)\} \leq t$. \\
 For $s \in [\max\{t_1(t,x,v), t_1(t,\tx,v),t]$, we have \eqref{t>max_gamma_x}. Next, we divide the cases to estimate \eqref{V_sp} for $s \in [0,\max\{t_1(t,x,v), t_1(t,\tx,v)\}]$. When $t_1(t,x,v) \neq -\infty$ and $t_1(t,\tx,v) \neq -\infty$, we obtain
\begin{align*}
     \int_{0}^{\max\{t_1(t,x,v), t_1(t,\tx,v)\}} \eqref{V_sp} ds \leq
    \int_{^{\min\{t_1(t,x,v), t_1(t,\tx,v)\}}}^{^{\max\{t_1(t,x,v), t_1(t,\tx,v)\}}} \eqref{V_sp} ds
\end{align*}
and use same arguments in \eqref{max>t>min_gamma_x}. When $t_1(t,x,v) \neq -\infty,\,t_1(t,\tx,v) = -\infty$ or $t_1(t,x,v) = -\infty,\,t_1(t,\tx,v) \neq -\infty$, we apply the same arguments \eqref{dbx>dx} since Lemma \ref{lem:tra_sin_x}-(2) holds for  $s \in [0,\max\{t_1(t,x,v), t_1(t,\tx,v)\}]$.

 \textbf{(3)} Assume $ \max\{t_1(t,x,v), t_1(t,\tx,v)\} \leq 0$. Then $ \int_{0}^t \eqref{V_sp} ds = 0$.

Therefore, 
\begin{align*}
    \int_0^t \eqref{V_sp} ds
     &\lesssim_{\vartheta_0} 
       |x-\tx|\|w_0f_0\|_{\infty} 
       G(\tx,v;\epsilon)e^{\varpi \langle v \rangle^{2} t}\mathfrak{V}(t,\varpi;\epsilon) \left(|v|+|v|^2\mathcal{T}_{sp}(x,\tx,v;t)\right) \\
       & \quad + |x-\tx| \|w_0f_0\|^2_{\infty} \left(1+(1+|v|) \mathcal{T}_{sp}(x,\tx,v;t)\right).
\end{align*}

Similarly, Lemma \ref{lem:tra_nsin_x}-(1) and (3) yield 
\begin{align*}
    \int_0^t \eqref{V_nsp} ds\lesssim_{\vartheta_0} |\tx-\bx| \|w_0f_0\|^2_{\infty} \frac{1}{|v|}.
\end{align*}

\vspace{3mm}
By combining the upper bound of $\int_0^t \eqref{X_sp}-\eqref{V_nsp}$, we conclude
   \begin{align*} 
       \int_0^t (\ref{X,V_sp_0}) ds \notag
       &\lesssim_{\vartheta_0}
       |x-\bx|\|w_0f_0\|_{\infty}
       (G(x,v;\epsilon)+G(\tx,v;\epsilon)+G(\bx,v;\epsilon)) 
       e^{\varpi \langle v \rangle^{2} t} \\
       &\quad \times \left(\mathfrak{X}(t,\varpi;\epsilon) +\mathfrak{V}(t,\varpi;\epsilon) \right)\left(1+|v|+|v|^2\mathcal{T}_{sp}(x,\tx,v;t)\right) \\
       & \quad + |x-\bx| \|w_0f_0\|^2_{\infty} \left( \frac{1}{|v|}+1+(1+|v|) \mathcal{T}_{sp}(x,\tx,v;t)\right).
   \end{align*}
  Finally, we apply
 \begin{align} \label{G(x,bx,v,v)_up}
     G(z,v;\epsilon)=\ln\left(1+\frac{1}{|v|}\right)\mathbf{1}_{\{d(z,\partial\O) \leq \epsilon\}} +1 &\leq  \frac{1}{|v|}+1
 \end{align}  
 for $z=x, \tx$ or $\bx$ to the above inequality.

 \vspace{3mm}
 Similarly, we estimate
 \begin{align*}
     \int_0^{t} \left|\nu(f)(s,X(s;t,x,v),V(s;t,x,v))-\nu(f)(s,X(s;t,\bx,v),V(s;t,\bx,v))\right| ds.
 \end{align*}
 However, we apply Lemma \ref{lem:diff_nu} instead of Lemma \ref{lem:est_Gam}. In doing so, estimating the difference of $\nu(f)$ becomes easier than that of $\Gamma_{gain}(f,f)$, and the term $\mathfrak{V}(t,\varpi;\epsilon)$ does not appear.
\end{proof}

Now, we prove the velocity variation of \( \Gamma_{\text{gain}}(f,f) \) and \( \nu(f) \). This case differs from Lemma \ref{lem:ga_X} in that the upper bound for \( \Gamma_{\text{gain}}(f,f) \) involves the term \( \mathbf{A}_{\frac{1}{2}}(f_0) \) in \eqref{dif_gam_A_0}. This term arises from the use of \eqref{pro:H_sub} in the proof, which is applied to reduce the order of the singularity \( |v + \zeta|^{-1} \); see \eqref{est_V_vel_s}. The reason for this reduction is to ensure the integrability of the singularity, as shown in Lemma \ref{lem:int_sing}.

\begin{lemma}[Velocity variation of $\Gamma_{gain}(f,f)$ and $\nu(f)$]\label{lem:ga_V}
 Assume that $x \in \O$ and $v, \bv \in \mathbb{R}^3$ with $|v-\bv|\leq 1$. Recall $T_1>0$ and $\varpi_1>1$ from Proposition \ref{pro:H}. For any $0<t <  \min\{T_1,1\}, \,\varpi>\varpi_1$, and $\epsilon>0$, the following hold:
   \small \begin{align} 
         &\int_0^{t} \left|\Gamma_{gain}(f,f)(s,X(s;t,x,v+\z),V(s;t,x,v+\z))-\Gamma_{gain}(f,f)(s,X(s;t,x,\bv+\z),V(s;t,x,\bv+\z))\right| ds  \notag \\
         &\lesssim_{\vartheta_{0}}
       |v-\bv|\mathcal{P}_2(\|w_0f_0\|_{\infty} )e^{\varpi \langle v+\z \rangle^{2} t} \left(\mathfrak{X}(t,\varpi;\epsilon)+\mathfrak{V}(t,\varpi;\epsilon)+\mathbf{A}_{\frac{1}{2}}(f_0)+1\right)  \label{dif_gam_A_0}\\
       &\quad \times  \left[ \frac{1}{|v+\z|}+\frac{1}{|\bv+\z|}+|v+\z|+\left(\min \left\{  \frac{1}{|v+\z|^{\frac{3}{2}}}, \frac{1}{|\bv+\z|^\frac{3}{2}}\right\}+1\right)|v+\z|^2\mathcal{T}_{vel}(x, v, \tv, \zeta;t)\right] \notag
    \end{align} \normalsize
    and
     \begin{align*} 
         &\int_0^{t} \left|\nu(f)(s,X(s;t,x,v+\z),V(s;t,x,v+\z))-\nu(f)(s,X(s;t,x,\bv+\z),V(s;t,x,\bv+\z))\right| ds\\
         &\lesssim
       |v-\bv|\mathcal{P}_1(\|w_0f_0\|_{\infty} )e^{\varpi \langle v+\z \rangle^{2} t} \left(\mathfrak{X}(t,\varpi;\epsilon)+1\right)  \\
       &\quad \times  \left[ \frac{1}{|v+\z|}+\frac{1}{|\bv+\z|}+|v+\z|+\left(\min \left\{  \frac{1}{|v+\z|}, \frac{1}{|\bv+\z|}\right\}+1\right)|v+\z|^2\mathcal{T}_{vel}(x, v, \tv, \zeta;t)\right].
    \end{align*} 
   \end{lemma}
   \begin{proof}
   When \eqref{assume_v} holds for given \( x \in \Omega \) and \( v, \bv, \z \in \mathbb{R}^3 \), there exists \( \tilde{v}(v, \bar{v}, \z) \in \R^3 \) as defined in \eqref{def_tildev}. If \eqref{assume_v} does not hold, we set \( \tilde{v} = v \). For $v, \tv$, and $\bv$, we split
   \small 
    \begin{align}  
        &\left|\Gamma_{gain}(f,f)(s,X(s;t,x,v+\z),V(s;t,x,v+\z))-\Gamma_{gain}(f,f)(s,X(s;t,x,\bv+\z),V(s;t,x,\bv+\z))\right|\label{X,V_0} \\ 
       &\leq
       \left|\Gamma_{gain}(f,f)(s,X(s;t,x,v+\z),V(s;t,x,v+\z))-\Gamma_{gain}(f,f)(s,X(s;t,x,\tv+\z),V(s;t,x,v+\z))\right|
       \label{X_vel_s} \\
       &\;+\left|\Gamma_{gain}(f,f)(s,X(s;t,x,\tv+\z),V(s;t,x,v+\z))-\Gamma_{gain}(f,f)(s,X(s;t,x,\tv+\z),V(s;t,x,\tv+\z))\right| \label{V_vel_s}  \\
       &\; +\left|\Gamma_{gain}(f,f)(s,X(s;t,x,\tv+\z),V(s;t,x,\tv+\z))-\Gamma_{gain}(f,f)(s,X(s;t,x,\bv+\z),V(s;t,x,\tv+\z))\right| \label{X_vel_ns} \\
       &\;+\left|\Gamma_{gain}(f,f)(s,X(s;t,x,\bv+\z),V(s;t,x,\tv+\z))-\Gamma_{gain}(f,f)(s,X(s;t,x,\bv+\z),V(s;t,x,\bv+\z))\right|.
       \label{V_vel_ns} 
    \end{align} \normalsize
    
   If we assume $|v-\bv| \geq \frac{1}{2}|\bv+\z|$, then
   \begin{align*} 
       \int_0^t (\ref{X,V_0}) ds
       &\lesssim_{\vartheta_{0}} \|w_0f_0\|_{\infty}^2 t \leq \|w_0f_0\|_{\infty}^2 t \frac{2}{|\bv+\z|}|v-\bv|
   \end{align*}
   using \eqref{gamma_upper}. From now on, we assume $|v-\bv| \leq \frac{1}{2}|\bv+\z|$. \\
   
   Next, we first estimate \eqref{X_vel_s} and \eqref{X_vel_ns}, respectively. By \eqref{gamma_x} and Lemma \ref{lem:tra_sin_v}-(1), we obtain
     \begin{align*}
         \int_0^t (\ref{X_vel_s}) ds 
         &\lesssim   |v-\tv|\|w_0f_0\|_{\infty}(G(x,v+\z;\epsilon)+G(x,\tv+\z;\epsilon))e^{\varpi \langle v+\z \rangle^{2} t}\mathfrak{X}(t,\varpi;\epsilon) \\
         &\quad \times \left(1+|v+\z|+|v+\z|^2\mathcal{T}_{vel}(x, v, \tv, \zeta;t)\right).
     \end{align*}
     By \eqref{gamma_x} and Lemma \ref{lem:tra_non_v}-(2), we obtain
     \begin{align*}
          \int_0^t (\ref{X_vel_ns}) ds 
         &\lesssim   |\tv-\bv|\|w_0f_0\|_{\infty}(G(x,\tv+\z;\epsilon)+G(x,\bv+\z;\epsilon))e^{\varpi \langle v+\z \rangle^{2} t}\mathfrak{X}(t,\varpi;\epsilon).\\
     \end{align*}

    To estimate \eqref{V_vel_s}, we consider the following three cases: (1), (2), and (3).
    
    \textbf{(1)} Assume $0 \leq \min\{t_1(t,x,v+\z), t_1(t,x,\tv+\z)\} \leq  \max\{t_1(t,x,v+\z), t_1(t,x,\tv+\z)\} \leq t$. 
    
     For $s \in [0,\min\{t_1(t,x,v+\z), t_1(t,x,\tv+\z)\}]\cup [\max\{ t_1(t,x,v+\z), t_1(t,x,\tv+\z)\}, t]$, we obtain
   \begin{align} \label{V_s_maxt}
   \begin{split}
       &\left(\int_0^{\min\{t_1(t,x,v+\z), t_1(t,x,\tv+\z)\}} +\int_{\max\{t_1(t,x,v+\z), t_1(t,x,\tv+\z)\}}^t \right)(\ref{V_vel_s}) ds \\
        &\lesssim   |v-\tv|\left(\|w_0f_0\|_{\infty} G(x,\tv+\z;\epsilon)e^{\varpi \langle v+\z \rangle^{2} t}\mathfrak{V}(t,\varpi;\epsilon)+\|w_0f_0\|^2_{\infty} \frac{1}{\langle v+\z \rangle} \right) \\
        &\quad \times \left(1+|v+\z|+|v+\z|^2\mathcal{T}_{vel}(x, v, \tv, \zeta;t)\right)
   \end{split}
   \end{align}
   by applying \eqref{gamma_v} and Lemma \ref{lem:tra_sin_v}-(2).

   Next, we consider $s \in [\min\{t_1(t,x,v+\z), t_1(t,x,\tv+\z)\},\max\{t_1(t,x,v+\z), t_1(t,x,\tv+\z)\}]$.
   By \eqref{gamma_v}, we obtain
   \begin{align*}  
       (\ref{V_vel_s})  
       &\lesssim \|w_0f_0\|_{\infty} \int_{\mathbb{R}^3} \frac{1}{|u|}e^{-c|u|^2}\\
       &\quad \times   \left|f(s,X(s;t,x,\tv+\z),V(s;t,x,v+\z) {+u})-f(s,X(s;t,x,\tv+\z),V(s;t,x,\tv+\z) {+u})\right|du\\
       &\quad + \|w_0f_0\|_{\infty}^2 \frac{1}{\langle  v+\z\rangle}|V(s;t,x,v+\z)-V(s;t,x,\tv+\z)|.
   \end{align*}
   By applying \eqref{pro:H_sub}, which is derived from Proposition \ref{pro:H}, with $\beta=\frac{1}{4}$ to the first term on the right-hand side of the above inequality,
   \begin{align}  \label{est_V_vel_s}
   \begin{split}
       (\ref{V_vel_s})&\lesssim_{\vartheta_{0}}\|w_0f_0\|_{\infty}\left(\mathbf{A}_{\frac{1}{2}}(f_0)+\|w_0f_0\|_{\infty}\right)e^{\varpi_1 \langle v +\z \rangle^2 t} |V(s;t,x,v+\z)-V(s;t,x,\tv+\z)|^{\frac{1}{2}}\\
       &\quad + \|w_0f_0\|_{\infty}^2 \frac{1}{\langle v +\z\rangle}|V(s;t,x,v+\z)-V(s;t,x,\tv+\z)|\\
       &\lesssim \mathcal{P}_2(\|w_0f_0\|_{\infty} )e^{\varpi \langle v+\z \rangle^{2} t} \left( \mathbf{A}_{\frac{1}{2}}(f_0)+1\right)|v+\z|^{\frac{1}{2}}.
   \end{split}
   \end{align}
   Then, we obtain
   \begin{align} \label{V_s_minmax}
   \begin{split}
       &\int_{\min\{t_1(t,x,v+\z), t_1(t,x,\tv+\z)\} }^{\max\{t_1(t,x,v+\z), t_1(t,x,\tv+\z)\}} (\ref{V_vel_s}) ds \\
       &\lesssim  \mathcal{P}_2(\|w_0f_0\|_{\infty} )e^{\varpi \langle v+\z \rangle^{2} t} \left( \mathbf{A}_{\frac{1}{2}}(f_0)+1 \right)|v+\z|^{\frac{1}{2}}|\tb(x,v+\z)-\tb(x,\tv+\z)| \\
       &\lesssim |v-\tv|\mathcal{P}_2(\|w_0f_0\|_{\infty} )e^{\varpi \langle v+\z \rangle^{2} t} \left( \mathbf{A}_{\frac{1}{2}}(f_0)+1\right)|v+\z|^{\frac{1}{2}}\mathcal{T}_{vel}(x, v, \tv, \zeta;t). 
  \end{split}
   \end{align}

    \vspace{3mm}
    \textbf{(2)}
    Assume $\min\{t_1(t,x,v+\z), t_1(t,x,\tv+\z)\} \leq 0 \leq \max\{t_1(t,x,v+\z), t_1(t,x,\tv+\z)\} \leq t$.\\
    For $s \in [\max\{t_1(t,x,v+\z), t_1(t,x,\tv+\z)\},t]$, we have \eqref{V_s_maxt}. Next, we divide the cases to estimate \eqref{V_vel_s} for $s \in [0,\max\{t_1(t,x,v+\z), t_1(t,x,\tv+\z)\}]$.

   When $\min\{t_1(t,x,v+\z), t_1(t,x,\tv+\z)\} \neq -\infty$, we have
    \begin{align*}
         \int_0^{\max\{t_1(t,x,v+\z), t_1(t,x,\tv+\z)\}} \eqref{V_vel_s} ds
         \leq \int_{\min\{t_1(t,x,v+\z), t_1(t,x,\tv+\z)\}}^{\max\{t_1(t,x,v+\z), t_1(t,x,\tv+\z)\}} \eqref{V_vel_s} ds,
    \end{align*} 
   and use same arguments in \eqref{V_s_minmax}.

   When $t_1(t,x,v+\z)=-\infty,\,t_1(t,x,\tv+\z)\neq-\infty$ or $t_1(t,x,v+\z)\neq -\infty,\,t_1(t,x,\tv+\z)=-\infty$. we apply the same arguments in \eqref{V_s_maxt} since Lemma \ref{lem:tra_sin_v}-(2) holds.

 \textbf{(3)} When $ \max\{t_1(t,x,v+\z), t_1(t,x,\tv+\z)\} \leq 0$, we use the same arguments as in \eqref{V_s_maxt} for $s \in [0,t]$. Here, we use \(|V(s;t,x,v+\z)-V(s;t,x,\tv+\z)| = |v-\tv|\) instead of Lemma \ref{lem:tra_sin_v}-(2).
  
From (1), (2), and (3), we obtain
\begin{align*} 
       \int_0^t (\ref{V_vel_s}) ds 
       &\lesssim_{\vartheta_{0}} |v-\bv|\mathcal{P}_2(\|w_0f_0\|_{\infty} )
       e^{\varpi \langle v+\z \rangle^{2} t} 
       \left(    G(x,\tv+\z;\epsilon)\mathfrak{V}(t,\varpi;\epsilon)+\mathbf{A}_{\frac{1}{2}}(f_0)+1\right)\\
        &\quad \times \left(1 + (|v+\z|^{\frac{1}{2}}+|v+\z|^2)\mathcal{T}_{vel}(x, v, \tv, \zeta;t)\right).\\
   \end{align*}

    To estimate \eqref{V_vel_ns}, we consider the following three cases: (4), (5), and (6).
   
   \textbf{(4)} Assume $0 \leq \min\{t_1(t,x,\tv+\z), t_1(t,x,\bv+\z)\} \leq  \max\{t_1(t,x,\tv+\z), t_1(t,x,\bv+\z)\} \leq t$. \\
   For $s \in [0,\min\{t_1(t,x,\tv+\z), t_1(t,x,\bv+\z)\}]$ and $s\in[\max\{t_1(t,x,\tv+\z), t_1(t,x,\bv+\z)\},t]$, we use the same arguments as in \eqref{V_s_maxt}. Using \eqref{gamma_upper} and Lemma \ref{lem:tra_non_v}-(1),
   \begin{align}  \label{min_max_vbv}
   \begin{split}
       \int_{\min\{t_1(t,x,\tv+\z), t_1(t,x,\bv+\z)\} }^{\max\{t_1(t,x,\tv+\z), t_1(t,x,\bv+\z)\}} (\ref{V_vel_ns}) ds &\lesssim_{\vartheta_{0}} \|w_0f_0\|_{\infty}^2|\tb(x,\tv+\z)-\tb(x,\bv+\z)| \\
       &\lesssim |v-\bv| \|w_0f_0\|_{\infty}^2\frac{1}{|v+\z|}.
   \end{split}
   \end{align}
   
\textbf{(5)} Assume $\min\{t_1(t,x,\tv+\z), t_1(t,x,\bv+\z)\} \leq 0 \leq \max\{t_1(t,x,\tv+\z), t_1(t,x,\bv+\z)\} \leq t$. For $s\in [\max\{t_1(t,x,\tv+\z), t_1(t,x,\bv+\z)\},t]$, we use the same arguments as in \eqref{V_s_maxt}. For $s \in [0,\max\{t_1(t,x,\tv+\z), t_1(t,x,\bv+\z)\}]$, we use \eqref{min_max_vbv}.

 \textbf{(6)} When $ \max\{t_1(t,x,\tv+\z), t_1(t,x,\bv+\z)\} \leq 0$, we use the same arguments as in \eqref{V_s_maxt} for $s \in [0,t]$.
   
   From (4), (5), and (6), we obtain
   \begin{align*} 
       \int_0^t (\ref{V_vel_ns}) ds &\lesssim_{\vartheta_{0}}   |\tv-\bv|\|w_0f_0\|_{\infty} G(x,\bv;\epsilon)e^{\varpi \langle v+\z \rangle^{2} t}\mathfrak{V}(t,\varpi;\epsilon) + |\tv-\bv| \|w_0f_0\|_{\infty}^2\frac{1}{|v+\z|}.
   \end{align*}

  \vspace{3mm}
   By combining the upper bound of $\int_0^t $\eqref{X_vel_s}-\eqref{V_vel_ns}$ds$ and using \eqref{G(x,bx,v,v)_up}, we obtain
   \begin{align*}
       \int_0^t \eqref{X,V_0} 
       &\lesssim
       |v-\bv|\mathcal{P}_2(\|w_0f_0\|_{\infty} )e^{\varpi \langle v+\z \rangle^{2} t} \left(\mathfrak{X}(t,\varpi;\epsilon)+\mathfrak{V}(t,\varpi;\epsilon)+\mathbf{A}_{\frac{1}{2}}(f_0)+1\right) \\
       &\quad \times  \left( \frac{1}{|v+\z|}+\frac{1}{|\bv+\z|}+|v+\z|+(|v+\z|^{\frac{1}{2}}+|v+\z|^2)\mathcal{T}_{vel}(x, v, \tv, \zeta;t)\right).
   \end{align*}
 Finally, applying \eqref{min_v,zeta}, we complete the estimate of $\int_0^t \eqref{X,V_0}$. \\

 Similarly, we estimate 
 \begin{align*}
     \int_0^{t} \left|\nu(f)(s,X(s;t,x,v+\z),V(s;t,x,v+\z))-\nu(f)(s,X(s;t,x,\bv+\z),V(s;t,x,\bv+\z))\right| ds.
 \end{align*}
However, we apply Lemma \ref{lem:diff_nu} instead of Lemma \ref{lem:est_Gam}. Previously, to estimate \eqref{V_vel_s}, we used Proposition \ref{pro:H}, as shown in \eqref{est_V_vel_s}. In the present case, however, we can directly use the inequality from \eqref{nu_v}:
\begin{align*}
    &\left|\nu(f)(s,X(s;t,x,\bv+\z),V(s;t,x,v+\z)) - \nu(f)(s,X(s;t,x,\bv+\z),V(s;t,x,\tv+\z))\right| \\
    &\leq |V(s;t,x,v+\z) - V(s;t,x,\tv+\z)|  \|w_0 f_0\|_{\infty} \leq 2|v+\z|\|w_0 f_0\|_{\infty},
\end{align*}
so Proposition \ref{pro:H} is not needed in this case. \\
\end{proof}

\section{Dynamical singular regime integration}
     The projected normal vector $n_{\parallel}$ is defined as
		\begin{equation*} 
		n_{\parallel}(x) := \text{Proj}_{ S}n(x) = \text{projection of $n(x)$ on $S$} = (I-\hat{q}\otimes \hat{q})n(x) ,
		\end{equation*}
		where $x\in\p\O$ and $\hat{q}$ is a unit vector orthogonal to the plane $S$. Recall the lemma of \cite{CD2023} concerning the uniform comparability of $n_{||}$.

	\begin{lemma}[Lemma 4.2 of \cite{CD2023}] \label{lem:unif n}
     Let $S$ be a plane in $\R^{3}$ and let $\O$ be as in Definition \ref{def:domain}. Assume that $\p\O\cap S$ is a closed curve. Then, $|n_{\parallel}(x) |$ is uniformly comparable for all $x\in \p\O\cap S$,  i.e., there exist uniformly positive constants $c$ and $C$, which only depend on $\O$,  such that
		\begin{align*}
		c < \frac{|n_{\parallel}(x)|}{|n_{\parallel}(y)|} \leq C, \quad \forall x,y \in \p\O\cap S.
		\end{align*}
		(For example,   it is obvious that $|n_{\parallel}(x)| = |n_{\parallel}(y)|$ for all $x,y \in \p\O \cap S$, if $\mathcal{O}$ is a sphere.)
	\end{lemma}

    In Section 5.1, we establish a geometric comparison between convex curves and a lower bound on approach angles. These results will be used in the proof of Lemma \ref{lem:int_sing}, where we compare a uniformly convex domain with a sphere.

\subsection{Geometric estimates for uniformly convex domains}

We consider two convex curves and a point \( x_1 \) outside both. Suppose a particle at \( x_1 \) grazes both curves at the same point with the same direction. If one curve has uniformly smaller curvature than the minimal curvature of the other, we observe that the corresponding specular reflection angle becomes larger, and thus its cosine becomes smaller.

\begin{lemma} \label{lem:angle_com}
We consider a smooth convex function $f : \mathbb{R}^{+}\cup \{0\} \rightarrow \mathbb{R}^{+}\cup \{0\}$, with $f(0)=0$ and $f'(0)=0$, and let $k(s)>0$ denote its curvature at $s\in \mathbb{R}^{+}\cup \{0\} $. We take another convex function $f_m : [0,k_m^{-1}] \rightarrow \mathbb{R}^{+}\cup \{0\}$ with constant curvature $k_m = \min_{s \geq 0} k(s)/2$, where $f_m(0)=0$ and $f_m'(0)=0$.$($i.e, $f_m$ represents one-fourth of a circle with radius of $k_m^{-1}$, and center at $(0,k_m^{-1}).)$ Let $g_{\delta} : \mathbb{R}^{+}\cup \{0\} \rightarrow \mathbb{R}^{+}\cup \{0\}$ be given as
\begin{align} \label{def_g}
    g_{\delta}(s) := ( x_1-s)\tan \delta 
\end{align} 
for $\delta \in [0, \pi/2]$ and $x_1 >0$. Let us define
\begin{align*}
    \mathbf{A}(s; F, \delta):=\left|\widehat{(-F'(s),1)} \cdot (-\cos \delta, \sin \delta) \right|=\frac{F'(s)\cos\delta+\sin \delta}{\sqrt{1+\left(F'(s)\right)^2}},
\end{align*} 
where $F=f \text{\;or\;} f_m$. Assume that there exist interaction points $p$ and $q>0$ for $f, f_m$ and $g_\delta$, respectively, such that
\begin{align}\label{def_p123}
   f(p) = g_{\delta}(p)
    \text{\quad and \quad}  f_m(q) = g_{\delta}(q).
\end{align} 
Then, there exists $\epsilon>0$ such that
\begin{align*} 
    \mathbf{A}(q;f_m, \delta) \leq \mathbf{A}(p;f, \delta) 
\end{align*} 
whenever $0<x_1\tan \delta < \epsilon$. (See Figure \ref{fig_angle_com}.) 
    
\end{lemma}
\begin{figure} [t] 
\centering
\includegraphics[width=7cm]{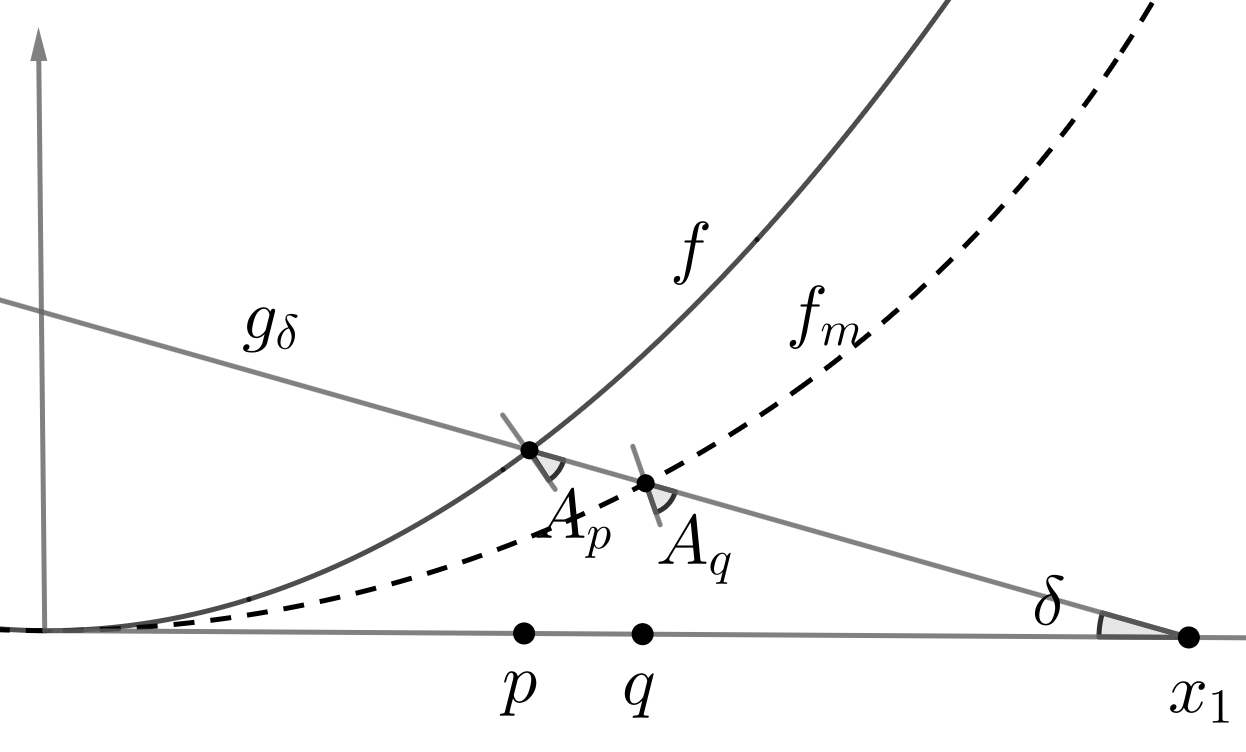}
\caption{When $\cos A_p:=\mathbf{A}(p;f, \delta) $ and $\cos A_q:=\mathbf{A}(q;f_m, \delta)$ as in the above figure, it holds that $\mathbf{A}(q;f_m, \delta) \leq \mathbf{A}(p;f, \delta)$.} \label{fig_angle_com}
\end{figure}

\begin{proof}
From the definition of curvature, the curvature $k(s)$ of the function $f$ satisfies
\begin{align} \label{curvature_by_f}
    k(s) = \frac{f''(s)}{(1+\left(f'(s)\right)^{2})^{3/2}} > 0
\end{align} 
at $s\in \mathbb{R}^{+}\cup \{0\}$. By using \eqref{curvature_by_f} and $f'(0)=0$, we have
\begin{align*}
\begin{split}
   \int_0^{s} k(z)f'(z)\;dz
   &=  \int_0^{s}  \frac{f'(z) f''(z)}{(1+\left(f'(z)\right)^{2})^{3/2}} dz = 1 -\frac{1}{\sqrt{1+(f'(s))^2}}
\end{split}
\end{align*} 
and since $f(0)=0$,
\begin{align*}
    \int_0^{s} k(z)f'(z)\;dz \geq  f(s)\inf_{s \geq 0} k(s).
\end{align*}
Thus, we have
\begin{align} \label{f'(s)_esti}
    f'(s) \geq \sqrt{\frac{1}{(1-f(s)\inf_{s \geq 0} k(s) )^2} -1 }
\end{align} 
If we replace $f, k(s)$ by $f_m, k_m$, respectively, then
\begin{align}\label{f1_f1'}
    f_m'(s)= \sqrt{\frac{1}{(1- f_m(s)  k_m )^2} -1 }.
\end{align}

By Taylor mean value theorem and \eqref{curvature_by_f}, we have
\begin{align} \label{TM_f}
    f(s) = f(0) + f'(0)s + \frac{1}{2}f''(c)s^2 \geq \frac{1}{2}k(c)s^2 \geq \frac{1}{2}k_ms^2
\end{align} at $s\in \mathbb{R}^{+}\cup \{0\}$, where $c \in [0,s]$ is some constant. Using the fact that $g$ is the decreasing function, and \eqref{TM_f}, we have
\begin{align*}
    x_1\tan\delta \geq g_{\delta}(p)=f(p)\geq \frac{1}{2}k_m p^2.
\end{align*} 
This implies $p = O((x_1 \tan \delta)^{1/2})$, and similarly, $q=O((x_1 \tan \delta)^{1/2})$.

Using Taylor expansion of the function $f$ and \eqref{curvature_by_f}, we also express $f$ as
\begin{align} \label{f(s)_esti}
    f(s) &= f(0) + f'(0)s  + \frac{f''(0)}{2}s^2+O(s^3) 
    = \frac{1}{2}k(0)s^2 ++O(s^3),
\end{align} and \eqref{f(s)_esti} also holds for $f_m(s)$ in place of $f$. By \eqref{def_p123} and \eqref{f(s)_esti}, it holds that 
\begin{align*}
    q\tan \delta +\frac{1}{2}k_m q^2 +O(q^3)=p\tan \delta +\frac{1}{2}k(0)p^2 +O(p^3)=x_1\tan \delta
\end{align*} at the interaction points $q$ and $p$.  From the first and second equations above, we have
\begin{align*}
    (q-p)\left(  \tan \delta + \frac{1}{2}k_m(p+q) \right) 
    &= \frac{1}{2}p^2 (k(0)-k_m) + O(p^3)-O(q^3) \\
    &= \frac{1}{2}(k(0)-k_m)  O(x_1 \tan \delta) + O((x_1 \tan \delta)^{\frac{3}{2}}).
\end{align*} 
Since $2k_m \leq k(0)$,
we find that there exists $0<\epsilon \ll 1$ such that $0<p \leq q$ whenever $0<x_1 \tan \delta<\epsilon$. It also holds that
\begin{align}\label{f12_ineq}
    0< f_m(q) = g_{\delta}(q) \leq f(p) =g_{\delta}(p)  \ll 1.
\end{align} for $0<x_1\tan \delta<\epsilon$. 

By \eqref{f'(s)_esti}, \eqref{f1_f1'} and \eqref{f12_ineq}, we have
\begin{align*}
   0< f_m'(q) =\sqrt{\frac{1}{(1-f_m(q) k_m)^2}-1} \leq  \sqrt{\frac{1}{1- (f(p)\inf_{s \geq 0} k(s))^2 }  -1 }\leq f'(p) \ll 1
\end{align*} 
for $0<x_1\tan \delta<\epsilon$. Finally, using $0<f_m'(q)\leq f'(p) \ll 1$ and $0<f_m(q) \leq f(p) \ll 1$ for $0<x_1\tan \delta<\epsilon$,  we obtain $  \mathbf{A}^2(q;f_m, \delta) \leq \mathbf{A}^2(p;f, \delta)$. 
\end{proof}

We consider a particle starting from a point on the \( x \)-axis and approaching a smooth convex curve from below. In this setting, the origin becomes the grazing point, but the point on the curve that is actually closest lies slightly to the right. The following lemma provides a lower bound on the slope between these two points, ensuring that the particle does not approach the curve too flatly.

\begin{lemma} \label{lem:the_g}
    We consider a smooth convex function $f : \mathbb{R}^{+}\cup \{0\} \rightarrow \mathbb{R}^{+}\cup \{0\}$, with $f(0)=0$ and $f'(0)=0$. Fix $x_1>0$. Let $(p^*, f(p^*))$ be the point on the function $f$ that is closest to $(x_1,0)$. (See Figure \ref{fig_the_g}.) Then $x_1 \geq p^*$ holds and there exists a constant $\epsilon>0$, independent of $x_1$, such that 
    \begin{align*}
        \frac{f(p^*)}{x_1-p^*} \geq \frac{\epsilon}{x_1}.
    \end{align*} 
\end{lemma}
\begin{figure} [t] 
\centering
\includegraphics[width=6cm]{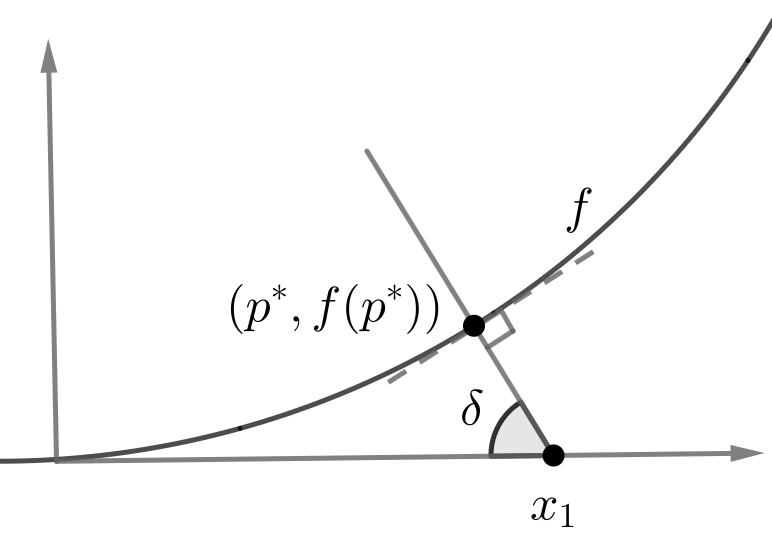}
\caption{When $\tan \delta=f(p^{*})/(x_1-p^{*})$, it holds that $\tan \delta \geq \epsilon/x_1$ for $\epsilon>0$.} \label{fig_the_g}
\end{figure}
\begin{proof}
    Let $h(\cdot; \epsilon): [0,\epsilon] \rightarrow  [0, \epsilon]$ be a function defined for $\epsilon > 0$ with constant curvature $\epsilon^{-1}$, given by
    \begin{align*}
        h(s;\epsilon) = \epsilon - \sqrt{\epsilon^2-s^2}.
    \end{align*} Additionally, let $g(\cdot;{\epsilon}) : [0,x_1] \rightarrow \mathbb{R}^{+}\cup \{0\} $ be a linear function defined as
    \begin{align*}
        g(s;{\epsilon}) = \frac{\epsilon}{x_1}(x_1-s).
    \end{align*}
    Let $\alpha(\epsilon)>0$ be the intersection point of the functions $f$ and $g(\cdot;\epsilon)$, and $\beta(\epsilon)>0$ be the intersection point of $h(\cdot;\epsilon)$ and $g(\cdot;\epsilon)$. Then $x_1 \geq \alpha(\epsilon),\beta(\epsilon) \geq 0$. 
    
    Using the fact that $g(\cdot;\epsilon)$ is the decreasing function, and \eqref{TM_f}, we obtain
    \begin{align*}
        \epsilon \geq g(\alpha(\epsilon);\epsilon)=f(\alpha(\epsilon))\geq \frac{1}{2}k_m \alpha(\epsilon)^2
    \end{align*} 
    for $k_m = \min_{s \geq 0} k(s)$ and
    \begin{align*}
        \epsilon \geq g(\beta(\epsilon);\epsilon)=h(\beta(\epsilon);\epsilon)\geq \frac{1}{2\epsilon} \beta(\epsilon)^2.
    \end{align*} 
    Thus,
    \begin{align} \label{big_O}
    \alpha(\epsilon)=O(\epsilon^{\frac{1}{2}}) \text{\quad and \quad }\beta(\epsilon)=O(\epsilon).
    \end{align}
    
    Using Taylor expansion of the function and \eqref{curvature_by_f}, we obtain \eqref{f(s)_esti} and
    \begin{align} \label{taylor_f_eps}
          h(s;\epsilon) 
    = \frac{1}{2\epsilon}s^2 ++O(s^3). 
    \end{align} 
    Then, since $g(\alpha(\epsilon);\epsilon)=f(\alpha(\epsilon))$ and $g(\beta(\epsilon);\epsilon)=h(\beta(\epsilon);\epsilon)$, 
    \begin{align} \label{at alpha,beta}
        \frac{1}{2}k(0)\alpha(\epsilon)^2+\frac{\epsilon}{x_1}\alpha(\epsilon) + O(\alpha(\epsilon)^3)
        =  \frac{1}{2\epsilon}\beta(\epsilon)^2+\frac{\epsilon}{x_1}\beta(\epsilon) + O(\beta(\epsilon)^3) = \epsilon.
    \end{align} 
    From the first and second equations above,
    \begin{align*}
        (\alpha(\epsilon)-\beta(\epsilon)) \left( \frac{1}{2}k(0)(\alpha(\epsilon)+\beta(\epsilon))+\frac{\epsilon}{x_1}\right) = O(\beta(\epsilon)^3)-O(\alpha(\epsilon)^3)+\left( \frac{1}{2\epsilon}-\frac{1}{2}k(0) \right)\beta(\epsilon)^2.
    \end{align*} 
    On the other hand, if $\epsilon$ goes to $0$, then
   \begin{align*}
       O(\beta(\epsilon)^3)-O(\alpha(\epsilon)^3)+\left( \frac{1}{2\epsilon}-\frac{1}{2}k(0) \right)\beta(\epsilon)^2 =
       O(\epsilon^{3})-O(\epsilon^\frac{3}{2})+ O(\epsilon)-O(\epsilon^2) \geq 0.
   \end{align*} 
   Hence, there exists $\epsilon_1>0$, which does not depend on $x_1$, such that if $\epsilon<\epsilon_1$, then $\alpha(\epsilon) \geq \beta(\epsilon)$.
    
    From \eqref{f(s)_esti} and \eqref{taylor_f_eps}, 
    \begin{align*}
        f'(s) = k(0)s + O(s^2)
        \text{\quad and \quad} h'(s;\epsilon)=\frac{1}{\epsilon}s + O(s^2).
    \end{align*} 
    Then, \eqref{at alpha,beta} can be rewritten as 
    \begin{align*}
        \frac{1}{2}f'(\alpha(\epsilon))\alpha(\epsilon) + \frac{\epsilon}{x_1}\alpha(\epsilon)+O(\alpha(\epsilon)^3) =
        \frac{1}{2}h'(\beta(\epsilon);\epsilon) \beta(\epsilon) + 
        \frac{\epsilon}{x_1}\beta(\epsilon)+O(\beta(\epsilon)^3) = \epsilon.
    \end{align*}
    From the first and second equations above, 
    \begin{align*}
       \frac{1}{2}\big(h'(\beta(\epsilon);\epsilon)-f'(\alpha(\epsilon))\big)\beta(\epsilon)
        =O(\alpha(\epsilon)^3)-O(\beta(\epsilon)^3)+(\alpha(\epsilon)-\beta(\epsilon))\left(\frac{1}{2}f'(\alpha(\epsilon))+\frac{\epsilon}{x_1}  \right) \geq 0.
    \end{align*} 
    Note that RHS is positive as $\epsilon \rightarrow 0$ in the above equation. There exists a constant $\epsilon_2 \in (0,\epsilon_1)$, independent of $x_1$, such that
    $h'(\beta(\epsilon);\epsilon) \geq f'(\alpha(\epsilon))$
    for all $0<\epsilon<\epsilon_2$.

    Next, we define the function $\Theta(\cdot;F): [0,x_1] \rightarrow  [0, \pi/2]$ as
\begin{align*}
    \tan\Theta(s;F) := \frac{F(s)}{x_1 - s},
\end{align*} 
and 
    \begin{align*}
    \mathbf{B}(s,F;\Theta)&:= \widehat{(1,F'(s))} \cdot \widehat{(1,-\tan\Theta(s;F))} \\
    &= \frac{1}{\sqrt{1+(F'(s))^2}}\frac{1}{\sqrt{1+\tan^2\Theta(s;F)}}
    \left( 1-F'(s) \tan\Theta(s;F) \right)
\end{align*} 
for $F= f$ or $h(\cdot;\epsilon)$. 

Fix $\epsilon \in (0, \epsilon_2)$. Since both $(\alpha(\epsilon), f(\alpha(\epsilon)))$ and $(\beta(\epsilon),h(\beta(\epsilon);\epsilon))$ are on the graph of $g(;\epsilon)$, 
\begin{align*}
    \tan\Theta(\alpha(\epsilon);f)
    = \tan\Theta(\beta(\epsilon);h(\cdot;\epsilon))
    = \frac{\epsilon}{x_1}.
\end{align*} 
$\mathbf{B}(\beta(\epsilon),h(\cdot,\epsilon);\Theta)=0$ equals zero since $h(\cdot,\epsilon)$ is a circle.
Thus, by $h'(\beta(\epsilon);\epsilon) \geq f'(\alpha(\epsilon))$ and the above inequality, 
\begin{align*} 
    \mathbf{B}(\alpha(\epsilon),f;\Theta ) \geq \mathbf{B}(\beta(\epsilon),h(\cdot,\epsilon);\Theta) = 0.
\end{align*}
Furthermore, since $f(s),f'(s)$ and $\tan\Theta(s;f)$ are increasing functions of the variable $s \in [0,x_1]$, it follows that $\mathbf{B}(s,f; \Theta)$ is a decreasing function of the variable $s \in [0,x_1]$.

Let $(p^*, f(p^*))$ be the point on the function $f$ that is closest to $(x_1,0)$, such that
\begin{align*} 
    \min_{s\geq 0} \left|(s,f(s)) - (x_1,0) \right| = \left|(p^*,f(p^*)) - (x_1,0) \right|.
\end{align*} 
Then $(1,f'(p^{*})) \bot (p^*-x_1,f(p^*))$ and $\mathbf{B}(p^*,f;\Theta)=0$ hold.
We also get $ x_1 \geq p^* \geq 0$ since the function $f$ is positioned above its tangent line at $(p^*, f(p^{*}))$ of the function $f$, which is perpendicular to $\overrightarrow{(x_1-p^*, f(p^*))}$. Because $\mathbf{B}(p^*,f;\Theta)=0,\; \mathbf{B}(\alpha(\epsilon),f;\Theta ) \geq 0$ and $\mathbf{B}(s,f; \Theta)$ is the decreasing function, we obtain $0 \leq \alpha(\epsilon) \leq p^* \leq x_1$. Therefore, we have
\begin{align*}
     \tan\Theta(p^*;f) = \frac{f(p^*)}{x_1-p^*}\geq  \tan\Theta(\alpha(\epsilon);f) = \frac{\epsilon}{x_1},
\end{align*} and the proof is complete.

\end{proof}

\subsection{Integrability for velocity}

In Lemma \ref{lem:int_cir}, we estimate an integral over all directions of incoming velocities that leads to specular reflection at a fixed point \(x\), where the boundary is a circle. The integrand is the reciprocal cosine of the incidence angle, which becomes large when the angle is close to grazing. However, the total integral is bounded by a logarithmic term depending on the distance from \( x \) to the boundary.

\begin{lemma} \label{lem:int_cir}
We define a domain $B_{R}(0)$ as
    \begin{align} \label{B_R_def}
    B_{R}(0) := \left\{(x,y) \in \mathbb{R}^2 : x^2+y^2 = R^2, \quad x \geq 0 \right\}
\end{align} for positive constant $R>0$. Fix  $x = (|x|,0)$ for $|x|>R$. We also define 
\begin{align} \label{alpha_g_def}
    \alpha_g(x,R) := \sin^{-1}\left(\frac{R}{|x|}\right) \in \left(0, \frac{\pi}{2}\right)
\end{align} (See Figure \ref{fig_int_cir}.)  and 
\begin{align} \label{tbxb_R_def}
\begin{split}
    &t_{\mathbf{b}}(x,u;R):=\sup \left\{s \geq 0 : x-\tau u \in  B_{R}(0) \text{\;for\;all\;} \tau \in (0,s)\right\}, \\
    &x_{\mathbf{b}}(x,u;R) := x-t_{\mathbf{b}}(x,u;R)u
\end{split}
\end{align} for $u \in \mathbb{R}^2$. Note that $\xb(x, u; R)$ depends only on the direction of $u$, not on its magnitude $|u|$. Then, we obtain
\begin{align*}
    \int_{\frac{1}{2}\alpha_g(x,R)}^{\alpha_g(x,R)}  \frac{1}{\left|\xb(x,u;R) \cdot \hat{u}\right|} du_{\theta} 
    \leq \frac{1}{R}\ln \left(  1+ \frac{2R}{\min_{y \in B_{R}(0)} |x - y|}  \right)
\end{align*} for $\hat{u} = (\cos u_{\theta}, \sin u_{\theta})$.
\end{lemma}
\begin{figure} [t] 
\centering
\includegraphics[width=7cm]{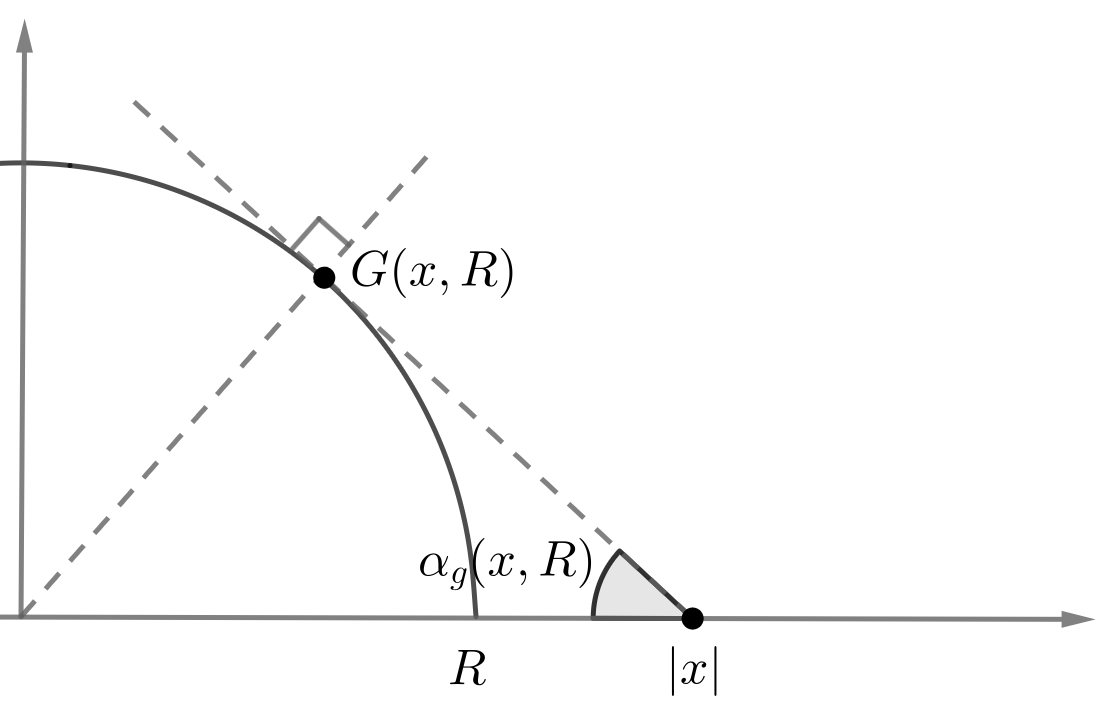}
\caption{$\alpha_g(x,R)$ and $G(x,R)$ on the ball $B_{R}(0)$.}\label{fig_int_cir}
\end{figure}
\begin{proof}
From $|\xb(x,u;R)|^2=|x-\tb(x,u;R)u|^2=R^2$, we obtain
 \begin{align*}
    |u|^2\tb(x,u;R) = (x\cdot u)-\sqrt{(x\cdot u)^2-|u|^2(|x|^2-R^2)}.
\end{align*} 
By plugging in $\tb(x,u;R)$, 
\begin{align*}
\xb(x,u;R) \cdot \hat{u} =  (x-\tb(x,u;R)u) \cdot \hat{u}=\sqrt{(x \cdot \hat{u})^2-|x|^2+R^2}.
\end{align*} 
For $ x= (|x|, 0)$ and $\hat{u} = (\cos u_{\theta}, \sin u_{\theta})$,
\begin{align*}
    \xb(x,u;R) \cdot \hat{u} =\sqrt{R^2-|x|^2\sin^2u_{\theta}}. 
\end{align*} 
For $u_{\theta} \in [\alpha_g(x,R)/2, \alpha_g(x,R)]$, 
\begin{align*}
    \sin u_{\theta} \geq \sin \frac{1}{2}\alpha_g(x,R) \geq \frac{1}{2}\sin \alpha_g(x,R) \geq \frac{R}{2|x|}.
\end{align*} 

By plugging in $\xb(x,u;R)\cdot \hat{u}$ and using the above inequality, 
\begin{align*}
      &\int_{\frac{1}{2}\alpha_g(x,R)}^{\alpha_g(x,R)}  \frac{1}{\left|\xb(x,u;R) \cdot \hat{u}\right|} du_{\theta} 
      = \int_{\frac{1}{2}\alpha_g(x,R)}^{\alpha_g(x,R)} \frac{1}{\sqrt{R^2-|x|^2\sin^2u_{\theta}}} du_{\theta}\\
      &\leq \frac{2}{R} \int_{\frac{1}{2}\alpha_g(x,R)}^{\alpha_g(x,R)} \frac{\sin u_{\theta}}{\sqrt{\frac{R^2}{|x|^2}-\sin^2u_{\theta}}} du_{\theta} \leq \frac{2}{R} \int_{0}^{\alpha_g(x,R)} \frac{\sin u_{\theta}}{\sqrt{\frac{R^2}{|x|^2}-\sin^2u_{\theta}}} du_{\theta}.
\end{align*}
By changing the variable $u_{\theta} \rightarrow t=\frac{R^2}{|x|^2}-\sin^2 u_{\theta}$, 
\begin{align*}
   &\int_{0}^{\alpha_g(x,R)} \frac{\sin u_{\theta}}{\sqrt{\frac{R^2}{|x|^2}-\sin^2u_{\theta}}} du_{\theta} = \frac{1}{2}\int_0^{\frac{R^2}{|x|^2}} \frac{1}{\sqrt{t(t+1-\frac{R^2}{|x|^2})}} dt \\
   &=  \ln \left(\sqrt{t+1-\frac{R^2}{|x|^2}}+\sqrt{t} \right) \bigg|_{t=0}^{t=\frac{R^2}
   {|x|^2}}   = \frac{1}{2}\ln \left( 1+ \frac{2R}{|x|-R}  \right).
\end{align*} 
Therefore, 
\begin{align*}
      &\int_{\frac{1}{2}\alpha_g(x,R)}^{\alpha_g(x,R)}  \frac{1}{\left|\xb(x,u;R) \cdot \hat{u}\right|} du_{\theta}  \leq \frac{1}{R} \ln \left( 1+ \frac{2R}{|x|-R}  \right)=\frac{1}{R} \ln \left( 1+ \frac{2R}{\min_{y \in B_{R}(0)} |x - y|}  \right).
\end{align*} 
\end{proof}

Next, we provide a comparison between the minimal distance to the boundary of the circle and the distance to the grazing point.

\begin{lemma} \label{lem:len_cir}
Fix $x = (|x|,0)$ for $|x|>R>0$. Recall the domain $B_{R}(0)$ and the grazing angle $\alpha_g(x,R)$ in \eqref{B_R_def} and \eqref{alpha_g_def}, respectively.
We define
\begin{align}\label{G_x,R_def}
    G(x,R) := \left( R\sin \alpha_g(x,R), R\cos \alpha_g(x,R)\right).
\end{align} 
(Note that $G(x,R)\bot\left(x-G(x,R)\right)$ holds. See Figure \ref{fig_int_cir}.) Then, we obtain
\begin{align*}
    \ln \left(1 + \frac{2R}{\min_{y \in B_{R}(0)} |x - y|}  \right) 
    \leq 2 \ln \left(1 + \frac{2\sqrt{3}R}{|x-G(x,R)|}\right).
\end{align*}
\end{lemma}
\begin{proof}
Let us denote
\begin{align*}
    l := |x -G(x,R)| \text{\quad and \quad} d:=\min_{y \in B_{R}(0)} |x - y|=|x|-R.
\end{align*} By the Pythagorean theorem, $ l^2  = (d+R)^2-R^2=d(d+2R)$. 
In the case of $d \geq R$, $l^2 \leq 3d^2$ holds, and we have
\begin{align*}
    \ln \left(1 +\frac{2R}{d} \right) \leq \ln \left(1 +\frac{2\sqrt{3}R}{l} \right).
\end{align*} 
In the case of $d \leq R$, $l^2\leq 3Rd$ holds, and we have
\begin{align*}
    \ln \left(1 +\frac{2R}{d} \right) \leq \ln \left(1 +\frac{6R^2}{l^2} \right) \leq 2 \ln\left(1 +\frac{\sqrt{6}R}{l} \right).
\end{align*} 
Finally, we combine the two above inequalities.
\end{proof}

We are now ready to apply the previous lemmas to prove the main estimate. In Lemma \ref{lem:int_sing}, we integrate over incoming directions that produce specular reflection at a fixed exterior point \( x \in \Omega \) of a uniformly convex domain. Although the integrand becomes singular near grazing angles, the total integral is bounded in terms of the distance from \( x \) to the boundary.

The proof is long and involved, so we begin with a brief outline of the main steps. We consider a uniformly convex domain in \( \mathbb{R}^3 \) and a fixed point \( x \) outside the domain. For each two-dimensional plane passing through \( x \), the intersection with the domain yields a uniformly convex curve. On each such curve, we focus on the grazing point and the grazing velocity.

We then construct a circle that passes through the grazing point, has a smaller curvature (i.e., a larger radius) than the curve, and shares the same grazing velocity. By applying Lemma \ref{lem:angle_com}, we find that the integral of the incidence angle over the curve is bounded by the corresponding integral over the circle. Lemma \ref{lem:int_cir} shows that the integral over the circle is controlled by the distance from \( x \) to the circle, and Lemma \ref{lem:len_cir} further bounds this by the distance from \( x \) to the grazing point.

Since the grazing point lies on both the curve and the circle, the angle integral over the curve is bounded by the distance from \( x \) to the grazing point in each plane. Finally, this distance is always larger than the minimal distance from \( x \) to the boundary of the uniformly convex domain in \( \mathbb{R}^3 \), which gives the desired bound.

\begin{lemma} \label{lem:int_sing}
Let $x \in \O, \,v \in \mathbb{R}^3$ be given with $d(x,\partial\O) > 0$. For $k <2$, we have
    \begin{align*}
           &\bigintsss_{\mathbb{R}^{3}} \left(\frac{e^{-\frac{c}{2}|u-v|^2}}{|u-v|}+|u-v|e^{-\frac{1}{8}|u|^2}\right) \frac{1}{|u|^k}\frac{ 1}{|\nabla \xi(\xb(x,u))\cdot\ \hat{u}|}\mathbf{1}_{\{\tb(x,u) < +\infty \} \cap \{d(x,\partial\O)\lesssim \langle u \rangle\}} du  \\
    &\lesssim C_k \left( \langle v \rangle + \langle v \rangle^{1-k} \right)\left[\ln\left(1+\frac{1}{d(x,\partial\O)} \right)+1\right],
    \end{align*}
    where $C_k = (2-k)^{-1}+1$.
\end{lemma} 
\begin{proof}
\begin{figure} [t] 
\centering
\includegraphics[width=6cm]{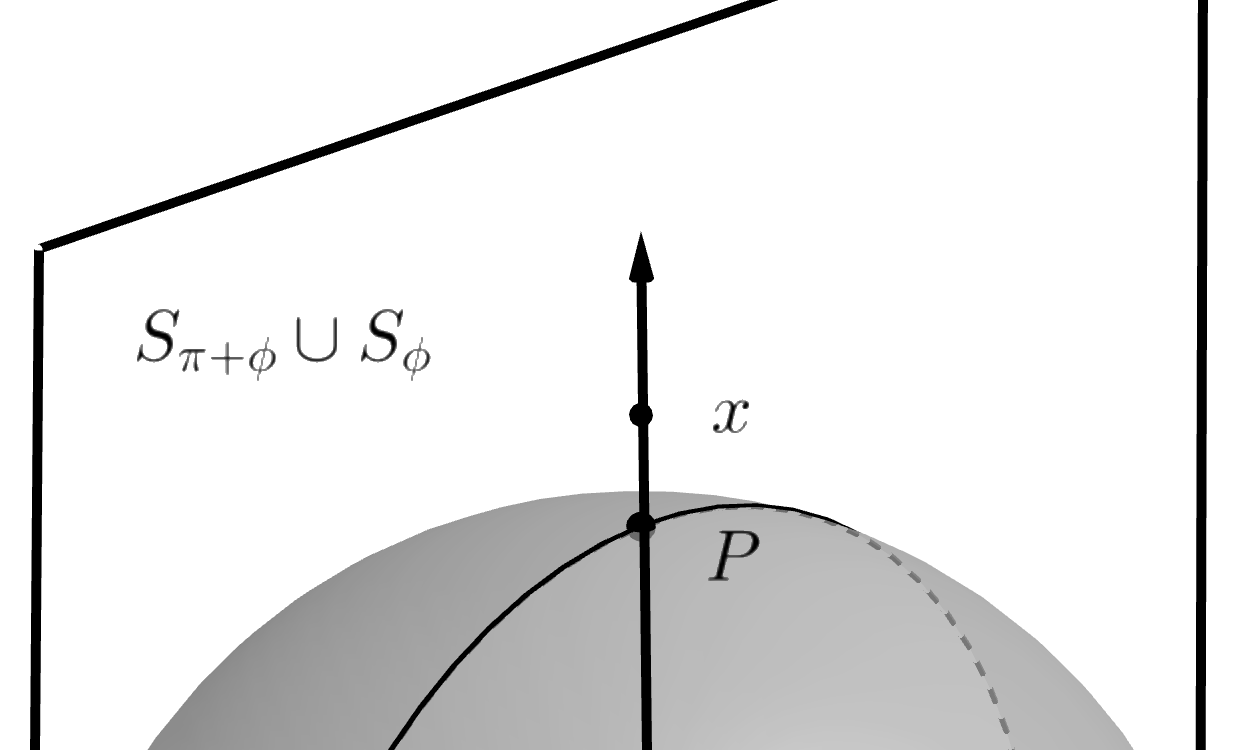}
\caption{Given $x \in \O$, we find the point $P \in \partial\O$ such that $\protect\overrightarrow{x-P} \perp T_P(\partial \O)$} \label{fig_3D}
\end{figure}
\textbf{(Step 1)}. Fix $x \in \O$. There exists a point $P \in \partial \O$ that is closest to $x$. Then 
\begin{align*}
    d(x,\partial\O):=\min_{y\in \partial \O} |x - y|= \left|x-P \right|
    \text{\quad and \quad } \overrightarrow{x-P} \;\bot\; T_{P}\left( \partial \O \right),
\end{align*}
where $ T_{P}\left( \partial \O \right)$ is tangential plane at $P$. Now, we consider spherical coordinates with $x$ as origin and $\overrightarrow{x-P}$ aligned with the $z$-axis.(See Figure \ref{fig_3D}.) Next, we define a half plane $S_{\phi}$ as
\begin{align*}
    S_{\phi} :=\left\{ (r,\theta,\phi):=(r\sin \theta\cos\phi, r \sin \theta \sin\phi, r\cos \theta) :  0 \leq r < +\infty, \quad 0 \leq \theta \leq \pi   \right\}
\end{align*}
for a fixed $\phi$(mod $2\pi$)$\in [0, 2\pi)$. Let the velocity $u$, start at $x=(0,0,0)$, where 
\begin{align} \label{u_sphere}
   u:=(|u|, \theta, \phi) =(|u|\sin \theta\cos\phi, |u| \sin \theta \sin\phi, |u|\cos \theta). 
\end{align} 
In the half-plane $S_{\phi}$, there exists a grazing velocity
\begin{align} \label{theta_g}
    u_{g,\phi}:=(|u|,\theta_{g,\phi}, \phi) \in S_{\phi}
    \text{\quad for \quad } \theta_{g,\phi} \in \left(0, \frac{\pi}{2}\right),
\end{align} 
starting at $x$, such that $ \nabla \xi\left(\xb(x,u_{g,\phi})\right) \cdot u_{g,\phi}= 0$ for $\xb(x,u_{g,\phi}) \in S_{\pi+\phi}$. By using $u_{g,\phi}$, we define a grazing point $G_{\phi}$ as 
\begin{align*}
   G_{\phi} := \xb(x,u_{g,\phi})= x- \tb(x,u_{g,\phi}) u_{g,\phi}\in \partial\O \cap S_{\pi+\phi}.
\end{align*} 
For the normal vector $   N_{\phi} := (-\sin \phi, \cos \phi, 0)$ of the plane $S_{\pi+\phi}$, we define the projection of the normal vector of the domain $\O$ as
\begin{align*} 
   n_{||,\phi}(s) := n(s) - \left( n(s) \cdot N_{\phi}\right)N_{\phi} 
\end{align*} 
at $s \in \partial \O \cap S_{\pi+\phi}$.(See Figure \ref{fig_curve}.) 
\begin{figure} [t] 
\centering
\includegraphics[width=7cm]{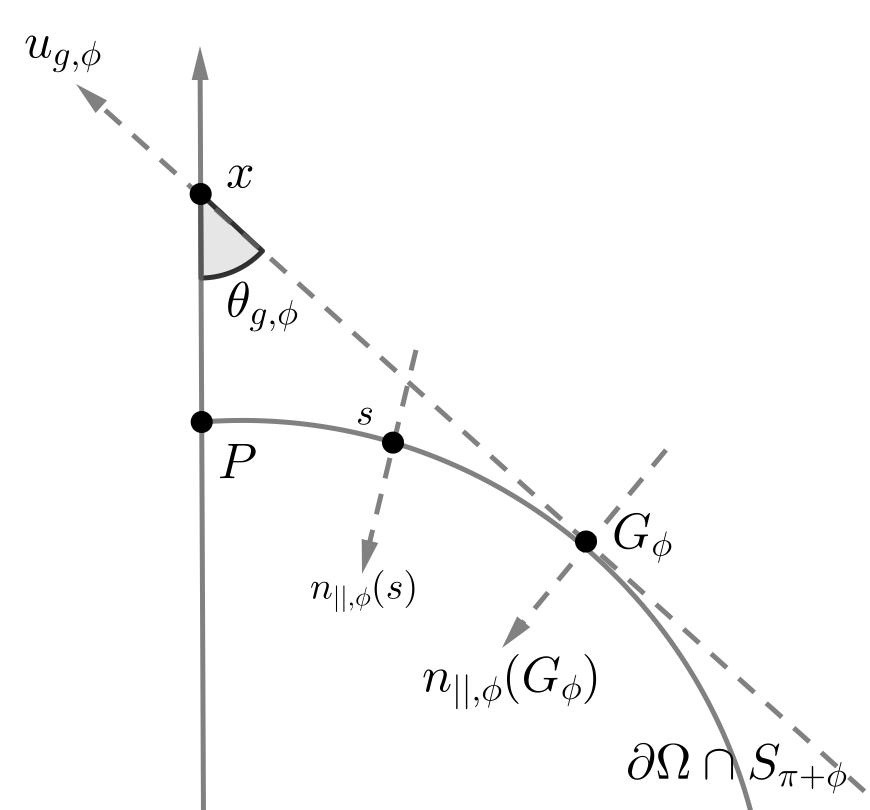}
\caption{The intersection of the boundary $\O$ with the place $S_{\pi+\phi}$, denoted by $\partial\Omega \cap S_{\pi+\phi}$.} \label{fig_curve}
\end{figure}

\vspace{3mm}
\textbf{(Step 2)}. Next, we focus on the curve $\partial\O \cap S_{\pi+\phi}$ that connects $P$ and $G_{\phi}$. 
We define $k_m>0$ as
\begin{align*}
    k_m := \frac{1}{2} \min_{\phi \in [0, 2\pi)}\min_{s \in \partial\O \cap S_{\phi}}k(s;\phi),
\end{align*} 
where $k(s;\phi)$ is the curvature of the curve $s \in \partial\O \cap S_{\phi}$. We define a circle with radius $k^{-1}_m$ that passes through $G_{\phi}$ and has centers $C_{m,\phi}  \in S_{\pi+\phi}\cup S_{\phi}$, such that 
\begin{align*} 
    B_{k^{-1}_m}(C_{m,\phi}) :=
    {\left\{s \in S_{\pi+\phi}: |s-C_{m,\phi}| = k^{-1}_m\right\}}
    \text{\quad for \quad}  C_{m,\phi}:=G_{\phi} + k^{-1}_m n_{||,\phi}( G_{\phi}).
\end{align*} 
Then, normal vectors of $B_{k^{-1}_m}(C_{m,\phi})$ are expressed by $n_{k^{-1}_m,\phi}(s) := \widehat{C_{m,\phi} - s}$ for $s \in B_{k^{-1}_m}(C_{m,\phi})$, (Note that $ \widehat{n_{||,\phi}}(G_{\phi}) = n_{k^{-1}_m,\phi}(G_{\phi}) \;\bot\; u_{g,\phi}$  holds.) and the angle between $\overrightarrow{x-G_{\phi}}$ and $\overrightarrow{x-C_{m,\phi}}$ is 
\begin{align}\label{theta_g_km}
    \theta_{g,k_m,\phi} := \tan^{-1} \left( \frac{1}{k_m|x-G_{\phi}|}\right) \in\left(0, \frac{\pi}{2}\right).
\end{align} 
(See Figure \ref{fig_draw_cir}.) Next, we define a backward exit time and position as
\begin{align*}
\begin{split}
    &t_{\mathbf{b}}(x,v;k_m^{-1},\phi):=\sup \left\{s \geq 0 : x-\tau v \in  B_{k^{-1}_m}(C_{m,\phi}) \text{\;for\;all\;} \tau \in (0,s)\right\},\\
    &x_{\mathbf{b}}(x,v; k_m^{-1},\phi) := x-t_{\mathbf{b}}(x,v;k_m^{-1})v
\end{split}
\end{align*}
for the domain $B_{k^{-1}_m}(C_{m,\phi})$ and $v \in S_{\phi}$. 

To apply Lemma \ref{lem:angle_com} and Lemma \ref{lem:the_g}, we consider a change the coordinate system. Let $S_{\pi+\phi}$ be the $xy$-plane, $G_{\phi}=(0,0)$, $\widehat{x-G_{\phi}}=(1,0)$ corresponding to the $x$-axis, and $\widehat{n_{||,\phi}}(G_{\phi})=(0,1)$ corresponding to the $y$-axis. Compared to Lemma \ref{lem:the_g}, we can match
\begin{align*}
    \partial \O \cap S_{\pi+\phi} = f, \quad |x-G_{\phi}| = x_1 \text{\quad and \quad} P = (p^*, f(p^*)),
\end{align*} 
where $f, x_1$ and $p^*$ are defined in Lemma \ref{lem:the_g}. Then, we can find a constant $\kappa_{\phi}>0$ such that
\begin{align} \label{kappa}
    \tan \theta_{g, \phi} = \frac{f(p^*)}{x_1-p^*}\geq \frac{\kappa_{\phi}}{|x-G_{\phi}|}=\frac{\kappa_{\phi}}{x_1}
\end{align} 
since $\theta_{g,\phi}$ defined in \eqref{theta_g} is the angle between $\overrightarrow{x-P}$ and the grazing velocity $u_{g,\phi}$. On the other hands, we recall $f_m, f, \delta, x$, $p$ and $q$ in Lemma \ref{lem:angle_com}, and we match
\begin{align*}
  B_{k^{-1}_m}(C_{m,\phi}) = f_m , \quad \partial \O \cap S_{\pi+\phi} = f,  \quad \theta_{g,\phi}-\theta = \delta,\\
  |x-G_{\phi}|=x_1, \quad \xb(x,u)=(p,f(p)), \quad \xb(x,u;k_m^{-1},\phi)=(q,f_m(q))
\end{align*} 
for $u = (|u|, \theta, \phi)$ in \eqref{u_sphere}. Then, we get
\begin{align*}
    \mathbf{A}(q;f_m, \delta) = |n_{k^{-1}_{m},\phi}(\xb(x,u;k_m^{-1},\phi)) \cdot \hat{u}|, \quad
    \mathbf{A}(p;f, \delta) = |\widehat{n_{||,\phi}}(\xb(x,u))  \cdot \hat{u}|.
\end{align*} 
By Lemma \ref{lem:angle_com}, there exists a constant $\epsilon_{\phi}>0$ such that
\begin{align} \label{eps_1}
    |\widehat{n_{||,\phi}}(\xb(x,u))  \cdot \hat{u} | \geq |n_{k^{-1}_{m},\phi}(\xb(x,u;k_m^{-1},\phi)) \cdot \hat{u}| 
\end{align} 
holds whenever $x_1\tan \delta = |x-G_{\phi}|\tan (\theta_{g,\phi}-\theta) \leq \epsilon_{\phi}$. 
\begin{figure} [t] 
\centering
\includegraphics[width=6cm]{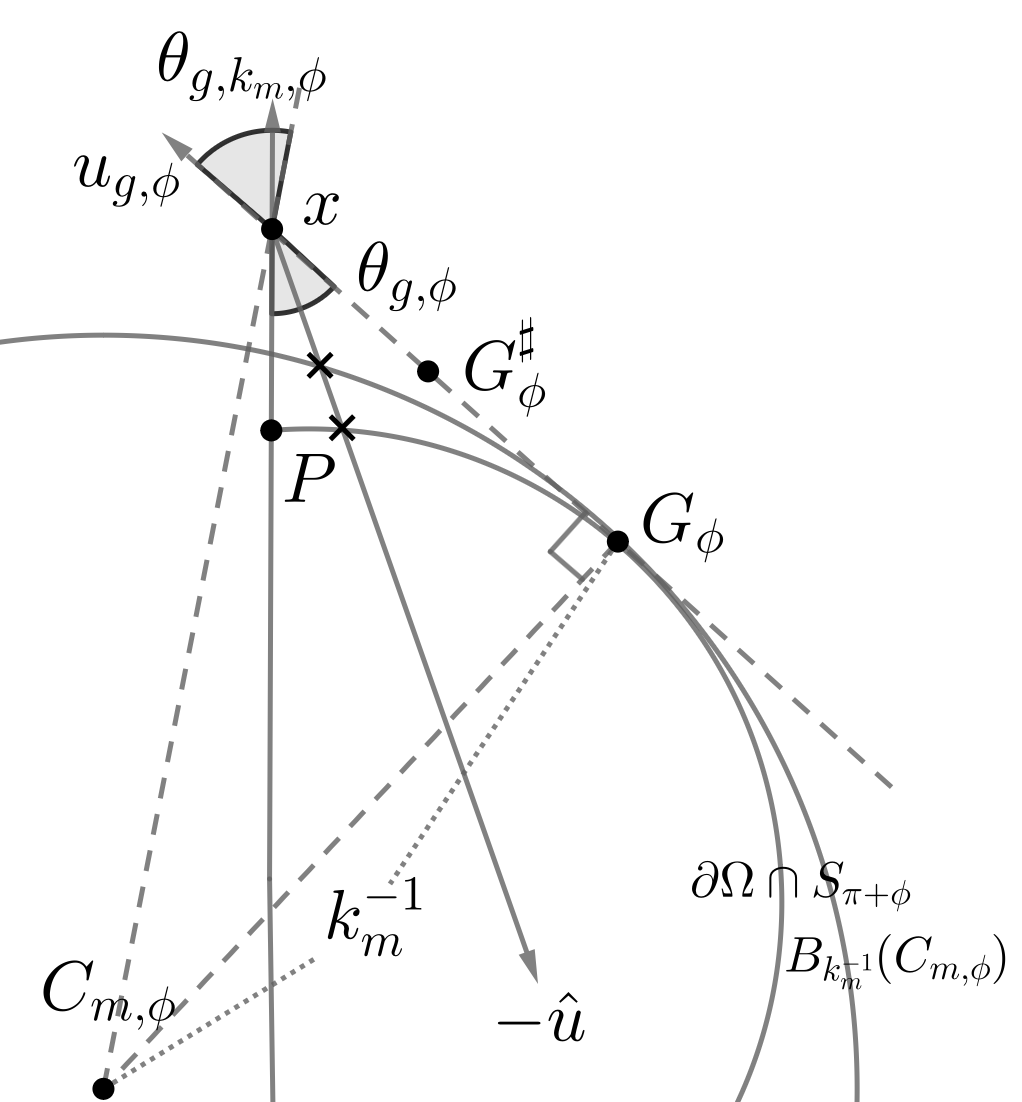}
\caption{A figure of a circle $B_{k_m^{-1}}(C_{m,\phi})$ passing through the point $G_{\phi}$.}\label{fig_draw_cir}
\end{figure}

\vspace{3mm}
\textbf{(Step 3)}. Let us denote $\min_{\phi \in [0,2\pi)}|x-G_{\phi}|=L>0$.
Using $L$, we define $G_{\phi}^{\sharp} \in S_{\pi+\phi}$ as
\begin{align*}
    x-G_{\phi}^{\sharp}=\frac{L}{|x-G_{\phi}|}(x-G_{\phi}).
\end{align*} 
We translate the circle $ B_{k^{-1}_m}(C_{m,\phi})$ in the direction of $\overrightarrow{x-G_{\phi}}$ by a length of $|G_{\phi}-G_{\phi}^{\sharp}|$. Then, the circle is expressed by 
\begin{align*} 
    B_{k^{-1}_m}(C_{m,\phi}^{\sharp}) :=
    {\left\{s \in S_{\pi+\phi}: |s-C_{m,\phi}^{\sharp}| = k^{-1}_m\right\}}
    \text{\quad for \quad}C_{m,\phi}^{\sharp}:=G_{\phi}^{\sharp} + k^{-1}_m n_{||,\phi}( G_{\phi}).
\end{align*}
The normal vectors of $B_{k^{-1}_m} (C_{m,\phi})$ are given by $ n_{k^{-1}_m,\phi}^{\sharp}(s):= \widehat{C_{m,\phi}^{\sharp} - s}$ for $s \in B_{k^{-1}_m}(C_{m,\phi}^{\sharp})$,
and the angle between $\overrightarrow{x-G_{\phi}^{\sharp}}$ and $\overrightarrow{x-C_{m,\phi}^{\sharp}}$ is 
\begin{align} \label{theta_sharp}
     \theta_{g,k_m,\phi}^{\sharp} := \tan^{-1}\left(\frac{1}{k_mL}\right)= \tan^{-1} \left( \frac{1}{k_m|x-G_{\phi}^{\sharp}|}\right) \in \left(0,\frac{\pi}{2}\right).
\end{align} 
For the domain $B_{k^{-1}_m}(C_{m,\phi}^{\sharp})$ and $v \in S_{\phi}$, we define
\begin{align*} 
\begin{split}
    &t_{\mathbf{b}}^{\sharp}(x,v;k_m^{-1}):=\sup \{s \geq 0 : x-\tau v \in  B_{k^{-1}_m}(C_{m,\phi}^{\sharp}) \text{\;for\;all\;} \tau \in (0,s)\}, \\
    &x_{\mathbf{b}}^{\sharp}(x,v;k_m^{-1}) := x-t_{\mathbf{b}}^{\sharp}(x,v;k_m^{-1})v.
\end{split} 
\end{align*} 
 Then, it holds that
\begin{align} \label{sharp_angle}
     |n_{k^{-1}_{m},\phi}(\xb(x,u;k_m^{-1},\phi)) \cdot \hat{u}| \geq  |n_{k^{-1}_{m},\phi}^{\sharp}(x_{\mathbf{b}}^{\sharp}(x,u;k_m^{-1})) \cdot \hat{u}|
\end{align} 
since $L \leq |x-G_{\phi}|$ and
\begin{align} \label{x-C_sharp}
    \left(\widehat{x-C_{m,\phi}^{\sharp}} \right)\cdot \hat{u}
    = \cos \left( \theta - \theta_{g,\phi}   +\theta_{g,k_m,\phi}^{\sharp} \right)
\end{align} for $(\widehat{x-P}) \cdot \hat{u} = \cos \theta$. (See Figure \ref{fig_sharp}.)

Now, we move the singularity $|n_{k^{-1}_{m},\phi}(\xb(x,u;k_m^{-1},\phi)) \cdot \hat{u}|$ from the plane $S_{\pi+\phi}$ to a $xy$-plane. First, we consider a general $xy$-plane. Let
\begin{align*}
    z:= \left(\sqrt{L^2+k_m^{-2}},0\right) = \left( \sqrt{|x-G_{\phi}^{\sharp}|^2+k_m^{-2}},0    \right)
\end{align*}
be fixed on the $xy$-plane. We define
$B_{k_m^{-1}}(0),\,\alpha_g(z,k_m^{-1})$ and $G(z,k_m^{-1})$ by substituting $R=k_m^{-1}$ and $x=z$ into \eqref{B_R_def}, \eqref{alpha_g_def}, and \eqref{G_x,R_def}, respectively, and observe 
\begin{align*}
\alpha_g(z,k_m^{-1})=\theta_{g,k_m,\phi}^{\sharp}, \quad   |z-G(z,k_m^{-1})| = |x-G_{\phi}^{\sharp}|=L.
\end{align*}
For $u \in \mathbb{R}^3$, we define $\tb(z,u;k_m^{-1})$ and $\xb(z,u;k_m^{-1})$ from \eqref{tbxb_R_def}. Let $u_1 \in \mathbb{R}^2$ be defined by $u_1 := ( \cos \omega, \sin\omega)$ for $\omega \in [0, \alpha_g(z;k_m^{-1})]$. Then, we obtain 
\begin{align} \label{sharp_xb}
   |\xb(z,u_1;k_m^{-1})\cdot u_1|= |n_{k^{-1}_{m},\phi}^{\sharp}(\xb(x,u_2;k_m^{-1})) \cdot u_2|,
\end{align}
where
\begin{align*}
    u_2 :=\left(1,\omega+\theta_{g,\phi}-\theta_{g,k_m,\phi}^{\sharp},\phi \right)
    \text{\quad such \;that \quad }
    \left(\widehat{x-C_{m,\phi}^{\sharp}} \right)\cdot u_2 = \cos \omega.
\end{align*} 
Here, $ |n_{k^{-1}_{m},\phi}^{\sharp}(\xb(x,u_2;k_m^{-1})) \cdot u_2|$ is defined in the above paragraph.

\begin{figure} [t] 
\centering
\includegraphics[width=7cm]{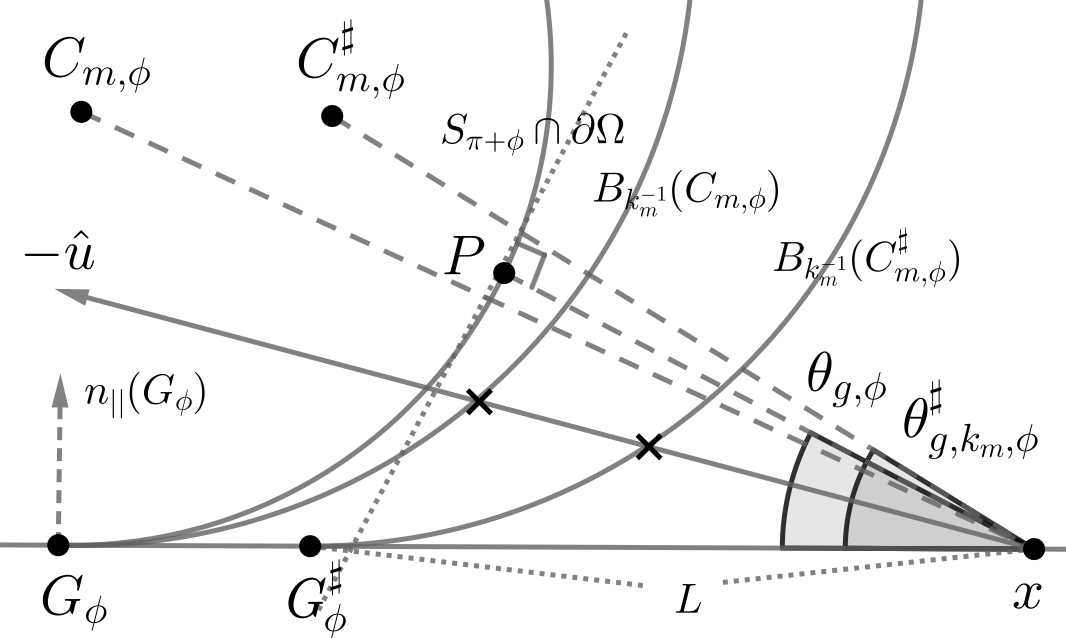}
\caption{A figure of a circle $B_{k_m^{-1}}(C_{m,\phi}^{\sharp})$ passing through the point $G_{\phi}^{\sharp}$.}\label{fig_sharp}
\end{figure}

\vspace{3mm}
\textbf{(Step 4)}.
Let $u$ be given by $u=(|u|,\theta,\phi)$ in \eqref{u_sphere}. Due to $n_{||,\phi}(P) = n(P)= \widehat{P-x}=(0,0,-1)$, $|n_{||,\phi}(P)|=1$ holds. Since both $\xb(x,u)$ and $P$ are on the curve $S_{\pi+\phi} \cap \partial \O$, we obtain  
\begin{align*}
    |n_{||,\phi}(\xb(x,u))| \gtrsim |n_{||,\phi}(P)|=1
\end{align*}
by Lemma \ref{lem:unif n}. Since $|\nabla \xi| \gtrsim_{\O} 1$, 
\begin{align*}
    |\nabla \xi(\xb(x,u))\cdot\ \hat{u}| 
    &=|\nabla \xi(\xb(x,u)| \cdot|n(\xb(x,u))\cdot \hat{u}|\\
    &=|\nabla \xi(\xb(x,u)|\cdot|  n_{||,\phi}(\xb(x,u))|\cdot |\widehat{n_{||,\phi}}(\xb(x,u)) \cdot \hat{u}| \\
    & \gtrsim |\widehat{n_{||,\phi}}(\xb(x,u)) \cdot \hat{u}|.
\end{align*}

By using the above inequality and the sphere coordinate \eqref{u_sphere}, we have
\begin{align*}
    &\int_{\mathbb{R}^{3}}\frac{e^{-\frac{c}{2}|u-v|^2}}{|u-v|} \frac{1}{|u|^k}\frac{ 1}{|\nabla \xi(\xb(x,u))\cdot\ \hat{u}|}\mathbf{1}_{\{\tb(x,u) < +\infty \} \cap \{d(x,\partial\O)\lesssim \langle u \rangle\}} du \notag\\
    &\lesssim
     \int_0^{\infty} \int_0^{2\pi} \int_{0}^{\theta_{g,\phi}}
    \frac{|u|^{2-k}}{|u-v|}e^{-\frac{c}{2}|u-v|^2} \frac{ 1}{|\widehat{n_{||,\phi}}(\xb(x,u)) \cdot \hat{u}|}\sin\theta\mathbf{1}_{\{\{d(x,\partial\O)\lesssim \langle u \rangle\}} \, d\theta d\phi d|u| 
\end{align*} 
for $\theta_{g,\phi} \in (0, \pi/2)$ in \eqref{theta_g}. 

We observe that if $u$ is a grazing velocity such that $\theta = \theta_{g,\phi}$, then
$|\widehat{n_{||,\phi}}(\xb(x,u)) \cdot \hat{u}|=0$. First, we define
\begin{align}\label{def_epsilon}
    \epsilon :=\min \left\{ \min_{\phi \in [0,2\pi)}\tan^{-1}\left(\frac{\epsilon_{\phi}}{|x-G_{\phi}|} \right), \quad \min_{\phi \in [0,2\pi)} \frac{1}{2} \theta_{g,\phi} ,\quad \frac{1}{2}  \alpha_g(z,k_m^{-1})  \right\},
\end{align} 
where $\theta_{g,\phi},\, \epsilon_{\phi}$, and $\alpha_g(z,k_m^{-1})$ are defined in \eqref{theta_g}, \eqref{eps_1}, and \eqref{theta_sharp}. Next, we divide the range of $\theta$ into the grazing part and non-grazing part: (1) $\theta \in  \left[\theta_{g,\phi}-\epsilon, \theta_{g,\phi} \right]$ and (2) $\theta \in  \left[0,\theta_{g,\phi}-\epsilon \right]$.

\vspace{3mm}
\textbf{(Step 4)-(1)} First, we consider the interval $\theta \in  \left[\theta_{g,\phi}-\epsilon, \theta_{g,\phi} \right]$. By \eqref{eps_1} and \eqref{sharp_angle}, 
\begin{align}
     &\int_0^{\infty} \int_0^{2\pi} \int_{\theta_{g,\phi}-\epsilon}^{\theta_{g,\phi}}
    \frac{|u|^{2-k}}{|u-v|}e^{-\frac{c}{2}|u-v|^2} \frac{ 1}{|\widehat{n_{||,\phi}}(\xb(x,u)) \cdot \hat{u}|}\sin\theta\, d\theta d\phi d|u|  \label{gra_temp1} \\
     &\leq \int_0^{\infty} \int_0^{2\pi} \int_{\theta_{g,\phi}-\epsilon}^{\theta_{g,\phi}}
    \frac{|u|^{2-k}}{|u-v|}e^{-\frac{c}{2}|u-v|^2} \frac{ 1}{|n_{k^{-1}_{m},\phi}^{\sharp}(x_{\mathbf{b}}^{\sharp}(x,u;k_m^{-1})) \cdot \hat{u}|}\sin\theta \,d\theta d\phi d|u|. \notag
\end{align} 
Note that $u=(|u|,\theta,\phi)$ in \eqref{gra_temp1} satisfies both $( \widehat{x-P}) \cdot \hat{u} = \cos \theta$ and \eqref{x-C_sharp}. By changing the variable $\theta \rightarrow \omega:=\theta-\theta_{g,\phi}+\theta_{g,k_m,\phi}^{\sharp}$ and using \eqref{sharp_xb}, we obtain
\begin{align}
    &(\ref{gra_temp1})= 
    \int_0^{\infty} \int_0^{2\pi} \int_{\alpha_{g}(z;k_m^{-1})-\epsilon}^{\alpha_{g}(z;k_m^{-1})}
    \frac{|u|^{2-k}}{|u-v|}e^{-\frac{c}{2}|u-v|^2} \frac{ \sin(\omega+\theta_{g,\phi}-\theta_{g,k_m,\phi}^{\sharp})}{|\xb(z,(\cos \omega, \sin \omega);k_m^{-1})\cdot (\cos \omega, \sin \omega)| }\,d\omega d\phi d|u|\notag \\
    &=\int_{\alpha_{g}(z;k_m^{-1})-\epsilon}^{\alpha_{g}(z;k_m^{-1})}\int_0^{2\pi} \int_0^{\infty}  \frac{ \sin(\omega+\theta_{g,\phi}-\theta_{g,k_m,\phi}^{\sharp})}{|\xb(z,(\cos \omega, \sin \omega);k_m^{-1})\cdot (\cos \omega, \sin \omega)| } 
    \frac{|u|^{2-k}}{|u-v|}e^{-\frac{c}{2}|u-v|^2} \,d|u| d\phi d\omega,
    \label{gra_temp2}
\end{align}
where $u=(|u|,\omega+\theta_{g,\phi}-\theta_{g,k_m,\phi}^{\sharp},\phi)$. 

 For fixed $\omega$, we change the variables such that
\begin{align} \label{dudphi_dA}
\begin{split}
    d|u| d\phi &= \frac{1}{\sin(\omega+\theta_{g,\phi}-\theta_{g,k_m,\phi}^{\sharp}) } d|u_p| d\phi \\
    &=\frac{1}{|u_p|\sin(\omega+\theta_{g,\phi}-\theta_{g,k_m,\phi}^{\sharp})} dA 
    =\frac{1}{|u|\sin^2(\omega+\theta_{g,\phi}-\theta_{g,k_m,\phi}^{\sharp}) } d A,
\end{split}
\end{align} 
where $u_p:=|u|\sin(\omega+\theta_{g,\phi}-\theta_{g,k_m,\phi}^{\sharp})\cdot (\cos \phi, \sin \phi) \in \mathbb{R}^2$. Here, $A$ is the $xy$-plane. 

By inequalities
\begin{align*}
    \omega \geq \alpha_{g}(z;k_m^{-1})-\epsilon=\theta_{g,k_m,\phi}^{\sharp}-\epsilon, \quad 0 \leq \epsilon \leq \theta_{g,\phi}/2,
\end{align*}
and \eqref{kappa}, we have
\begin{align*} 
    \sin (\omega+\theta_{g,\phi}-\theta_{g,k_m,\phi}^{\sharp}) \geq \sin (\theta_{g,\phi}-\epsilon)
    \geq \sin \frac{1}{2}\theta_{g,\phi}
    \geq \frac{1}{2}\sin \theta_{g,\phi}
    \geq \frac{1}{2}\frac{\kappa_\phi}{\sqrt{|x-G_{\phi}|^2+\kappa_\phi^2}},
\end{align*} 
and this implies that
\begin{align}\label{sin_lower}
    \frac{1}{\sin (\omega+\theta_{g,\phi}-\theta_{g,k_m,\phi}^{\sharp})} \lesssim \frac{1}{\kappa_{\phi}}(|G_{\phi}|+|x|)+1 \lesssim 1+|x|
\end{align} since $|G_{\phi}|\lesssim 1$ and $k_{\phi}\lesssim 1$.

We first estimate 
\begin{align} \label{dudphi} 
     \int_0^{2\pi} \int_0^{\infty} 
    \frac{|u|^{2-k}}{|u-v|}e^{-\frac{c}{2}|u-v|^2}\sin(\omega+\theta_{g,\phi}-\theta_{g,k_m,\phi}^{\sharp}) \,d|u| d\phi.
\end{align}
by dividing the range of $k$ into cases. Set $v_p:=v \cdot (\sqrt{2}^{-1},\sqrt{2}^{-1},0)$, which belongs to $A$. For $0 \leq k<1$, 
\begin{align*}
    \frac{|u|^{1-k}}{|u-v|} e^{-\frac{c}{2}|u-v|^2}\leq \frac{1}{|u_p-v_p|^k}+ \frac{|v|^{1-k}}{|u_p-v_p|} e^{-\frac{c}{2}|u_p-v_p|^2}.
\end{align*}
Then, for $0 \leq k<1$, by applying \eqref{dudphi_dA} and \eqref{sin_lower}, we obtain
\begin{align*}
    \eqref{dudphi} \lesssim (1+|x|) \int_{\mathbb{R}^2} \left(\frac{1}{|u_p-v_p|^k}+ \frac{|v|^{1-k}}{|u_p-v_p|} \right)e^{-\frac{c}{2}|u_p-v_p|^2} dA \lesssim (1+|x|)\langle v \rangle^{1-k}.
\end{align*}
Similarly, for $k \leq 0$, we obtain
\begin{align*}
     \eqref{dudphi} &\lesssim (1+|x|) 
    \int_{\mathbb{R}^2} \left(e^{-\frac{c}{4}|u_p-v_p|^2} + \frac{|v|^{1-k}}{|u_p-v_p|}e^{-\frac{c}{2}|u_p-v_p|^2} \right) dA
    \lesssim (1+|x|)\langle v \rangle^{1-k}.
\end{align*}
For $1 \leq k <2$,
\begin{align*}
    \frac{1}{|u|^{k-1}|u-v|} e^{-\frac{c}{2}|u-v|^2} &\leq  \frac{1}{|u-v|^{k}}e^{-\frac{c}{2}|u-v|^2} \mathbf{1}_{\{|u-v|\leq |u|\}}+\frac{1}{|u|^k}e^{-\frac{c}{2}|u|^2} \mathbf{1}_{\{|u|\leq |u-v|\}} \\
    &\leq \frac{1}{|u_p-v_p|}e^{-\frac{c}{2}|u_p-v_p|^2}+\frac{1}{|u_p|}e^{-\frac{c}{2}|u_p|^2}.
\end{align*}
Then, for $1 \leq k <2$, by applying \eqref{dudphi_dA} and \eqref{sin_lower}, we obtain
\begin{align*}
    \eqref{dudphi} 
    &\lesssim (1+|x|) 
    \int_{\mathbb{R}^2} \left( \frac{1}{|u_p-v_p|^{k}}e^{-\frac{c}{2}|u_p-v_p|^2} +\frac{1}{|u_p|^k}e^{-\frac{c}{2}|u_p|^2}\right) dA
    \lesssim (2-k)^{-1}(1+|x|).
\end{align*}
Therefore, for all $k < 2$, we apply the above results for \eqref{dudphi} to \eqref{gra_temp2}
\begin{align*}
    \eqref{gra_temp1} \lesssim C_k(1+|x|)(1+\langle v \rangle^{1-k}) \int_{\alpha_{g}(z;k_m^{-1})-\epsilon}^{\alpha_{g}(z;k_m^{-1})} \frac{ 1}{|\xb(z,(\cos \omega, \sin \omega);k_m^{-1})\cdot (\cos \omega, \sin \omega)| } d\omega,
\end{align*}
where $C_k=(2-k)^{-1}+1$.

 By applying Lemma \ref{lem:int_cir} for $R=k_m^{-1}$ and $x=z$ to the above inequality, we obtain
\begin{align*} 
    \eqref{gra_temp1}  &\lesssim C_k(1+|x|)(1+\langle v \rangle^{1-k})\ln \left(  1+ \frac{2k_m^{-1}}{\min_{y \in B_{k_m^{-1}}(0)} |z - y|}  \right) 
\end{align*} 
since $\epsilon \leq \alpha_{g}(z;k_m^{-1})/2$. By applying Lemma \ref{lem:len_cir} for $R=k_m^{-1}$ and $x=z$ and using $ |z-G(z,k_m^{-1})|=L \geq|x-P|$, we have
\begin{align*} 
     \eqref{gra_temp1}
    &\lesssim C_k(1+|x|)(1+\langle v \rangle^{1-k})\ln \left(  1+ \frac{4k_m^{-1}}{|z-G(z,k_m^{-1})|}  \right) \\
     &\lesssim C_k(1+|x|)(1+\langle v \rangle^{1-k})\ln \left(  1+ \frac{4k_m^{-1}}{|x-P|}  \right).
\end{align*} 
Therefore, we obtain
\begin{align*}
    (\ref{gra_temp1})
    &\lesssim C_k (1+\langle v \rangle^{1-k})(1+|x-P|)\ln \left(  1+ \frac{4k_m^{-1}}{|x-P|}  \right) \\
    &\lesssim C_k (1+\langle v \rangle^{1-k}) \left[  \ln \left(  1+ \frac{1}{|x-P|}  \right)+1\right],
\end{align*}
since $|P| \lesssim 1$, $k_m^{-1} \lesssim 1$ and $x \ln (1+x^{-1})  \lesssim 1$ for $x>0$.

\vspace{3mm}
\textbf{(Step 4)-(2)} Next, we consider the non-grazing part, corresponding to the interval $\theta \in  \left[0,\theta_{g,\phi}-\epsilon \right]$. For $u$ in \eqref{u_sphere},
\begin{align*}
   |\widehat{n_{||,\phi}}(\xb(x,u)) \cdot \hat{u}|
   &\geq  |\widehat{n_{||,\phi}}(G_{\phi}) \cdot \hat{u}| =
   \sin (\theta_{g,\phi}- \theta) \geq \sin \epsilon.
\end{align*} 
From the definition of $\epsilon$ in \eqref{def_epsilon}, 
\begin{align*} 
    &\sin \epsilon = \min \left\{  \min_{\phi \in [0,2\pi)}\frac{\epsilon_\phi}{\sqrt{|x-G_{\phi}|^2+\epsilon_\phi^2}}, \quad\min_{\phi \in [0,2\pi)} \sin \frac{1}{2} \theta_{g,\phi} , \quad \sin \frac{1}{2}  \alpha_g(z,k_m^{-1})     \right\}\\
    &\geq  \min \left\{  \min_{\phi \in [0,2\pi)}\frac{\epsilon_\phi}{\sqrt{|x-G_{\phi}|^2+\epsilon_\phi^2}}, \quad \frac{1}{2}\min_{\phi \in [0,2\pi)} \frac{\kappa_\phi}{\sqrt{|x-G_{\phi}|^2+\kappa_\phi^2}} , \quad\frac{1}{2}\frac{k_m^{-1}}{ \sqrt{|x-G_{\phi}^{\sharp}|^2+k_m^{-2}}} \right\}
\end{align*} 
by \eqref{kappa} and \eqref{theta_sharp}, since $\alpha_g(z,k_m^{-1})=\theta_{g,k_m,\phi}^{\sharp}$. Combining the above two inequalities with $d(x,\partial\O) \lesssim \langle u \rangle$, 
\begin{align*}
    \frac{1}{ |\widehat{n_{||,\phi}}(\xb(x,u)) \cdot \hat{u}|}
    \leq \frac{1}{\sin \epsilon} \lesssim |x|+1 \lesssim \langle u \rangle
\end{align*} 
since $1 \lesssim |G_{\phi}|, |G_{\phi}^{\sharp}|, \epsilon_{\phi}, \kappa_{\phi}, k_m^{-1}  \lesssim 1$. Therefore, for $k<2$, we conclude
\begin{align*}
    &\int_0^{\infty} \int_0^{2\pi} \int_{\theta_{g,\phi}-\epsilon}^{\theta_{g,\phi}}
    \frac{|u|^{2-k}}{|u-v|}e^{-\frac{c}{2}|u-v|^2} \frac{ 1}{|\widehat{n_{||,\phi}}(\xb(x,u)) \cdot \hat{u}|} \mathbf{1}_{\{d(x,\partial\O) \lesssim \langle u \rangle\}}\sin\theta\, d\theta d\phi d|u| \\
    &\lesssim  \int_{\mathbb{R}^3} \frac{\langle u \rangle}{|u|^{k}|u-v|}e^{-\frac{c}{2}|u-v|^2} du \lesssim C_k(1+ \langle v \rangle^{1-k}).
\end{align*}

Finally, we combine the results of (1) and (2) to obtain
\begin{align*}
     &\int_{\mathbb{R}^{3}} \frac{e^{-\frac{c}{2}|u-v|^2}}{|u-v|} \frac{1}{|u|^k}\frac{ 1}{|\nabla \xi(\xb(x,u))\cdot\ \hat{u}|}\mathbf{1}_{\{\tb(x,u) < +\infty \} \cap \{d(x,\partial\O)\lesssim \langle u \rangle\}} du \\
    &\lesssim C_k (1+\langle v \rangle^{1-k}) \left[  \ln \left(  1+ \frac{1}{|x-P|}  \right)+1\right].
\end{align*}
Similar arguments to those used in the above inequality are applied, which allow us to easily obtain
\begin{align*}
    &\int_{\mathbb{R}^{3}} |u-v|e^{-\frac{1}{8}|u|^2} \frac{1}{|u|^k}\frac{ 1}{|\nabla \xi(\xb(x,u))\cdot\ \hat{u}|}\mathbf{1}_{\{\tb(x,u) < +\infty \} \cap \{d(x,\partial\O)\lesssim \langle u \rangle\}} du \\
    &\lesssim C_k\langle v \rangle
     \left[\ln \left(  1+ \frac{1}{|x-P|}  \right)+1   \right].
\end{align*} 
\end{proof}

\subsection{Integrability for time}

We estimate a logarithmically singular integral along backward trajectories and show that it can be controlled in terms of a function of \( |v| \).

\begin{lemma}\label{lem:ds}
Let $0<t \leq 1, \; x \in \O,\;v \in \mathbb{R}^3$ be given. For $\varpi>1$, we obtain 
    \begin{align} \label{ds_result}
      \bigintsss_0^t e^{-\varpi \langle v \rangle^{2} (t-s)}\ln \left(  1+ \frac{1}{d(X(s;t,x,v),\partial\O)} \right) ds \lesssim \frac{1}{\varpi \langle v \rangle^2}\left[\ln\left(1+\frac{\varpi\langle v \rangle^2}{|v|}\right)+1 \right].
    \end{align} 
\end{lemma}
\begin{proof}
We assume $\tb(x,v) < \infty$. For $s\in [0,t]$, we denote $X(s):=X(s;t,x,v)$, and let $P(s) \in \partial\O$ be the closest point to $X(s)$, that is
\begin{align*}
    d(X(s),\partial\O) = \min_{y \in \partial\O} |X(s)-y | = |X(s)-P(s)|.
\end{align*}
For $s \in [0,t]$, let $U(s)$ denote the plane containing both $P(s)$ and $\xb(x,v)$, with normal vector
\begin{align*}
    \hat{n}(\xb(x,v)) \times (P(s)-\xb(x,v)).
\end{align*}
There exists a point $O(s)\in U(s)$ such that
\begin{align*}
    \frac{O(s)-\xb(x,v)}{|O(s)-\xb(x,v)|} = (0,0,1)=\hat{n}(\xb(x,v)) 
    \text{\quad and \quad} (O(s) - P(s)) \;\bot\; (\xb(x,v) - P(s)).
\end{align*}
Define
\begin{align*}
   C(s):= \frac{1}{2} \left(\xb(x,v)+O(s)\right).
\end{align*}
Then, we observe that $O(s), P(s)$, and $\xb(x,v)$ lie on a common circle with center $C(s)$, and that $|\xb(x,v)-C(s)|=|P(s)-C(s)|$ holds. Among all $C(s)$ for $s \in [0,t]$, we choose a point $\mathbf{C}$ such that
\begin{align*}
    \max_{s \in [0,t]}|\xb(x,v)-C(s)| = |\xb(x,v)-\mathbf{C}|:=L.
\end{align*}
(See Figure \ref{fig_int_s}.) 

We consider two circles: one centered at $C(s)$,
passing through $\xb(x,v)$ and $P(s)$, and the other centered at $\mathbf{C}$, passing through $\xb(x,v)$ with radius $L$. Since the smaller circle is entirely contained within the larger circle, we obtain 
\begin{align*}
    |\xb(x,v)-\mathbf{C}| \geq  |P(s)-\mathbf{C}|.
\end{align*}
From now on, let us assume $\mathbf{C}=(0,0,0)$ and $\hat{n}(\xb(x,v)) = (0,0,1)$ without loss of generality. 
\begin{figure} [t] 
\centering
\includegraphics[width=8cm]{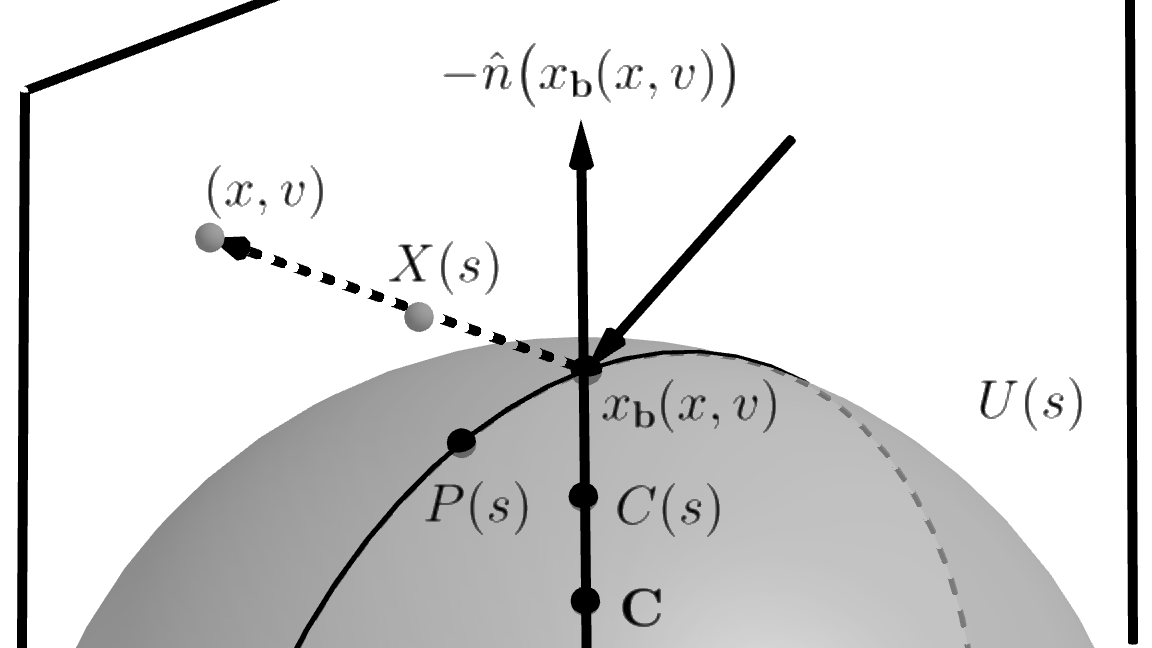}
\caption{Given $x \in \O$ and $v \in\mathbb{R}^3$, we find the point $\mathbf{C}$.} \label{fig_int_s}
\end{figure} 
By the above inequality, we obtain
    \begin{align} \label{x-y>x-xb}
    \min_{y \in \partial \O}|X(s)-y|=|X(s)-P(s)|
    \geq |X(s)|-|P(s)|\geq|X(s)|-|\xb(x,v)|.
    \end{align} 

Because $\hat{n}(\xb(x,v))\cdot v \leq 0$, we obtain 
\begin{align*}
    \xb(x,v) \cdot v = (\xb(x,v)-\mathbf{C})\cdot v
    = - L \hat{n}(\xb(x,v))\cdot v \geq 0
\end{align*}
and
    \begin{align} \label{X2-P2}
       |X(s)|^2-|\xb(x,v)|^2 = |\xb(x,v)+v(s-t_1(t,x,v))|^2-|\xb(x,v)|^2 \geq |v|^2(s-t_1(t,x,v))^2
    \end{align} 
for $s \in [t_1(t,x,v), t]$. Because $\hat{n}(\xb(x,v))\cdot (-R_{\xb(x,v)}v) \leq 0$, we obtain
    \begin{align*} 
     &\xb(x,v) \cdot (-R_{\xb(x,v)}v)
    = L \hat{n}(\xb(x,v))\cdot R_{\xb(x,v)}v \geq 0
    \end{align*} 
and
\begin{align} \label{X2-P2'}
    \begin{split}
       |X(s)|^2-|\xb(x,v)|^2 &= |\xb(x,v)-R_{\xb(x,v)}v(t_1(t,x,v)-s)|^2-|\xb(x,v)|^2 \\
       &\geq |v|^2(t_1(t,x,v)-s)^2
   \end{split}
   \end{align} 
for $s \in [0,t_1(t,x,v)]$. 

When $d(X(s),\partial\O) \geq 1$, we directly obtain
\begin{align*}
\bigintsss_0^t e^{-\varpi \langle v \rangle^{2} (t-s)}\ln \left(  1+ \frac{1}{d(X(s),\partial\O)} \right)\mathbf{1}_{\{d(X(s),\partial\O) \geq 1\}} ds \lesssim \frac{1}{\varpi \langle v \rangle^2}.
\end{align*}
Therefore, we consider $s$ such that $d(X(s),\partial\O) \leq 1$. This implies that $|X(s)|\lesssim 1$. Then, by combining \eqref{x-y>x-xb} with \eqref{X2-P2} and \eqref{X2-P2'}, 
   \begin{align}\label{min x-y}
   \begin{split}
       \min_{y \in \partial \O}|X(s)-y| 
       \geq \frac{|X(s)|^2-|\xb(x,v)|^2}{|X(s)|+|\xb(x,v)|} 
       \gtrsim |v|^2(s-t_1(t,x,v))^2.
   \end{split}
   \end{align}

   Next, we divide the range of $s$ into three cases: (1) $s \in [0, t_1(t,x,v)]$, (2) $s \in [t_1(t,x,v), (t+t_1(t,x,v))/2]$, and (3) $s \in [(t+t_1(t,x,v))/2, t]$.
 
   \textbf{(1)} We consider $s \in [0, t_1(t,x,v)]$. By using \eqref{min x-y} and $t_1(t,x,v)\leq t$, we obtain
   \begin{align}
        &\bigintsss_0^{t_1(t,x,v)} e^{-\varpi \langle v \rangle^{2} (t-s)}\ln \left(  1+ \frac{1}{d(X(s),\partial\O)} \right) \mathbf{1}_{\{d(X(s),\partial\O) \leq 1\}} ds \notag \\
        &\lesssim \bigintsss_0^{t_1(t,x,v)}
        e^{-\varpi \langle v \rangle^{2} (t_1(t,x,v)-s)}\ln \left(  1+ \frac{1}{|v|(t_1(t,x,v)-s)} \right) ds. \label{<1_0t1}
   \end{align} 
   By changing the variable $s \rightarrow u=\varpi \langle v \rangle^2 (t_1(t,x,v)-s)$, we obtain
   \begin{align*}
         \eqref{<1_0t1} &\leq
       \frac{1}{\varpi \langle v \rangle^2}
       \bigintsss_0^{\varpi \langle v \rangle^2 t_1(t,x,v)} e^{-u} \ln
       \left(1 +\frac{\varpi \langle v \rangle^2}{|v|}\frac{1}{u}  \right) du \\
       &\lesssim  \frac{1}{\varpi \langle v \rangle^2} \ln
       \left(1 + \varpi\frac{\langle v \rangle^{2}}{|v|} \right)\int_0^{\infty} e^{-u} du +  \frac{1}{\varpi \langle v \rangle^2} \int_0^{\infty} e^{-u} \ln \left( 1+\frac{1}{u}\right) du \\
      &\lesssim \frac{1}{\varpi \langle v \rangle^2}\left[\ln\left(1+\frac{\varpi\langle v \rangle^2}{|v|}\right)+1 \right].
   \end{align*} 
   
   \textbf{(2)} We consider $t_1(t,x,v) \leq s \leq (t+t_1(t,x,v))/2$. By \eqref{min x-y} and $0 \leq s-t_1(t,x,v) \leq t-s$, we obtain
   \begin{align*}
        &\bigintsss_{t_1(t,x,v)}^{(t+t_1(t,x,v))/2} e^{-\varpi \langle v \rangle^{2} (t-s)}\ln \left(  1+ \frac{1}{d(X(s),\partial\O)} \right) \mathbf{1}_{\{d(X(s),\partial\O) \leq 1\}} ds \notag \\
        &\lesssim \bigintsss_{t_1(t,x,v)}^{(t+t_1(t,x,v))/2}
        e^{-\varpi \langle v \rangle^{2} (s-t_1(t,x,v))}\ln \left(  1+ \frac{1}{|v|(s-t_1(t,x,v))} \right) ds. 
   \end{align*} 
   
   \textbf{(3)} We consider $(t+t_1(t,x,v))/2 \leq s \leq t$. By \eqref{min x-y} and $0 \leq t-s \leq s-t_1(t,x,v)$, we obtain
   \begin{align*}
       &\bigintsss_{(t+t_1(t,x,v))/2}^{t} e^{-\varpi \langle v \rangle^{2} (t-s)}\ln \left(  1+ \frac{1}{d(X(s),\partial\O)} \right) \mathbf{1}_{\{d(X(s),\partial\O) \leq 1\}} ds \notag \\
        &\lesssim \bigintsss_{(t+t_1(t,x,v))/2}^{t}
        e^{-\varpi \langle v \rangle^{2} (t-s)}\ln \left(  1+ \frac{1}{|v|(t-s)} \right) ds.
   \end{align*}
   The last terms of the inequalities in (2) and (3) are similar to \eqref{<1_0t1}, and we can use the same arguments as in (1). In particular, when $t_1(t,x,v) \leq 0$, we have $\min_{y \in \partial \O}|X(s)-y| 
       \gtrsim |v|^2 s^2$ by \eqref{X2-P2}, and we can use this inequality in place of \eqref{min x-y}.
   
   \vspace{3mm}
   Next, we assume $\tb(x,v) = \infty$. There exists $\alpha $ such that
   \begin{align*}
       \min_{s \in [0.t]} d(X(s),\partial\O) = d(X(\alpha) , \partial\O).
   \end{align*}
   Then, $P(\alpha) \in \partial \O$ satisfies
   \begin{align*}
       d(X(\alpha), \partial\O) = |X(\alpha)-P(\alpha)|
       \text{\quad and \quad} \hat{n}(P(\alpha)) =\widehat{ P(\alpha)-X(\alpha) } \,\bot\, v. 
   \end{align*}
Now, we follow the arguments in the first and second paragraphs, replacing $\xb(x,v)$ with $P(\alpha)$. Then, we can find $\mathbf{C}$ such that
  \begin{align*}
      |P(\alpha)-\mathbf{C}| \geq |P(s)-\mathbf{C}|
       \text{\quad and \quad} 
       \hat{n}(P(\alpha)) =\widehat{\mathbf{C}- P(\alpha)}
  \end{align*}
for $0\leq s \leq t$. Moreover, we get $|X(\alpha)- \mathbf{C}| \geq |P(\alpha)-\mathbf{C}|$. Set $\mathbf{C}=(0,0,0)$ and $\hat{n}(P(\alpha))=(0,0,1)$. Because $X(\alpha) \cdot v =0$, we obtain
\begin{align*}
    |X(s)|^2 - |X(\alpha)|^2
    =|X(\alpha)+v(s-\alpha)|^2- |X(\alpha)|^2 =|v|^2(s-\alpha)^2.
\end{align*}
  Then
\begin{align*}
     \min_{y \in \partial \O}|X(s)-y| 
       \geq |X(s)|-|P(\alpha)| 
       \geq |X(s)|-|X(\alpha)|
       \gtrsim |X(s)|^2 -|X(\alpha)|^2
       \geq |v|^2(s-\alpha)^2.
\end{align*}
Finally, when $\alpha \in [0,t]$, we use the above inequality in place of \eqref{min x-y}, and replace $t_1(t,x,v)$ with $\alpha$. We then apply the same arguments as those in (1), (2), and (3) above. In particular, when $\alpha \leq 0$, we have $\min_{y \in \partial \O}|X(s)-y| 
       \gtrsim |v|^2 s^2$, and
when $\alpha \geq t$, we have $\min_{y \in \partial \O}|X(s)-y| 
       \gtrsim |v|^2 (t-s)^2$ for $0 \leq s \leq t$.
\end{proof} 

When \( x \) is away from the boundary, the integral considered in the previous lemma remains bounded as \( v \rightarrow 0 \).

\begin{lemma}\label{lem:ds_2}
Let $0<t \leq 1, \; x \in \O,\;v \in \mathbb{R}^3$ and $\varpi>1$ be given. Assume $d(x,\partial\O) \geq \epsilon$ for $0<\epsilon<1$. Then, we obtain
    \begin{align*} 
      \bigintsss_0^t e^{-\varpi \langle v \rangle^{2} (t-s)}\ln \left(  1+ \frac{1}{d(X(s;t,x,v),\partial\O)} \right) ds \lesssim \frac{1}{\sqrt{\varpi}\langle v \rangle }
      \ln\left(1+\frac{1}{\epsilon}\right).
    \end{align*} 
\end{lemma}
\begin{proof}
 Let us denote $X(s):=X(s;t;x,v)$. We divide the cases into $ t_1(t,x,v) \geq 0$ and $t_1(t,x,v) \leq 0$.

\textbf{(1)} For $0 \leq t_1(t,x,v) \leq t$, we have
    \begin{align*} 
        |v| \geq \frac{d(x,\partial\O)}{t-t_1(t,x,v)} \geq \frac{\epsilon}{t} \geq \epsilon.
    \end{align*}
    By Lemma \ref{lem:ds} and the above inequality,
    \begin{align} \label{ln<1/v}
    \begin{split}
        &\bigintsss_0^{t} e^{-\varpi \langle v \rangle^2 (t-s)}\ln \left(  1+ \frac{1}{d(X(s),\partial\O)}  \right) ds 
        \lesssim \frac{1}{\varpi \langle v \rangle^2}\left[\ln\left(1+\frac{\varpi\langle v \rangle^2}{|v|}\right)+1 \right] \\
        &\lesssim 
        \frac{1}{\varpi \langle v \rangle^2}      \left[\ln\left(1+\varpi\langle v \rangle^2\right)+\ln\left(1+\frac{1}{|v|}\right)+1 \right] \\
        &\lesssim
        \frac{1}{\varpi \langle v \rangle^2}      \left[\sqrt{\varpi}\langle v \rangle +\ln\left(1+\frac{1}{\epsilon}\right)+1 \right] 
        \lesssim \frac{1}{\sqrt{\varpi}\langle v \rangle }\ln\left(1+\frac{1}{\epsilon}\right).
   \end{split}
   \end{align}     

   \textbf{(2)} We assume $t_1(t,x,v) \leq 0$. For each $s$, there is $P(s) \in \partial \O$ such that $d(X(s),\partial\O)=|X(s)-P(s)|$. When $d(X(s),\partial\O) \leq \epsilon/2$, we have
   \begin{align*}
       |v|(t-s)=|x-X(s)| \geq |x-P(s)|-|X(s)-P(s)| \geq d(x,\partial\O) -d(X(s),\partial\O) \geq \frac{\epsilon}{2}
   \end{align*}
   for $0 \leq s \leq t$. This implies that
   \begin{align*}
       |v| \geq \frac{\epsilon}{2(t-s)} \geq \frac{\epsilon}{2t} \geq \frac{\epsilon}{2}.
   \end{align*}
  Therefore, we obtain \eqref{ln<1/v}, where the range of $s$ is restricted to ${d(X(s), \partial\O) \leq \epsilon/2}$. On the other hand, we have
  \begin{align*} 
        \bigintsss_0^{t} e^{-\varpi \langle v \rangle^2 (t-s)}\ln \left(  1+ \frac{1}{d(X(s),\partial\O)}  \right) \mathbf{1}_{\{d(X(s),\partial\O) \geq \epsilon/2\}} ds \leq \ln \left( 1 + \frac{2}{\epsilon} \right) \frac{1}{\varpi \langle v \rangle^2}.
   \end{align*}    
  By combining the results of (1) and (2), we obtain the lemma. \\
\end{proof}

\section{Estimates of \texorpdfstring{$\mathfrak{X}$}{} and \texorpdfstring{$\mathfrak{V}$}{}}

In Section 6, we estimate $\mathfrak{X}$ and $\mathfrak{V}$, as defined in \eqref{def:iter}. These estimates play a key role in the proof of Theorem \ref{thm:Holder0.5}. Before proceeding, we decompose $f(t,x,v)$ at the points $(x,v)$ and $(\bx, \bv)$. For $x, \bx \in \O$, $v, \bv \in \mathbb{R}^3$, and $t > 0$, we denote
   \begin{align} \label{nota:bX,bV}
   \begin{split}
       &(X(s),V(s)):=(X(s;t,x,v),V(s;t,x,v)),\\
       &(\overline{X}(s),\overline{V}(s)):=(X(s;t,\bx,\bv),V(s;t,\bx,\bv))\text{\quad for \quad} 0 \leq s \leq t.
   \end{split}
   \end{align} 
Applying the triangle inequality to \eqref{f_expan}, we obtain
\begin{equation}  \notag
\begin{split}
& |f(t,x,v)-f(t,\bar{x}, \bar{v})|  \\
&\leq   
e^{- \int^t_ 0 \nu(f) (\tau, X(\tau ), V(\tau )) d \tau}
\left|f(0,X(0 ), V(0 ))-f(0,\overline{X}(0 ), \overline{V}(0 ))\right| 
\\
&\quad + \int^t_0 
e^{- \int^t_ s \nu(f) (\tau, X(\tau ), V(\tau )) d \tau} 	
\left|\Gamma_{\text{gain}}(f,f)(s,X(s ), V(s ))
-\Gamma_{\text{gain}}(f,f)(s,\overline{X}(s ), \overline{V}(s ))\right|  ds 
\\
&\quad +  
\left| e^{- \int^t_ 0 \nu(f) (\tau, X(\tau ), V(\tau )) d \tau}  -  
e^{- \int^t_0 \nu(f) (\tau, \overline{X}(\tau ), \overline{V}(\tau )) d \tau} \right|	
f(0, \overline{X}(0), \overline{V}(0))
\\
&\quad +\int^t_0
\left| e^{- \int^t_ s \nu(f) (\tau, X(\tau ), V(\tau )) d \tau}  -  
e^{- \int^t_s \nu(f) (\tau, \overline{X}(\tau ), \overline{V}(\tau )) d \tau} \right|
\left|	\Gamma_{\text{gain}}(f,f)(s,\overline{X}(s ), \overline{V}(s ))\right| ds.
\end{split}
\end{equation}
Since $|e^{-a}-e^{-b}| \leq |a-b|$ for $a\geq b \geq 0$ and \eqref{gamma_upper}, we obtain  {\eqref{basic f-f}--\eqref{basic f-f4}}.

Next, we establish a very simple lemma, which will be used in Lemmas \ref{pro:X<1} and \ref{pro:V<1} to control the term $e^{\varpi s \langle u \rangle^2}$; see \eqref{exp:v} and \eqref{exp_est_V}.

\begin{lemma} \label{lem:ws}
Let $u,v \in \mathbb{R}^3$ be given. For $0 \leq \varpi s \leq \min\{c/4, 1/8\}$, we obtain
\begin{align*}
     e^{-\varpi s \left( \langle v \rangle^2-\langle u \rangle^2\right)} \leq e^{
     \frac{c}{2}|u-v|^2} 
     \text{\quad and \quad} 
     e^{-\varpi s \left( \langle v \rangle^2-\langle u \rangle^2\right)} \leq  e^{\frac{1}{8}|u|^2}. 
\end{align*}
\end{lemma}
\begin{proof} 
Using the triangle inequality, we obtain
    \begin{equation*}
        \varpi s \left(|u|^2-|v|^2 \right) \leq 2\varpi s|u-v|^2 \leq \frac{c}{2}|u-v|^2,
    \end{equation*} and thus, we get the first inequality. 
\end{proof}

First, we estimate $\mathfrak{X}$. In its definition, we insert the weight
\[
\max\left\{\frac{1}{\langle v\rangle}, (t-s)^{\frac{1}{2}}\right\}
\]
into the integral. This produces the factor $\langle v \rangle^{-1}$ in \eqref{exp:v}, which is the reason we include this weight in the definition of $\mathfrak{X}$.

\begin{lemma}[Estimate for $ \mathfrak{X}$] \label{pro:X<1}
   Recall $T^*>0$ from Lemma \ref{lem:loc}.  There exist constants $0<T<\min\{T^*,1\}$ and $\varpi > 1$ such that
\begin{align*}
    \sup_{0 \leq t \leq T}\mathfrak{X}(t,\varpi;\epsilon)\lesssim_{\epsilon, \vartheta_0} \frac{1}{\sqrt{\varpi} }\mathcal{P}_3(\|w_0f_0\|_{\infty} ) \left(\sup_{0 \leq t \leq T} \mathfrak{X}(t,\varpi;\epsilon)+\sup_{0 \leq t \leq T}\mathfrak{V}(t,\varpi;\epsilon)+ \mathbf{A}_{\frac{1}{2}}(f_0)+1 \right).
\end{align*} 
and $\varpi T \leq \min\{c/4, 1/8\} $ for any $0<\epsilon<1$.
\end{lemma}
\begin{proof}
     Let $x_1, \bx_1 \in \O, v_1 \in \mathbb{R}^3$ and $0<t_1<  \min\{T^*,1\}$ be given. 
    In \eqref{basic f-f}-\eqref{basic f-f4}, we replace $x$ with $x_1$, and both $v, \bv$ with $v_1$. We apply Lemma \ref{lem:s=0_spec_x} to \eqref{basic f-f1}, Lemma \ref{lem:ga_X} to \eqref{basic f-f2} and \eqref{basic f-f4}. Then
   \begin{align} \label{iter_x,bx}
   \begin{split}
       &\frac{|f(t_1,x_1,v_1)-f(t_1,\bx_1, v_1)|}{|x_1-\bx_1|}\mathcal{P}_3^{-1}(\|w_0f_0\|_{\infty} )\\
         &\lesssim_{\vartheta_0} \mathbf{1}_{\{|x_1-\bx_1|\leq 1\}
       } e^{\varpi \langle v_1 \rangle^{2} t_1}
        \left(\mathfrak{X}(t_1,\varpi;\epsilon)+\mathfrak{V}(t_1,\varpi;\epsilon)+\mathbf{A}_{\frac{1}{2}}(f_0)+1\right) \\
        &\quad \times 
            \left(\frac{1}{|v_1|}+|v_1|+(1+|v_1|^2)\mathcal{T}_{sp}(x_1,\tx_1,v_1;t_1)\right)+ \mathbf{1}_{\{|x_1-\bx_1|\geq 1\}}.
   \end{split}
   \end{align} 
   
   For $x,\bx \in \O, \,v,\bv \in \mathbb{R}^3$ and $0\leq s \leq t<1$, we define
   \begin{align*}
           (X(s),V(s)):=(X(s;t,x,v),V(s;t,x,v)),
           \quad \bX(s) := X(s;t,\bx,\bv)
   \end{align*}
   for $|(x,v)-(\bx,\bv)|\leq 1$. For $u \in \R^3$, we define
   \begin{align*}
       \widetilde{X}(s) := \tx(X(s), \bX(s), u).
   \end{align*}
To transform the (LHS) of \eqref{iter_x,bx} into the form 
    \begin{align} \label{iter_X,bX_du}
    \begin{split}
       e^{-\varpi \langle v \rangle^2 t }&\bigintsss_0^t \max\left\{\frac{1}{\langle v\rangle},(t-s)^{\frac{1}{2}}\right\} \\
       &\quad\bigintsss_{\mathbb{R}^3}
       \left( \frac{e^{-c|u-V(s)|^2}}{|u-V(s)|}+|u-V(s)|e^{-\frac{1}{4}|u|^2}    \right)\frac{|f(s,X(s),u)-f(s,\overline{X}(s), u)|}{|X(s)-\overline{X}(s)|} du ds,
    \end{split}
    \end{align}
we substitute $t_1,x_1,\bx_1$, and $v_1$ with $s, X(s), \overline{X}(s)$, and $u$ in \eqref{iter_x,bx}, respectively, and multiply both sides of \eqref{iter_x,bx} by suitable factors and integrate with respect to $u$. Then, we obtain
   \begin{align} \label{I123}
        (\ref{iter_X,bX_du})\mathcal{P}_3^{-1}(\|w_0f_0\|_{\infty} ) \lesssim_{\vartheta_0}    
        \left(\sup_{0 \leq t \leq T}\mathfrak{X}(t,\varpi;\epsilon)+\sup_{0 \leq t \leq T}\mathfrak{V}(t,\varpi;\epsilon)+\mathbf{A}_{\frac{1}{2}}(f_0)+1\right) I_1+I_2
   \end{align} 
   for $0 \leq t \leq T <  \min\{T^*,1\}$. Here, $I_1$ and $I_2$ are given by
   \begin{align*} 
       I_1 &= e^{-\varpi \langle v \rangle^2 t}\bigintsss_0^t \max\left\{\frac{1}{\langle v\rangle},(t-s)^{\frac{1}{2}}\right\}\mathbf{1}_{\{|X(s)-\bX(s)|\leq 1\}} \\
       &\times \bigintsss_{\mathbb{R}^3}  e^{\varpi \langle u \rangle^2 s}\left( \frac{e^{-c|u-V(s)|^2}}{|u-V(s)|}+|u-V(s)|e^{-\frac{1}{4}|u|^2}    \right) \left(\frac{1}{|u|}+|u|+(1+|u|^2)\mathcal{T}_{sp}(X(s),\widetilde{X}(s),u;s)\right)  du ds
   \end{align*} 
   and
   \begin{align*}
       I_2 &= e^{-\varpi \langle v \rangle^2 t}\bigintsss_0^t \max\left\{\frac{1}{\langle v\rangle},(t-s)^{\frac{1}{2}}\right\}\bigintsss_{\mathbb{R}^3}   \left( \frac{e^{-c|u-V(s)|^2}}{|u-V(s)|}+|u-V(s)|e^{-\frac{1}{4}|u|^2}    \right)  du ds.
   \end{align*}

\vspace{3mm}
We first estimate $I_1$. Using Lemma \ref{lem:ws} and $|V(s)|=v$, we obtain 
\begin{align} \label{exp:v}
\begin{split}
    &e^{-\varpi \langle v \rangle^2 t}e^{\varpi \langle u \rangle^2 s} \max\left\{\frac{1}{\langle v\rangle},(t-s)^{\frac{1}{2}}\right\} = e^{-\varpi \langle V(s) \rangle^2 t}e^{\varpi \langle u \rangle^2 s} \max\left\{\frac{1}{\langle V(s)\rangle},(t-s)^{\frac{1}{2}}\right\}  \\
    &\leq e^{-\varpi \langle V(s) \rangle^2 (t-s)} \min \left\{e^{
     \frac{c}{2}|u-V(s)|^2}, e^{\frac{1}{8}|u|^2} \right\} \max\left\{\frac{1}{\langle V(s)\rangle},(t-s)^{\frac{1}{2}}\right\} \\
     &\leq e^{-\frac{\varpi}{2} \langle v \rangle^2 (t-s)}\min \left\{e^{
     \frac{c}{2}|u-V(s)|^2}, e^{\frac{1}{8}|u|^2} \right\}\frac{2}{\langle v\rangle}
\end{split}
\end{align}
for $\varpi T \leq \min\{c/4, 1/8\}$. In the last inequality, we use the fact that $x e^{-x} \leq 1$ and $\varpi > 1$.

We apply Corollary \ref{cor:ave}-(1) to $\mathcal{T}_{sp}(X(s),\widetilde{X}(s),u;s)$ and use \eqref{exp:v}. Then 
\begin{align*}
    I_1 &\leq \frac{2}{\langle v\rangle}\bigintsss_0^t e^{-\frac{\varpi}{2} \langle v \rangle^2 (t-s)} \mathbf{1}_{\{|X(s)-\bX(s)|\leq 1\}} \bigintsss_{\mathbb{R}^3} 
 \left( \frac{e^{-\frac{c}{2}|u-V(s)|^2}}{|u-V(s)|}+|u-V(s)|e^{-\frac{1}{8}|u|^2}    \right) \left(\frac{1}{|u|}+|u|\right) \\
 &\quad \times \left[1+ \left(\frac{\mathbf{1}_{\{\tb(X(s),u) < \infty\}}}{|\hat{u}\cdot\nabla\xi(\xb(X(s), u)) |} + \frac{\mathbf{1}_{\{\tb(\bX(s),u) < \infty\}}}{|\hat{u}\cdot\nabla\xi(\xb(\bX(s), u)) |}\right)\mathbf{1}_{\{d(X(s),\partial\O) \lesssim  \langle u \rangle \}}\right] du ds
\end{align*}
since $|\hat{u}\cdot\nabla\xi(\xb(\tX(s), u)) |=|\hat{u}\cdot\nabla\xi(\xb(\bX(s), u))|$.

By applying Lemma \ref{lem:int_sing} with $k = -1, \,1$ and $x = X(s), \,\bX(s)$ to the above inequality, we obtain 
\begin{align*}
    I_1 &\lesssim \langle v\rangle\bigintsss_0^t e^{-\frac{\varpi}{2} \langle v \rangle^2 (t-s)}\mathbf{1}_{\{|X(s)-\bX(s)|\leq 1\}}
    \left[ \ln\left( 1+ \frac{1}{d(X(s),\partial\O)}\right)+\ln\left( 1+ \frac{1}{d(\bX(s),\partial\O)}\right)+1  \right] ds \\
    &\leq \langle v\rangle\bigintsss_0^t e^{-\frac{\varpi}{2} \langle v \rangle^2 (t-s)}
    \left[ \ln\left( 1+ \frac{1}{d(X(s),\partial\O)}\right)+1  \right] ds \\
    &\quad + 4 \langle \bv\rangle\bigintsss_0^t e^{-\frac{\varpi}{32} \langle \bv \rangle^2 (t-s)} 
    \ln\left( 1+ \frac{1}{d(\bX(s),\partial\O)}\right) ds
\end{align*}
since $|v-\bv| \leq 1$. 

We consider the first term on the (RHS) of the above inequality. When $d(x,\partial\O) \leq \epsilon$, we apply Lemma \ref{lem:ds}. When $d(x,\partial\O) \geq \epsilon$, we apply Lemma \ref{lem:ds_2}. For $0<\epsilon<1$, we obtain
\begin{align*}
    &\langle v\rangle\bigintsss_0^t e^{-\frac{\varpi}{2} \langle v \rangle^2 (t-s)}
    \left[ \ln\left( 1+ \frac{1}{d(X(s),\partial\O)}\right)+1  \right] ds\\ &\lesssim \frac{1}{\varpi \langle v \rangle}\left[\ln\left(1+\frac{\varpi\langle v \rangle^2}{|v|}\right)+1 \right]\mathbf{1}_{\{d(x,\partial\O) \leq \epsilon\}}+\ln\left(1+ \frac{1}{\epsilon} \right)\frac{1}{\sqrt{\varpi}}\mathbf{1}_{\{d(x,\partial\O) \geq \epsilon\}}\\
    &\leq    \frac{1}{\varpi \langle v \rangle}\left[ \ln(1+\varpi\langle v \rangle^2)+\ln\left(1+\frac{1}{|v|} \right) +1\right]\mathbf{1}_{\{d(x,\partial\O) \leq \epsilon\}}+\ln\left(1+ \frac{1}{\epsilon} \right)\frac{1}{\sqrt{\varpi}}\mathbf{1}_{\{d(x,\partial\O) \geq \epsilon\}}\\
    &\leq \frac{1}{\varpi \langle v \rangle}\left[ \sqrt{\varpi}\langle v \rangle+\ln\left(1+\frac{1}{|v|} \right) +1\right]\mathbf{1}_{\{d(x,\partial\O) \leq \epsilon\}}+\ln\left(1+ \frac{1}{\epsilon} \right)\frac{1}{\sqrt{\varpi}}\mathbf{1}_{\{d(x,\partial\O) \geq \epsilon\}}\\
    &\lesssim \ln\left(1+ \frac{1}{\epsilon} \right)\frac{1}{\sqrt{\varpi}}
    \left[ \ln \left( 1+\frac{1}{|v|} \right)\mathbf{1}_{\{d(x,\partial\O) \leq \epsilon\}}+1\right] \\
    &= \ln\left(1+ \frac{1}{\epsilon}\right)\frac{1}{\sqrt{\varpi}} G(x,v;\epsilon),
\end{align*}
since $\ln(1+xy) \leq \ln(1+x)+\ln(1+y)$ and $\ln(1+x) \leq \sqrt{x}$ for $x>0$. Therefore, we obtain
\begin{align*}
    I_1 \lesssim \frac{1}{\sqrt{\varpi}}\ln\left(1+ \frac{1}{\epsilon}\right) \left(G(x,v;\epsilon)+G(\bx,\bv;\epsilon)\right).
\end{align*}

On the other hand, we easily compute
\begin{align*}
    I_2 \lesssim \bigintsss_0^t e^{-\varpi\langle v \rangle^2(t-s)}\bigintsss_{\mathbb{R}^3}   \left( \frac{e^{-c|u-V(s)|^2}}{|u-V(s)|}+|u-V(s)|e^{-\frac{1}{4}|u|^2}    \right)  du ds \lesssim \frac{1}{\varpi \langle v \rangle}.
\end{align*}
Hence, the upper bounds of $I_1$ and $I_2$ lead to
\begin{align*}
      (\ref{iter_X,bX_du}) 
      &\lesssim_{\vartheta_0} \frac{1}{\sqrt{\varpi}}\mathcal{P}_3(\|w_0f_0\|_{\infty} ) \ln\left(1+ \frac{1}{\epsilon} \right) 
      \\&\quad \times \Bigg[\left( G(x,v;\epsilon)+G(\bx,\bv;\epsilon) \right) 
         \left(\sup_{0 \leq t \leq T}\mathfrak{X}(t,\varpi;\epsilon)+\sup_{0 \leq t \leq T}\mathfrak{V}(t,\varpi;\epsilon)+\mathbf{A}_{\frac{1}{2}}(f_0)+1\right) +\frac{1}{\langle v \rangle} \Bigg]
\end{align*}
from \eqref{I123}. 

Therefore, we multiply both sides of the above inequality by $(G(x,v;\epsilon)+G(\bx,\bv;\epsilon))^{-1}(\leq 1/2)$, and then
\begin{align*}
      &(G(x,v;\epsilon)+G(\bx,\bv;\epsilon))^{-1}(\ref{iter_X,bX_du}) \\
      &\lesssim_{\vartheta_0, \epsilon}\frac{1}{\sqrt{\varpi}}\mathcal{P}_3(\|w_0f_0\|_{\infty} ) \left(\sup_{0 \leq t \leq T}\mathfrak{X}(t,\varpi;\epsilon)+\sup_{0 \leq t \leq T}\mathfrak{V}(t,\varpi;\epsilon)+\mathbf{A}_{\frac{1}{2}}(f_0)+1\right).
\end{align*}
Finally, we take the supremum over $(x,\bx,v,\bv)\in \O\times \O\times \R^{3}\times \R^{3}$ and $0 \leq t \leq T$, and obtain the lemma.
\end{proof}

Next, we estimate $\mathfrak{V}$. When estimating the difference
\[
|f(t_1,x_1,v_1+u) - f(t_1,x_1,\bar{v}_1+u)|,
\]
we divide the domain into two cases: \( d(x_1, \Omega) \lesssim 1 \) and \( 1 \lesssim d(x_1, \Omega) \lesssim \langle v_1 + u \rangle \), as in \eqref{est_V_txv}.  
This division corresponds to the terms $J_1$ and $J_2$ in the proof. The reason for this splitting is to control the growth order in $\langle v \rangle$.

\begin{lemma}[Estimate for $ \mathfrak{V}$ ]\label{pro:V<1}
Recall $T_1>0$ and $\varpi_1>1$ from Proposition \ref{pro:H}.
   There exist constants $0<T<\min\{T_1,1\}$ and $1<\varpi_1 <\varpi$ such that
\begin{align*}
     \sup_{0 \leq t \leq T}\mathfrak{V}(t,\varpi;\epsilon) \lesssim_{\vartheta_0,\epsilon} \frac{1}{\sqrt{\varpi} }\mathcal{P}_3(\|w_0f_0\|_{\infty} ) \left(\sup_{0 \leq t \leq T}\mathfrak{X}(t,\varpi;\epsilon)+\sup_{0 \leq t \leq T}\mathfrak{V}(t,\varpi;\epsilon)+ \mathbf{A}_{\frac{1}{2}}(f_0)+1\right).
\end{align*} 
and $\varpi T \leq \min\{c/4, 1/8\} $ for any $0<\epsilon<1$.
\end{lemma}
\begin{proof}
  Let \( x_1 \in \Omega \), \( v_1, \bar{v}_1, u \in \mathbb{R}^3 \). Let \( 0 < t_1 < \min\{T_1, 1\} \) and \( 1 < \varpi_1 < \varpi \) be given.
 In \eqref{basic f-f}-\eqref{basic f-f4}, we replace $x$ with $x_1$ and $\bx$ with $\bx_1$ and $\z$ with $u$. We apply Lemma \ref{lem:s=0_spec_v} to \eqref{basic f-f1}, and apply Lemma \ref{lem:ga_V} to \eqref{basic f-f2} and \eqref{basic f-f4}. Then
    \begin{align} \label{iter_v,bv}
   \begin{split}
       &\frac{|f(t_1,x_1,v_1+u)-f(t_1,x_1, \bv_1+u)|}{|v_1-\bv_1|} \mathcal{P}_3^{-1}(\|w_0f_0\|_{\infty} )\\
       &\lesssim_{\vartheta_0}  \mathbf{1}_{\{|v_1-\bv_1|\geq 1\}}+\mathbf{1}_{\{|v_1-\bv_1|\leq 1\}}  e^{\varpi \langle v_1+u \rangle^{2} t_1} \left(\mathfrak{X}(t_1,\varpi;\epsilon)+\mathfrak{V} (t_1,\varpi;\epsilon)+\mathbf{A}_{\frac{1}{2}}(f_0)+1\right)  \\
       &\; \times  \left[ \frac{1}{|v_1+u|}+\frac{1}{|\bv_1+u|}+|v_1+u|+\left(\min \left\{  \frac{1}{|v_1+u|^{\frac{3}{2}}}, \frac{1}{|\bv_1+u|^\frac{3}{2}}\right\}+1\right)|v_1+u|^2\mathcal{T}_{vel}(x_1, v_1, \tv_1, u;t_1)\right].
   \end{split}
   \end{align} 

    Next, we apply Corollary \ref{cor:ave}-(2) to $\mathcal{T}_{vel}(x_1, v_1, \tv_1, u;t_1)$. Then, we obtain
   \begin{align} \label{est_V_txv}
   \begin{split}
       &\frac{|f(t_1,x_1,v_1+u)-f(t_1,x_1, \bv_1+u)|}{|v_1-\bv_1|} \mathcal{P}_3^{-1}(\|w_0f_0\|_{\infty} )\\
       &\lesssim_{\vartheta_0}  \mathbf{1}_{\{|v_1-\bv_1|\leq 1\}}e^{\varpi \langle v_1+u \rangle^{2} t_1}
          \left(\mathfrak{X}(t_1,\varpi;\epsilon)+\mathfrak{V}(t_1,\varpi;\epsilon)+\mathbf{A}_{\frac{1}{2}}(f_0)+1\right) \\
           &\quad \times \left(     \frac{\mathbf{1}_{\{\tb(x_1,v_1+u)<+\infty\}}}{|\widehat{v_1+u}\cdot\nabla\xi(\xb(x_1, v_1+u)) |}+ \frac{\mathbf{1}_{\{\tb(x_1,\bv_1+u)<+\infty\}}}{|\widehat{\bv_1+u}\cdot\nabla\xi(\xb(x_1, \bv_1+u)) |} \right) \\
           &\quad \times 
           \Bigg[  \left(\min \left\{  \frac{1}{|v_1+u|^{\frac{3}{2}}}, \frac{1}{|\bv_1+u|^\frac{3}{2}}\right\}+1\right) \mathbf{1}_{\{ d(x_1,\partial\O) \lesssim 1 \}} \\
           &\quad  +\left(\min \left\{  \frac{1}{|v_1+u|^{\frac{3}{2}}}, \frac{1}{|\bv_1+u|^\frac{3}{2}}\right\}+|v_1+u|\right) \mathbf{1}_{\{ 1 \lesssim d(x_1,\partial\O) \lesssim \langle v_1+u \rangle \}}\Bigg] \\
           &\quad + \mathbf{1}_{\{|v_1-\bv_1|\leq 1\}}e^{\varpi \langle v_1+u \rangle^{2} t_1} \left(\mathfrak{X}(t_1,\varpi;\epsilon)+\mathfrak{V}(t_1,\varpi;\epsilon)+\mathbf{A}_{\frac{1}{2}}(f_0)+1\right) \\
           &\quad \times \left(  \frac{1}{|v_1+u|}+\frac{1}{|\bv_1+u|}+|v_1+u|\right)  +\mathbf{1}_{\{|v_1-\bv_1|\geq 1\}}.
   \end{split}
   \end{align}
   
    For $x,\bx \in \O, \,v,\bv \in \mathbb{R}^3$, and $t>0$, we define
   \begin{align*}
           (X(s),V(s)):=(X(s;t,x,v),V(s;t,x,v)),
           \quad \bV(s) := V(s;t,\bx,\bv)
   \end{align*}
   for $|(x,v)-(\bx,\bv)|\leq 1$. For $u \in \R^3$, we define
   \begin{align*}
       \quad\widetilde{V}(s) := \tv\left( V(s),\bV(s),u \right)
   \end{align*}
   such that $|V(s)+u|=|\widetilde{V}(s)+u|$.  We substitute $t_1,x_1,v_1$ and $\bv_1$ with $s, X(s), V(s)$ and $\bV(s)$. Multiply both sides of the inequality with $e^{-c|u|^2}/|u|$, integrate with respect to $u$ and $s$, and transform the (LHS) into the following form:
     \begin{align}\label{v:int_du}
       e^{-\varpi \langle v \rangle^{2} t}  \bigintsss_0^t\bigintsss_{\mathbb{R}^3}  \frac{e^{-c|u|^2}}{|u|}\frac{|f(s,X(s),V(s)+u)-f(s,X(s),\bV(s)+u)|}{|V(s)-\bV(s)|} du ds. 
   \end{align}  
   Then, we obtain
      \begin{align} \label{J123}
        (\ref{v:int_du})\mathcal{P}_3^{-1}(\|w_0f_0\|_{\infty} ) 
        \lesssim_{\vartheta_0}  \left(\sup_{0 \leq t \leq T}\mathfrak{X}(t,\varpi;\epsilon)+\sup_{0 \leq t \leq T}\mathfrak{V}(t,\varpi;\epsilon)+ \mathbf{A}_{\frac{1}{2}}(f_0)+1\right)(J_1+J_2+J_3) + J_4
    \end{align}
     for $0 < t \leq T <  \min\{T_1,1\}$. Here, $J_1, J_2, J_3$, and $J_4$ are given by
   \begin{align*}
       J_1 &= e^{-\varpi \langle v \rangle^{2} t} \bigintsss_0^t  \mathbf{1}_{\{|V(s)-\bV(s)|\leq 1\}}\mathbf{1}_{\{d(X(s),\partial\O) \lesssim 1\}} \bigintsss_{\mathbb{R}^3}  e^{\varpi \langle V(s)+u \rangle^2 s}\frac{e^{-c|u|^2}}{|u|}\\
          &\quad \times 
        \left(  \frac{1}{|V(s)+u|^{\frac{3}{2}}}+ 1 \right)
        \left(   \frac{\mathbf{1}_{\{\tb(X(s),V(s)+u)<+\infty\}}}{\left|\nabla \xi(\xb(X(s),V(s)+u)\cdot (\widehat{V(s)+u})\right|}+1 \right) duds\\
        &\quad + e^{-\varpi \langle v \rangle^{2} t} \bigintsss_0^t  \mathbf{1}_{\{|V(s)-\bV(s)|\leq 1\}}\mathbf{1}_{\{d(X(s),\partial\O) \lesssim 1\}} \bigintsss_{\mathbb{R}^3}  e^{\varpi \langle V(s)+u \rangle^2 s}\frac{e^{-c|u|^2}}{|u|}\\
          &\quad \times 
        \left(  \frac{1}{|\bV(s)+u|^{\frac{3}{2}}}+ 1 \right)
        \left(   \frac{\mathbf{1}_{\{\tb(X(s),\bV(s)+u)<+\infty\}}}{\left|\nabla \xi(\xb(X(s),\bV(s)+u)\cdot(\widehat{\bV(s)+u})\right|}+1 \right) duds
   \end{align*}
   and
    \begin{align*}
       J_2 &= e^{-\varpi \langle v \rangle^{2} t} \bigintsss_0^t  \mathbf{1}_{\{|V(s)-\bV(s)|\leq 1\}}\mathbf{1}_{\{1 \lesssim d(X(s),\partial\O) \lesssim \langle V(s)+u \rangle\}} \bigintsss_{\mathbb{R}^3}  e^{\varpi \langle V(s)+u \rangle^2 s}\frac{e^{-c|u|^2}}{|u|}\\
        &\quad \times 
        \left(  \frac{1}{|V(s)+u|^{\frac{3}{2}}}+ |V(s)+u| \right)
        \left(   \frac{\mathbf{1}_{\{\tb(X(s),V(s)+u)<+\infty\}}}{\left|\nabla \xi(\xb(X(s),V(s)+u)\cdot \widehat{(V(s)+u)}\right|}+1 \right) duds\\
        &\quad + e^{-\varpi \langle v \rangle^{2} t} \bigintsss_0^t \mathbf{1}_{\{|V(s)-\bV(s)|\leq 1\}}\mathbf{1}_{\{ 1 \lesssim d(X(s),\partial\O) \lesssim \langle \bV(s)+u\rangle\}}\bigintsss_{\mathbb{R}^3}  e^{\varpi \langle V(s)+u \rangle^2 s}\frac{e^{-c|u|^2}}{|u|}\\
          &\quad \times 
        \left(  \frac{1}{|\bV(s)+u|^{\frac{3}{2}}}+ |\bV(s)+u| \right)
        \left(   \frac{\mathbf{1}_{\{\tb(X(s),\bV(s)+u)<+\infty\}}}{\left|\nabla \xi(\xb(X(s),\bV(s)+u)\cdot \widehat{(\bV(s)+u)}\right|}+1 \right) duds
   \end{align*}
   and
    \begin{align*}
       J_3 &= e^{-\varpi \langle v \rangle^{2} t} \bigintsss_0^t  \mathbf{1}_{\{|V(s)-\bV(s)|\leq 1\}} \bigintsss_{\mathbb{R}^3}  e^{\varpi \langle V(s)+u \rangle^2 s}\frac{e^{-c|u|^2}}{|u|} \\
        &\quad \times \left(  \frac{1}{|V(s)+u|}+ \frac{1}{|\bV(s)+u|}+ |V(s)+u| \right)duds
   \end{align*}
   and
   \begin{align*}
       J_4 &= e^{-\varpi \langle v \rangle^{2} t} \bigintsss_0^t  \bigintsss_{\mathbb{R}^3}  \frac{e^{-c|u|^2}}{|u|}duds.
   \end{align*}

  By Lemma \ref{lem:ws} and $|V(s)|=v$, we have
    \begin{align} \label{exp_est_V}
    \begin{split}
        e^{-\varpi \langle v \rangle^{2} t}e^{\varpi \langle V(s)+u \rangle^{2} s}  &=e^{-\varpi \langle V(s) \rangle^{2} (t-s)}
        e^{-\varpi \langle V(s) \rangle^{2} s}e^{\varpi \langle V(s)+u \rangle^{2} s} \\
        &\leq e^{-\varpi \langle V(s) \rangle^{2} (t-s)}e^{\frac{c}{2}|u|^2}=e^{-\varpi \langle v \rangle^{2} (t-s)}e^{\frac{c}{2}|u|^2}
    \end{split}
    \end{align}
    for $\varpi T \leq \min\{c/4, 1/8\}$.

    First, we consider $J_1$. Applying \eqref{exp_est_V} and performing the change of variable $u \rightarrow w = V(s)+u $ for the first term of $J_1$ and $u \rightarrow w = \bV(s)+u $ for the second term, we obtain
    \begin{align*}
        J_1 &\leq  \bigintsss_0^t  e^{-\varpi \langle v \rangle^{2} (t-s)}\mathbf{1}_{\{|V(s)-\bV(s)|\leq 1\}}\mathbf{1}_{\{d(X(s),\partial\O) \lesssim 1\}} \\
          &\quad \times \bigintsss_{\mathbb{R}^3}  \left(\frac{e^{-\frac{c}{2}|w-V(s)|^2}}{|w-V(s)|}+\frac{e^{-\frac{c}{2}|w-\bV(s)|^2}}{|w-\bV(s)|} \right)
        \left(  \frac{1}{|w|^{\frac{3}{2}}}+ 1 \right)
        \left(   \frac{\mathbf{1}_{\{\tb(X(s),w)<+\infty\}}}{\left|\nabla \xi(\xb(X(s),w)\cdot \widehat{w}\right|}+1 \right) duds.
    \end{align*}
    By Lemma \ref{lem:int_sing} with $k=0, \,3/2$ and $\langle \bv \rangle \leq 2\langle v \rangle$, we have
    \begin{align*}
        J_1 \lesssim \langle v \rangle \bigintsss_0^t e^{-\varpi \langle v \rangle^{2} (t-s)} \left[\ln \left(  1+ \frac{1}{d(X(s;t,x,v),\partial\O)} \right)+1\right] ds.
    \end{align*}
    When $d(x,\partial\O) \leq \epsilon$, we apply Lemma \ref{lem:ds}. When $d(x,\partial\O) \geq \epsilon$, we apply Lemma \ref{lem:ds_2}. Then, we obtain 
    \begin{align*}
        J_1 \lesssim \ln\left(1+ \frac{1}{\epsilon} \right)\frac{1}{\sqrt{\varpi}}
    \left[ \ln \left( 1+\frac{1}{|v|} \right)\mathbf{1}_{\{d(x,\partial\O) \leq \epsilon\}}+1\right] \lesssim_{\epsilon} \frac{1}{\sqrt{\varpi}}G(x,v;\epsilon).
    \end{align*} \\

    Next, we consider $J_2$. Applying \eqref{exp_est_V} and Lemma \ref{lem:int_sing} with $k=3/2,\,1$, we have
    \begin{align*}
        J_2 \lesssim \langle v \rangle^2 \bigintsss_0^t e^{-\varpi \langle v \rangle^{2} (t-s)} \left[\ln \left(  1+ \frac{1}{d(X(s;t,x,v),\partial\O)} \right)+1\right]\mathbf{1}_{\{1 \lesssim d(X(s),\partial\O) \}} ds. 
    \end{align*}
    Since $d(X(s),\partial\O) \gtrsim 1$ holds, we obtain
    \begin{align*}
        J_2 \lesssim \langle v \rangle^2
        \int_0^t e^{-\varpi \langle v \rangle^{2} (t-s)} ds 
        \lesssim \frac{1}{\varpi}.
    \end{align*}

     Moreover, we can compute 
    \begin{align*}
        J_3 &\lesssim  \bigintsss_0^t e^{-\varpi \langle v \rangle^{2} (t-s)} \bigintsss_{\mathbb{R}^3}  \frac{e^{-\frac{c}{2}|u|^2}}{|u|}
        \left(  \frac{1}{|V(s)+u|}+ \frac{1}{|\bV(s)+u|}+ |V(s)+u| \right)duds \\
        &\lesssim 
        \bigintsss_0^t e^{-\varpi \langle v \rangle^{2} (t-s)} ds\left(\bigintsss_{\mathbb{R}^3} \frac{e^{-\frac{c}{2}|z|^2}}{|z|^2} dz+ \langle  v \rangle \bigintsss_{\mathbb{R}^3} \left( 1+\frac{1}{|u|}\right)e^{-\frac{c}{2}|u|^2} du\right) \lesssim \frac{1}{\varpi \langle v \rangle}
    \end{align*}
    by using \eqref{exp_est_V} and 
    \begin{align*}
        J_4 \lesssim \frac{1}{\varpi \langle v \rangle^2}.
    \end{align*}

   Therefore, combining the upper bounds of $J_1$ to $J_4$, we obtain
   \begin{align*}
       (\ref{v:int_du}) 
      &\lesssim_{\vartheta_0, \epsilon} \frac{1}{\sqrt{\varpi}}\mathcal{P}_3(\|w_0f_0\|_{\infty} )   G(x,v;\epsilon) 
         \left(\sup_{0 \leq t \leq T}\mathfrak{X}(t,\varpi;\epsilon)+\sup_{0 \leq t \leq T}\mathfrak{V}(t,\varpi;\epsilon)+\mathbf{A}_{\frac{1}{2}}(f_0)+1\right)
   \end{align*}
   from \eqref{J123}. Finally, we multiply both sides of the inequality by $G^{-1}(x,v;\epsilon)$ and take the supremum over $(x,\bx,v,\bv)\in \O\times \O\times \R^{3}\times \R^{3}$ and $0 \leq t \leq T$, and obtain the lemma.
\end{proof}

Finally, Proposition \ref{pro:iter_pro} follows from Lemma \ref{pro:X<1} and Lemma \ref{pro:V<1}.

\begin{proposition}\label{pro:iter_pro}
    There exist constants $0<T \ll 1 $ and $\varpi \gg \mathcal{P}_6(\|w_0f_0\|_{\infty})$, depending on $\epsilon$ and $\vartheta_0$, such that
\begin{align*}
     \sup_{0 \leq t \leq T}\mathfrak{X}(t,\varpi;\epsilon)+ \sup_{0 \leq t \leq T}\mathfrak{V}(t,\varpi;\epsilon)\lesssim_{\vartheta_0, \epsilon} \mathbf{A}_{\frac{1}{2}}(f_0)+1
\end{align*} 
and $\varpi T \leq \min\{c/4, 1/8\} $ for any $0<\epsilon<1$.
\end{proposition}
\begin{proof}
    By Lemma \ref{pro:X<1} and Lemma \ref{pro:V<1}, we obtain 
    \begin{align*}
        &\sup_{0 \leq t \leq T}\mathfrak{X}(t,\varpi;\epsilon)+\sup_{0 \leq t \leq T}\mathfrak{V}(t,\varpi;\epsilon) \\ &\lesssim_{\vartheta_0, \epsilon} \frac{1}{\sqrt{\varpi} }\mathcal{P}_3(\|w_0f_0\|_{\infty} ) \left(\sup_{0 \leq t \leq T}\mathfrak{X}(t,\varpi;\epsilon)+\sup_{0 \leq t \leq T}\mathfrak{V}(t,\varpi;\epsilon) \right)+ \frac{1}{\sqrt{\varpi} }\mathcal{P}_3(\|w_0f_0\|_{\infty} )\left( \mathbf{A}_{\frac{1}{2}}(f_0)+1 \right).
    \end{align*}
   Lastly, we take a sufficiently large $\varpi$ such that $\varpi \gg \mathcal{P}_6(\|w_0f_0\|_{\infty} )$ for fixed $\vartheta_0$ and $\epsilon$. \\
\end{proof}

\section{Proof of the Theorem}

\subsection{\texorpdfstring{$C_{x,v}^{0,\frac{1}{2}}$}{} estimates}

First, we examine the \( C_{x,v}^{0,\frac{1}{2}} \) estimates along the trajectory.
 Recall Lemma 7.1 in \cite{CD2023}. Using these estimates, we further obtain $C_{x,v}^{0,\frac{1}{2}}$ estimates for \( \Gamma_{gain}(f,f) \) and \( \nu(f) \).

\begin{lemma}[$C_{x,v}^{0,\frac{1}{2}}$ estimation of the trajectory] \label{lem:tra_0.5}
Let $t > 0$, $x, \bx \in \Omega$, and $v, \bv \in \mathbb{R}^3$ be such that $|(x,v)-(\bx,\bv)| \leq 1$. 

\noindent (1)   When $\tb(x, v) < \infty$ and $\tb(\bx,\bv) < \infty$, we obtain
         \begin{align*}
                \max\{|v|,|\bv|\}|\tb(x,v)-\tb(\bx,\bv)| \lesssim \left(
     			|x- \bar x|^{\frac{1}{2}} +  \min\left\{\sqrt{\tb(x,v)},\sqrt{\tb(\bx,\bv)}\right\}|v- \bar v |^{\frac{1}{2}}
     			\right).
        \end{align*} 
        
\noindent (2) For $0 \leq s \leq t$, we obtain
    \begin{align*}
        |X(s;t,x,v ) -X(s;t,\bar x, \bar v)|
        \lesssim
        \left(1+  \langle v\rangle |t-s|  \right)
        \left(
        |x- \bx|^{\frac{1}{2}} +  |t-s|^{\frac{1}{2}}|v- \bar v |^{\frac{1}{2}}
        \right).
\end{align*}

\noindent (3) For $0 \leq s \leq \min\{t_1(t,x,v), t_1(t,\bx,\bv)\} \leq t$ and $ \max\{t_1(t,x,v), t_1(t,\bx,\bv)\} \leq s \leq t$, we obtain
    \begin{align} \label{1/2-3}
        |V(s;t,x,v) -V(s;t,\bar x,\bv)|
        \lesssim
        |v- \bar v |
        +   \langle v\rangle  
        \left(
        |x- \bar x|^{\frac{1}{2}} +  |t-s|^{\frac{1}{2}}|v- \bar v |^{\frac{1}{2}}
        \right). 
    \end{align}
    When $\min\{t_1(t,x,v), t_1(t,\bx,v)\} =-\infty$, \eqref{1/2-3} with $\bv=v$ also holds for $0 \leq s \leq t$. \\
    When $\min\{t_1(t,\bx,v), t_1(t,\bx,\bv)\} =-\infty$, \eqref{1/2-3} with $x=\bx$ also holds for $0 \leq s \leq t$.
     
    \noindent (4) When $s=0$, we obtain
        \begin{align*} 
                &|f(0,X(0;t,x,v),V(0;t,x,v))-f(0,X(0;t,\bx,\bv),V(0;t,\bx,\bv))|\\
                &\lesssim \langle v \rangle \left( 1+t^{\frac{3}{2}}\right) \left( |x- \bx |^{\frac{1}{2}} +|v- \bv|^{\frac{1}{2}} \right) \left( \mathbf{A}_{\frac{1}{2}}(f_0)+\|w_0 f_0\|_{\infty} \right).
        \end{align*}  
    \end{lemma}
    \begin{proof}
        In \cite{CD2023}, (1) and (2) are proved in Lemma 7.1, and \eqref{1/2-3} is also obtained for
        $0 \leq s \leq \min\{t_1(t,x,v), t_1(t,\bx,\bv)\} \leq t$ and $ \max\{t_1(t,x,v), t_1(t,\bx,\bv)\} \leq s \leq t$ when $t_1(t,x,v),\, t_1(t,\bx,\bv) \neq -\infty$. Next, when $\min\{t_1(t,x,v), t_1(t,\bx,v)\} =-\infty$, we estimate $|V(s;t,x,v)-V(s;t,\bx,v)|$ for $0\leq s \leq t$ using the same arguments as in the proof of Lemma \ref{lem:tra_sin_x}, parts (d), (e), and (f). Similarly, when $\min\{t_1(t,\bx,v), t_1(t,\bx,\bv)\} =-\infty$, we estimate $|V(s;t,\bx,v)-V(s;t,\bx,\bv)|$ for $0\leq s \leq t$ in the same way. Moreover, (4) follows from (1), (2), and (3) by applying the same arguments as in the proof of Lemma \ref{lem:s=0_spec_x}.      
    \end{proof}

  We apply similar arguments as in Lemma \ref{lem:ga_X} and Lemma \ref{lem:ga_V}, but instead use Lemma \ref{lem:tra_0.5}. In this case, the singularity associated with the grazing velocities does not appear, and we obtain $|(x,v) - (\bx,\bv)|^{1/2}$. Moreover, we divide the analysis into two cases: $d(\bx,\partial\O) \geq \epsilon$ and $d(\bx,\partial\O) \leq \epsilon$. In the former case, the singularity of $|v|$ can be eliminated. In the latter case, we apply \eqref{pro:H_sub}, which is derived from Proposition \ref{pro:H}, to reduce the singularity in $|v|$.

     \begin{lemma}[$C_{x,v}^{0,\frac{1}{2}}$ estimation of $\Gamma_{gain}(f,f)$ and $\nu(f)$]\label{lem:gam_x_0.5}

     Assume that $x,\bx \in \O$ and $v, \bv \in \mathbb{R}^3$ with $|(x,v)-(\bx,\bv)|\leq 1$. Recall $T_1>0$ and $\varpi_1>1$ from Proposition \ref{pro:H}. For $0<t <  \min\{T_1,1\}, \,\varpi>\varpi_1$, and $0< \epsilon,\,\delta < 1$, the following hold:
    \begin{align*} 
         &\int_0^{t} \left|\Gamma_{gain}(f,f)(s,X(s;t,x,v),V(s;t,x,v))-\Gamma_{gain}(f,f)(s,X(s;t,\bx,\bv),V(s;t,\bx,\bv))\right| ds\\
           &\lesssim_{\delta,\epsilon, \vartheta_0} |(x,v)-(\bx,\bv)|^{\frac{1}{2}} \mathcal{P}_2(\|w_0f_0\|_{\infty} )\left(\mathfrak{X}(t,\varpi;\epsilon)+ \mathfrak{V}(t,\varpi;\epsilon)+\mathbf{A}_{\frac{1}{2}}(f_0)+1\right) \\
           &\quad \times\langle v \rangle e^{4\varpi \langle v \rangle^{2} t}\left( \frac{1}{|v|^{\delta}}\mathbf{1}_{\{d(x,\partial\O)\leq\epsilon\}}+\frac{1}{|\bv|^{\delta}}\mathbf{1}_{\{d(\bx,\partial\O)\leq\epsilon\}}+1\right)
    \end{align*} 
    and
    \begin{align*} 
         &\int_0^{t} \left|\nu(f)(s,X(s;t,x,v),V(s;t,x,v))-\nu(f)(s,X(s;t,\bx,\bv),V(s;t,\bx,\bv))\right| ds\\
           &\lesssim_{\delta,\epsilon} |(x,v)-(\bx,\bv)|^{\frac{1}{2}} \mathcal{P}_1(\|w_0f_0\|_{\infty} )\left(\mathfrak{X}(t,\varpi;\epsilon)+1\right) \\
           &\quad \times \langle v \rangle  e^{\varpi \langle v \rangle^{2} t}\left( \frac{1}{|v|^{\delta}}\mathbf{1}_{\{d(x,\partial\O)\leq\epsilon\}}+\frac{1}{|\bv|^{\delta}}\mathbf{1}_{\{d(\bx,\partial\O)\leq\epsilon\}}+1\right).
    \end{align*} 
   \end{lemma}
   \begin{proof} 
   We assume $d(\bx,\partial\O) \leq d(x,\partial\O)$ without loss of generality. Next, we split
\begin{align}
        &\left|\Gamma_{gain}(f,f)(s,X(s;t,x,v),V(s;t,x,v))-\Gamma_{gain}(f,f)(s,X(s;t,\bx,\bv),V(s;t,\bx,\bv))\right|
        \label{X,V_1/2}\\ 
       &\leq
       \left|\Gamma_{gain}(f,f)(s,X(s;t,x,v),V(s;t,x,v))-\Gamma_{gain}(f,f)(s,X(s;t,\bx,v),V(s;t,x,v))\right|
       \label{G_Xx_1/2} \\ 
       &\quad  {+}
       \left|\Gamma_{gain}(f,f)(s,X(s;t,\bx,v),V(s;t,x,v))-\Gamma_{gain}(f,f)(s,X(s;t,\bx,v),V(s;t,\bx,v))\right|
       \label{G_Vx_1/2} \\ 
       &\quad  {+}
       \left|\Gamma_{gain}(f,f)(s,X(s;t,\bx,v),V(s;t,\bx,v))-\Gamma_{gain}(f,f)(s,X(s;t,\bx,\bv),V(s;t,\bx,v))\right|
       \label{G_Xv_1/2} \\ 
         &\quad  {+}
       \left|\Gamma_{gain}(f,f)(s,X(s;t,\bx,\bv),V(s;t,\bx,v))-\Gamma_{gain}(f,f)(s,X(s;t,\bx,\bv),V(s;t,\bx,\bv))\right|.
        \label{G_Vv_1/2}
    \end{align}
    
     First, we consider \eqref{G_Xx_1/2}. By using \eqref{gamma_x}, we obtain
      \begin{align*}
          \int_0^t (\ref{G_Xx_1/2}) ds 
         &\lesssim   \|w_0f_0\|_{\infty}\int_0^t |X(s;t,x,v)-X(s;t,\bx,v)|\\
         &\quad \times \int_{\mathbb{R}^3} \frac{e^{-c|u|^2}}{|u|} 
         \frac{\left|f(s,X(s;t,x,v),V(s;t,x,v)+u)-f(s,X(s;t,\bx,v),V(s;t,x,v)+u)\right|}{|X(s;t,x,v)-X(s;t,\bx,v)|} du ds.
     \end{align*}
     By Lemma \ref{lem:tra_0.5}-(2), we obtain
     \begin{align*}
          |X(s;t,x,v)-X(s;t,\bx,v)| 
          &\leq |x-\bx|^{\frac{1}{2}} (1+\langle v \rangle(t-s))\\
        &\leq
        |x-\bx|^{\frac{1}{2}} \langle v \rangle
        \max \left\{ (t-s)^{\frac{1}{2}}, \frac{1}{\langle v \rangle}    \right\}.
     \end{align*}
     Therefore,
     \begin{align*} 
        \int_0^t (\ref{G_Xx_1/2}) ds 
        &\lesssim   \|w_0f_0\|_{\infty}|x-\bx|^{\frac{1}{2}} \langle v \rangle e^{\varpi \langle v \rangle^{2} t}
        \left( G(x,v;\epsilon)+G(\bx,v;\epsilon)\right)\mathfrak{X}(t,\varpi;\epsilon).
     \end{align*}
     Similarly, we obtain 
     \begin{align*} 
        \int_0^t (\ref{G_Xv_1/2}) ds 
        &\lesssim   \|w_0f_0\|_{\infty}|v-\bv|^{\frac{1}{2}} \langle v \rangle e^{\varpi \langle v \rangle^{2} t}
        \left( G(\bx,v;\epsilon)+G(\bx,\bv;\epsilon)\right)\mathfrak{X}(t,\varpi;\epsilon).
     \end{align*}

    \vspace{5mm}
   Suppose $0 \leq \min\{t_1(t,x,v), t_1(t,\bx,v)\} \leq  \max\{t_1(t,x,v), t_1(t,\bx,v)\} \leq t$. We estimate the integral
    \begin{align*}
        \int_{ \min\{t_1(t,x,v), t_1(t,\bx,v)\}}^{ \max\{t_1(t,x,v), t_1(t,\bx,v)\}} \eqref{G_Vx_1/2} ds
    \end{align*}
    by considering two cases: $d(\bx,\partial\O) \geq \epsilon$ and $d(\bx,\partial\O) \leq \epsilon$.
    
    We assume $d(\bx,\partial\O) \geq \epsilon$. Since 
    \begin{align} \label{d>epsilon_case}
    \begin{split}
        |v|(t-s) &\geq  |v|\left(t-\max\{t_1(t,x,v), t_1(t,\bx,v)\}\right) \\
        &\geq \min \left\{d(\bx,\partial\O), d(x,\partial\O)\right\}\geq \epsilon,
    \end{split}
    \end{align}
    we have $|v| \geq \epsilon/t \geq \epsilon$.
   Then, by applying \eqref{gamma_upper}, we obtain
    \begin{align*} 
        \begin{split}
         &\int_{\min\{t_1(t,x,v), t_1(t,\bx,v)\}}^{\max\{t_1(t,x,v), t_1(t,\bx,v)\}} (\ref{G_Vx_1/2}) ds 
         \lesssim_{\vartheta_0}\|w_0f_0\|_{\infty}^2|\tb(x,v)-\tb(\bx,v)| \\
         &\leq |x-\bx|^{\frac{1}{2}}\|w_0f_0\|_{\infty}^2 \frac{1}{|v|} \lesssim_{\epsilon} |x-\bx|^{\frac{1}{2}}\|w_0f_0\|_{\infty}^2.
        \end{split}
    \end{align*}
    
      We assume that $d(\bx, \partial\O) \leq \epsilon$. We follow the same arguments as in \eqref{est_V_vel_s}, but instead of substituting $\beta = \frac{1}{4}$ when applying \eqref{pro:H_sub} (which is derived from Proposition~\ref{pro:H}), we retain the general case $\beta \in (0,1)$.
 Then 
   \begin{align}  \label{make_vbeta}
   \begin{split}
       (\ref{G_Vx_1/2})
       &\lesssim_{\beta} 2\mathcal{P}_2(\|w_0f_0\|_{\infty} )e^{\varpi \langle v\rangle^{2} t} \left( \mathbf{A}_{\frac{1}{2}}(f_0)+1 \right)|v|^{\beta}.
   \end{split}
   \end{align}
   From Lemma \ref{lem:tra_0.5}-(1),
   \begin{align} \label{d<epsilon_case}
       |\tb(x,v)-\tb(\bx,v)| 
       &\lesssim 
      \frac{1}{|v|} |x-\bx|^{\frac{1}{2}}. 
   \end{align}
   Using \eqref{make_vbeta} and the above inequality, we obtain
   \begin{align*}
       &\int_{\min\{t_1(t,x,v), t_1(t,\bx,v)\}}^{\max\{t_1(t,x,v), t_1(t,\bx,v)\}} \eqref{G_Vx_1/2} ds \notag\\
       &\lesssim_{\beta, \vartheta_0}  \mathcal{P}_2(\|w_0f_0\|_{\infty} )e^{\varpi \langle v \rangle^{2} t} \left( \mathbf{A}_{\frac{1}{2}}(f_0)+1 \right)|v|^{\beta}|\tb(x,v)-\tb(\bx,v)| \\
       &\lesssim |x-\bx|^{\frac{1}{2}}\mathcal{P}_2(\|w_0f_0\|_{\infty} )e^{\varpi \langle v \rangle^{2} t} \left( \mathbf{A}_{\frac{1}{2}}(f_0)+1 \right)\frac{1}{|v|^{1-\beta}}.
   \end{align*}

   By combining the estimates for the cases $d(\bx,\partial\O) \geq \epsilon$ and $d(\bx,\partial\O) \leq \epsilon$, we obtain
   \begin{align} \label{max_min_1/2}
   \begin{split}
        &\int_{\min\{t_1(t,x,v), t_1(t,\bx,v)\}}^{\max\{t_1(t,x,v), t_1(t,\bx,v)\}} (\ref{G_Vx_1/2}) ds \\
        &\lesssim_{\beta, \vartheta_0} |x-\bx|^{\frac{1}{2}}\mathcal{P}_2(\|w_0f_0\|_{\infty} )e^{\varpi \langle v \rangle^{2} t} \left( \mathbf{A}_{\frac{1}{2}}(f_0)+1 \right)
        \left(\frac{1}{|v|^{1-\beta}}\mathbf{1}_{\{d(\bx,\partial\O) \leq \epsilon\}}+1
        \right).
    \end{split}
   \end{align}
   
   To estimate $\int_0^t\eqref{G_Vx_1/2}$, we divide the cases $0 \leq \min\{t_1(t,x,v), t_1(t,\bx,v)\} \leq  \max\{t_1(t,\bx,v), t_1(t,x,v)\} \leq t$, $\min\{t_1(t,x,v), t_1(t,\bx,v)\} \leq  0 \leq \max\{t_1(t,\bx,v), t_1(t,x,v)\} \leq t$, and $\max\{t_1(t,\bx,v), t_1(t,x,v)\} \leq 0$. Next, we follow the same arguments as in the estimation of \eqref{V_sp} in Lemma \ref{lem:ga_X}. Although we previously applied Lemma \ref{lem:tra_sin_x}-(2), we now apply Lemma \ref{lem:tra_0.5}-(3) instead. Moreover, for the case $s \in [\min\{t_1(t,x,v), t_1(t,\bx,v)\}, \max\{t_1(t,x,v), t_1(t,\bx,v)\} ] $, we now use \eqref{max_min_1/2}, rather than \eqref{max>t>min_gamma_x} used previously. Then we derive
   \begin{align*}
       \int_0^t (\ref{G_Vx_1/2}) ds 
       &\lesssim_{\vartheta_0, \beta}
       |x-\bx|^{\frac{1}{2}} \mathcal{P}_2(\|w_0f_0\|_{\infty} )\mathfrak{V}(t,\varpi;\epsilon)\langle v \rangle e^{\varpi \langle v \rangle^{2} t} G(\bx,v;\epsilon) \\
       &\quad + |x-\bx|^{\frac{1}{2}}\mathcal{P}_2(\|w_0f_0\|_{\infty} )e^{\varpi \langle v \rangle^{2} t} \left( \mathbf{A}_{\frac{1}{2}}(f_0)+1 \right)
        \left(\frac{1}{|v|^{1-\beta}}\mathbf{1}_{\{d(\bx,\partial\O) \leq \epsilon\}}+1
        \right).
   \end{align*}

    Next, we estimate $\int_0^t \eqref{G_Vv_1/2}$ in a similar way. \\
    Suppose $0 \leq \min\{t_1(t,\bx,v), t_1(t,\bx,\bv)\} \leq  \max\{t_1(t,\bx,v), t_1(t,\bx,\bv)\} \leq t$. We estimate the integral
    \begin{align*}
        \int_{ \min\{t_1(t,\bx,v), t_1(t,\bx,\bv)\}}^{ \max\{t_1(t,\bx,v), t_1(t,\bx,\bv)\}} \eqref{G_Vv_1/2} ds.
    \end{align*}
     When $d(\bx,\partial\O) \geq \epsilon$, we apply
     \begin{align*}
         \max\{|\bv|,|v|\}(t-s) &\geq  \max\{|\bv|,|v|\}\left(t-\max\{t_1(t,\bx,v), t_1(t,\bx,\bv)\}\right)\geq d(\bx,\partial\O)\geq \epsilon
     \end{align*}
     instead of \eqref{d>epsilon_case}. When $d(\bx,\partial\O) \leq \epsilon$, we apply
     \begin{align*} 
       (\ref{G_Vv_1/2})
       &\lesssim_{\beta} \mathcal{P}_2(\|w_0f_0\|_{\infty} )e^{\varpi \langle v\rangle^{2} t} \left( \mathbf{A}_{\frac{1}{2}}(f_0)+1 \right)(|v|^{\beta}+|\bv|^{\beta}).
   \end{align*}
   instead of \eqref{make_vbeta} and
   \begin{align*}
       |\tb(\bx,v)-\tb(\bx,\bv)| 
       &\lesssim 
       \min\left\{\frac{1}{|v|},\frac{1}{|\bv|}\right\} |v-\bv|^{\frac{1}{2}} 
   \end{align*}
   instead of \eqref{d<epsilon_case}. Then, similar to $\int_0^t \eqref{G_Vx_1/2}$, we obtain
   \begin{align*}
       \int_0^t (\ref{G_Vv_1/2}) ds 
       &\lesssim_{\vartheta_0, \beta}
       |v-\bv|^{\frac{1}{2}} \mathcal{P}_2(\|w_0f_0\|_{\infty} )\mathfrak{V}(t,\varpi;\epsilon)\langle v \rangle e^{\varpi \langle \bv \rangle^{2} t} G(\bx,\bv;\epsilon) \\
       &\quad + |v-\bv|^{\frac{1}{2}}\mathcal{P}_2(\|w_0f_0\|_{\infty} )e^{\varpi \langle v \rangle^{2} t} \left( \mathbf{A}_{\frac{1}{2}}(f_0)+1 \right)
        \left(\frac{1}{|v|^{1-\beta}}\mathbf{1}_{\{d(\bx,\partial\O) \leq \epsilon\}}+\frac{1}{|\bv|^{1-\beta}}\mathbf{1}_{\{d(\bx,\partial\O) \leq \epsilon\}}+1
        \right).
   \end{align*}
   
    By combining the upper bounds of $\int_0^t \eqref{G_Xx_1/2}-\eqref{G_Vv_1/2} \,ds$, and using the inequality
    \begin{align*}
       G(x,v;\epsilon) = \ln\left(1+\frac{1}{|v|} \right)\mathbf{1}_{\{d(x,\partial\O) \leq \epsilon\}}  + 1 \lesssim_{\beta} \frac{1}{|v|^{1-\beta}}\mathbf{1}_{\{d(x,\partial\O) \leq \epsilon\}} +1,
   \end{align*} 
   together with $\langle \bv \rangle \leq 2\langle v \rangle$, we obtain
   \begin{align*}
       &\int_0^t \eqref{X,V_1/2} ds \\
      &\lesssim_{\beta,\epsilon, \vartheta_0} \langle v \rangle|(x,v)-(\bx,\bv)|^{\frac{1}{2}} \mathcal{P}_2(\|w_0f_0\|_{\infty} )e^{4\varpi \langle v \rangle^{2} t}\left(\mathfrak{X}(t,\varpi;\epsilon)+ \mathfrak{V}(t,\varpi;\epsilon)+\mathbf{A}_{\frac{1}{2}}(f_0)+1\right) \\
           &\quad \times  \left(\frac{1}{|v|^{1-\beta}}\mathbf{1}_{\{d(\bx,\partial\O) \leq \epsilon\}}+\frac{1}{|\bv|^{1-\beta}}\mathbf{1}_{\{d(\bx,\partial\O) \leq \epsilon\}}+G(x,v;\epsilon)+G(\bx,v;\epsilon)+G(\bx,\bv;\epsilon)+1
        \right) \\
         &\lesssim_{\beta,\epsilon, \vartheta_0} \langle v \rangle|(x,v)-(\bx,\bv)|^{\frac{1}{2}} \mathcal{P}_2(\|w_0f_0\|_{\infty} )e^{4\varpi \langle v \rangle^{2} t} \left(\mathfrak{X}(t,\varpi;\epsilon)+ \mathfrak{V}(t,\varpi;\epsilon)+\mathbf{A}_{\frac{1}{2}}(f_0)+1\right) \\
           &\quad \times  \left(\frac{1}{|v|^{1-\beta}}\mathbf{1}_{\{d(\bx,\partial\O) \leq \epsilon\}}+\frac{1}{|\bv|^{1-\beta}}\mathbf{1}_{\{d(\bx,\partial\O) \leq \epsilon\}}+\frac{1}{|v|^{1-\beta}}\mathbf{1}_{\{d(x,\partial\O) \leq \epsilon\}}+1
        \right).
   \end{align*}
   Finally, we use the fact that $d(\bx,\partial\O) \leq d(x,\partial\O)$  and replace $1-\beta$ with $\delta$ in the above inequality.

    Similarly, we estimate
 \begin{align*}
     \int_0^{t} \left|\nu(f)(s,X(s;t,x,v),V(s;t,x,v))-\nu(f)(s,X(s;t,\bx,\bv),V(s;t,\bx,\bv))\right| ds.
 \end{align*}
 by applying Lemma \ref{lem:diff_nu} instead of Lemma \ref{lem:est_Gam}. 
  \end{proof}
   
\subsection{Proof of the Theorem}

We are now ready to prove the main theorem. To this end, we apply Lemma \ref{lem:tra_0.5}-(4), Lemma \ref{lem:gam_x_0.5}, and Proposition \ref{pro:iter_pro}.

\begin{proof}[Proof of Theorem \ref{thm:Holder0.5}]
We recall
\begin{align}  
& |f(t,x,v)-f(t,\bar{x}, \bar{v})|  \notag \\
&\lesssim_{\vartheta_0}   
\left|f(0,X(0;t,x,v), V(0;t,x,v))-f(0,X(0;t,\bx,\bv), V(0;t,\bx,\bv))\right| 
 \label{basic f-f1'} \\
&\; + \int^t_0 
\left|\Gamma_{\text{gain}}(f,f)(s,X(s;t,x,v ), V(s;t,x,v))
-\Gamma_{\text{gain}}(f,f)(s,X(s;t,\bx,\bv ), V(s;t,\bx,\bv))\right|   ds
 \label{basic f-f2'} \\
&\;+ \mathcal{P}_2(\|w_0f_0\|_{\infty})
\int_{0}^{t } \left| \nu(f) (s, X(s;t,x,v), V(s;t,x,v ))  - \nu(f)(s, X(s;t,\bx,\bv), V(s;t,\bx,\bv )) \right| ds. \label{basic f-f4'} 
\end{align}
By Lemma \ref{lem:tra_0.5}-(4), we obtain
\begin{align*}
     \eqref{basic f-f1'}
     &\lesssim   \langle v \rangle |(x,v)-(\bx,\bv)|^{\frac{1}{2}}  \left(\mathbf{A}_{\frac{1}{2}}(f_0)+\|w_0  f_0\|_{\infty} \right).
\end{align*}
Applying Lemma \ref{lem:gam_x_0.5} to \eqref{basic f-f2'} and \eqref{basic f-f4'}, we obtain
\begin{align*}
    &\eqref{basic f-f2'}+\eqref{basic f-f4'} \\
    &\lesssim_{\delta,\epsilon, \vartheta_0} |(x,v)-(\bx,\bv)|^{\frac{1}{2}} \langle v \rangle  e^{4\varpi \langle v \rangle^{2} t}  \left( \frac{1}{|v|^{\delta}}\mathbf{1}_{\{d(x,\partial\O)\leq\epsilon\}}+\frac{1}{|\bv|^{\delta}}\mathbf{1}_{\{d(\bx,\partial\O)\leq\epsilon\}}+1\right)\\
        &\quad \times  \mathcal{P}_3(\|w_0f_0\|_{\infty} )
        \bigg(\sup_{0 \leq t \leq T}\mathfrak{X}(t,\varpi;\epsilon)+\sup_{0 \leq t \leq T}\mathfrak{V}(t,\varpi;\epsilon)+\mathbf{A}_{\frac{1}{2}}(f_0)+1\bigg)
\end{align*}
for $0 < \delta, \,\epsilon <1$. By Proposition \ref{pro:iter_pro}, we find a constant $0<T \ll 1$ such that
\begin{align*}
    &\eqref{basic f-f2'}+\eqref{basic f-f4'} \\
    &\lesssim_{\delta,\epsilon,\vartheta_0} 
    |(x,v)-(\bx,\bv)|^{\frac{1}{2}} \langle v \rangle  e^{4\varpi \langle v \rangle^{2} t} \\
    &\quad \times\left( \frac{1}{|v|^{\delta}}\mathbf{1}_{\{d(x,\partial\O)\leq\epsilon\}}+\frac{1}{|\bv|^{\delta}}\mathbf{1}_{\{d(\bx,\partial\O)\leq\epsilon\}}+1\right)  \mathcal{P}_3(\|w_0f_0\|_{\infty} )\left( \mathbf{A}_{\frac{1}{2}}(f_0)+1 \right).
\end{align*}
for $0<t<T$. 

From the upper bounds of \eqref{basic f-f1'}, \eqref{basic f-f2'}, and \eqref{basic f-f4'}, we obtain
\begin{align*}
    &|f(t,x,v)-f(t,\bar{x}, \bar{v})|    \\
     &\lesssim_{\delta,\epsilon,\vartheta_0} 
    |(x,v)-(\bx,\bv)|^{\frac{1}{2}} \langle v \rangle  e^{4\varpi \langle v \rangle^{2} t} \\
    &\quad \times\left( \frac{1}{|v|^{\delta}}\mathbf{1}_{\{d(x,\partial\O)\leq\epsilon\}}+\frac{1}{|\bv|^{\delta}}\mathbf{1}_{\{d(\bx,\partial\O)\leq\epsilon\}}+1\right)  \mathcal{P}_3(\|w_0f_0\|_{\infty} )\left( \mathbf{A}_{\frac{1}{2}}(f_0)+1 \right).
\end{align*}
Finally, we multiply both sides by $\langle v \rangle^{-1} e^{-4\varpi \langle v \rangle^{2} t}W$. Then, we obtain
\begin{align*}
    W^{-1}e^{-4\varpi \langle v \rangle^{2} t}\langle v \rangle^{-1}\frac{|f(t,x,v)- (t,\bar{x}, \bar{v})|}{ |(x,v)-(\bx,\bv)|^{\frac{1}{2}}}
    \lesssim_{\delta, \epsilon, \vartheta_0} \mathcal{P}_3(\|w_0f_0\|_{\infty}) \left( \mathbf{A}_{\frac{1}{2}}(f_0)+1 \right),
\end{align*}
where 
\begin{align*}
    W\left((x,v), (\bx,\bv);\epsilon,\delta\right) &:= 
    \frac{1}{|v|^{\delta}}\mathbf{1}_{\{d(x,\partial\O) \leq \epsilon\}}+ \frac{1}{|\bv|^{\delta}}\mathbf{1}_{\{d(\bx,\partial\O) \leq \epsilon\}}+1.
\end{align*}
\end{proof}

\noindent{\bf Acknowledgment} 
D. Lee thanks Professor Chanwoo Kim for helpful discussions. The authors also thank the anonymous referee for helpful comments and suggestions that improved the presentation of this paper. \newline

\noindent{\bf Declarations}\\
\noindent{\bf Funding:}
G. An is supported by grant No. RS-2024-00406821. D. Lee is supported by the National Research Foundation of Korea (NRF) grant funded by the Korea government (MSIT) (No. RS-2023-00212304 and No. RS-2023-00219980).  \newline

\noindent{\bf Author Contributions Statement:} We would like to state that Gayoung An and Donghyun Lee contributed equally to this work.
\newline

\noindent{\bf Data availability:} No data was used for the research described in the article.
\newline

\noindent{\bf Conflict of interest:} The authors declare that they have no conflict of interest.

\bibliographystyle{abbrv}
\bibliography{reference.bib}
  
\end{document}